\documentclass[11pt]{amsart}
\usepackage{amsfonts, amsmath, enumerate, amssymb}

\usepackage{hyperref}

\newtheorem{theorem}{Theorem}[section]

\newtheorem{lemma}[theorem]{Lemma}
\newtheorem{corollary}[theorem]{Corollary}
\newtheorem{proposition}[theorem]{Proposition}
\theoremstyle{definition}
\newtheorem{definition}[theorem]{Definition}
\newtheorem{definitions}[theorem]{Definitions}
\newtheorem{example}[theorem]{Example}
\newtheorem{examples}[theorem]{Examples}

\theoremstyle{remark}
\newtheorem{remark}[theorem]{Remark}
\newtheorem{remarks}[theorem]{Remarks}

\numberwithin{equation}{section}
\usepackage{amscd}
\usepackage{xy}
\usepackage{amsmath, amssymb}

\numberwithin{equation}{subsection}
\newcommand{\be}%
  {\protect\setcounter{equation}{\value{subsubsection}}}
  \newcommand{\ee}%
   {\protect\setcounter{subsubsection}{\value{equation}}}

  {\protect\setcounter{subsubsection}{\value{equation}}}

\def \AbPsh{\rm {AbPsh}}

\def \B{\mathcal B}

\def \C{\mathcal C}

\def \Cl{\mathbb C}
\def \colim{\underset \rightarrow  {\hbox {lim}}}

\def \colimn{\underset {n \rightarrow \infty}  {\hbox {lim}}}
\def \colimk{\underset {k \rightarrow \infty}  {\hbox {lim}}}

\def \colimalpha{\underset {\alpha}  {\hbox {colim}}}

\def \colimK.{\underset {\underset K^.  \rightarrow}  {\hbox {lim}}}

\def \colimU.{\underset {\underset U_.  \rightarrow}  {\hbox {lim}}}

\def \D{\mathcal D}

\def \EG1{E{(G \times {\mathbb C}^*)}{\underset {G\times {\mathbb C}^*} 
\times}}

\def \EZ(s)1{E{(Z(s) \times {\mathbb C}^*)}{\underset {(Z(s)\times {\mathbb
C}^*)}  \times}}

\newcommand{\eps}{ \, {\boldsymbol\varepsilon} \,}

\def \EM(u){EM(u){\underset {M(u)}  \times}}
\def \EM(us){EM(u,s){\underset {M(u, s)}  \times}}

\def \F{\mathcal F}

\def \group{{\mathbf G}}

\def\holimD{\mathop{\textrm{holim}}\limits_{\Delta }}
\def\holim{\mathop{\textrm{holim}}\limits}
\def \hocolimD{\underset \Delta  {\hbox {hocolim}}}

\def \H{\rm {\bf H}}

\def \Hom{\underline {Hom}}
\def \Hom{{\mathcal H}om}

\def \invlim1{\underset {\infty \leftarrow q}  {\hbox {lim}}^1}

\def \oI{\rm I}
\def \bI{\mathbf I}

\def \oJ{\rm J}

\def \L3{\Lambda \times \Lambda \times \Lambda}
\def \L2{\Lambda \times \Lambda}

\def \lim{\underset \leftarrow  {\hbox {lim}}}

\def \longright2arrow{{\overset \longrightarrow  {\overset {} 
\longrightarrow}}}

\def \L{L\times \Cl ^*}

\def \M{\mathcal M}

\def \O{{\mathcal O}}

\def \PSh{\rm {PSh}}
\def \Spt{\rm {Spt}}

\def \ra{\rightarrow}

\def \RG^{R(G)^{\hat {}}\ }

\def \res{respectively}

\def \RHom{{{\mathcal R}{\mathcal H}om}}
\def \R{{\mathcal R}}

\def \S{\mathcal S}
\def \bS{\mathbf S}

\def \Sm{{\rm {Sm}}}
\def \Sch{{\rm {Sch}}}

\def \Sh{\rm {Sh}}

\def\Spt{\rm {\bf Spt}}

\def \topGcoh*{^{top, *} _{G}}
\def \topGho*{ _{top,*} ^{G}}

\def \T{{\mathbf T}}

\def \W{\mathbf W}

\def \X{\mathcal X}

\def \Z(s){Z(s) \times {\mathbb C}^*}
\def \Z{\mathbb Z}

\begin{document}

\title{Equivariant motivic homotopy theory}
\author{Gunnar Carlsson}
\address{Department of Mathematics, Stanford University, Building 380, Stanford,
California 94305}
\email{gunnar@math.stanford.edu}
\thanks{  }  
\author{Roy Joshua}
\address{Department of Mathematics, Ohio State University, Columbus, Ohio,
43210, USA.}
\email{joshua@math.ohio-state.edu}
\thanks{The first author was supported by the NSF. The second author was supported by the IHES and  grants
from the NSA and NSF at different stages of this work. AMS subject classification: 14F42.}
\begin{abstract} In this paper, we develop the theory of equivariant motivic
homotopy theory, both unstable and stable. While our original interest was in the case of profinite group actions on smooth schemes,
we discuss our results in as broad a  setting as possible so as to be applicable in a variety of  contexts, for example to the case
of smooth group scheme actions on schemes that are not necessarily smooth. We also discuss how ${\mathbb A}^1$-localization
behaves  with respect to  mod-$\ell$-completion, where $\ell$ is a fixed prime. 

\end{abstract}
\maketitle

\centerline{\bf Table of contents}
\vskip .3cm 
1. Introduction
\vskip .3cm
2. The basic framework of equivariant motivic homotopy theory: the unstable theory
\vskip .3cm 
3. The basic framework of equivariant motivic homotopy theory: the stable theory
\vskip .3cm 
4. Effect of ${\mathbb A}^1$-localization  on mod-$\ell$ completions 
\vskip .3cm
5. References.
\vskip .3cm 
\markboth{Gunnar Carlsson and Roy Joshua}{{Equivariant motivic homotopy theory}}
\input xypic
\vfill \eject
\section{\bf Introduction}
\label{intro}
\vskip .3cm 
The original purpose of this paper was to set up a suitable framework for a series of papers
exploring  the conjectures due to the
first author on equivariant algebraic K-theory, equivariant with respect to the
action of a Galois group, see \cite{Carl1}. We began this project a few years ago. It was clear that
a framework suitable for such applications would be that of equivariant motivic
stable homotopy, equivariant with respect to the action of a profinite group, which typically will
be the Galois group of a field extension. In the intervening years
other applications also emerged and it became clear that one needs to work over a general Noetherian
base scheme $S$ allowing actions of smooth $S$-group schemes. In fact certain applications dictate
that one needs to not only work in this context, but also perform all the operations {\it over $S$}, or
relative to $S$, i.e. the corresponding homotopy theory corresponds to what is called {\it ex-homotopy 
theory}. We hope that the present paper addresses several of these concerns.
\vskip .3cm
The study of algebraic cycles and motives using techniques from algebraic topology
has been now around for about 15 years, originating with the work of Voevodsky on
the Milnor conjecture and the paper of Morel and Voevodsky on ${\mathbb A}^1$-homotopy theory:
see \cite{Voev} and \cite{MV}. The paper \cite{MV} only dealt with the unstable theory,
and the stable theory was subsequently developed by several authors in somewhat different contexts,
for example, \cite{Hov-3} and \cite{Dund2}. 
An equivariant version of the unstable theory already
appears in \cite{MV} and also in \cite{Guill}, but both are for the action of a finite group. 
\vskip .3cm
Equivariant stable homotopy theory, even in the classical setting, needs a fair amount
of technical machinery. As a result, and possibly because concrete applications of the equivariant
stable homotopy framework in the motivic setting have been lacking till now, equivariant stable motivic 
homotopy theory has not been worked out or even considered in any detail up until now. The present paper
hopes to change all this, by working out equivariant stable motivic homotopy theory in some detail and in
as broad a setting as possible so that it readily applies to any of the sites that one encounters in the
motivic context.
\vskip .3cm
One of the problems in handling the equivariant situation in stable homotopy is that the spectra  are no longer
indexed by the non-negative integers, but by representations of the group. One  way of circumventing this problem
that is commonly adopted is to use only multiples of the regular representation. This works when the
representations are in characteristic $0$, because then the regular representation breaks up as a sum of
all irreducible representations. However, when considering the motivic situation, unless one wants to restrict
to schemes in characteristic $0$, the above decomposition of the regular representation is no longer true in general. 
As a result there seem to be
serious technical issues in adapting the setting of symmetric spectra to  study equivariant stable motivic
homotopy theory. On the other hand the technique of enriched functors as worked out in \cite{Dund1} and \cite{Dund2}
only requires as indexing objects, certain finitely presented objects in the unstable category. As a result 
this approach  requires the least amount of extra work to cover the equivariant setting, and we have chosen to adopt
this as the appropriate framework for our work.
\vskip .3cm
In classical algebraic topology itself, it is often necessary to work in a relative setting over an arbitrary space as a base: for example,
this is needed in the context of the classical Becker-Gottlieb transfer. This is handled by working there in the context of what
are called {\it ex-spaces}: see \cite{James} , \cite{BG76} and also \cite{IJ}. The corresponding framework for doing motivic stable homotopy theory means
not only that one needs to work over a general base scheme, but also that several natural operations need to be adapted to a relative setting. Most
of our results are worked out in this context so that they readily apply to various current and future applications of motivic homotopy theory.
\vskip .3cm
The following is an outline of the paper. Though our primary focus is on obtaining a full-fledged version of stable motivic homotopy theory
in the equivariant setting for smooth schemes, there are many advantages to working out the basic theory in much more generality. For example,
this way our basic results would apply readily to various different Grothendieck topologies that one encounters in the motivic
context, for example it could be the Nisnevich topology on smooth schemes, the cdh topology on schemes (which may not be smooth),
the equivariant Nisnevich topology on smooth schemes with group action, the \'etale topology on schemes, the isovariant \'etale topology on schemes with group actions etc.
 It could also apply equally well to other Grothendieck topologies, for example the fppf or syntomic sites which are important
in algebraic geometry. 
\vskip .3cm
In view of these, we try to work on general Grothendieck topologies as much as possible and specialize to particular topologies only when
it becomes absolutely necessary. Nearly half the paper is devoted to working out  the unstable theory in full detail and in as broad a 
framework as possible. 
\vskip .3cm
We will fix a {\it small category} $\bS$ which is closed under all finite sums and finite limits, with a terminal object $\B$. $\B$ will be called {\it the base object}
of $\bS$. {\it Such a category, often provided with a Grothendieck topology, will be our input.} In particular, we do not concern ourselves
with questions or issues that deal with intrinsic properties of these categories or sites, since they have already been worked out in detail elsewhere. 
(For example, sites where every object also has an action by the group are discussed in detail in somewhat  different contexts in 
\cite{T2}, \cite{J03}, \cite{Serp}, \cite{KO}, and \cite{Her}.)
Later on we will specialize to sites that are of particular importance in the motivic contexts.
\vskip .3cm
If $Y$ is any object of $\bS$, 
$\bS/Y$ will denote the comma category of  $Y$-objects in $\bS$ (i.e. objects $X$ in $\bS$ together with a chosen map $X \ra Y$ in $\bS$.) 
We will assume that $S \eps \bS$ is a fixed object and is also provided
with a fixed section $s: \B \ra S$ to the structure map $p_S:S \ra \B$ of $S$.  
\vskip .3cm
A {\it pointed $S$-object} $X$ will be an $S$-object provided with a section $s_X:S \ra X$ to the structure map  $p_X$. Given any object
 $X$ over $S$, one defines the associated pointed object by 
\[X_{+_{S}} = X\sqcup S.\]
Pointed $\B$-objects will have the corresponding meaning. Given an object $X$ (over $\B$), one lets $X_+ = X \sqcup \B$ denote the 
corresponding pointed object (over $\B$).
 Throughout the paper $\group$ will denote one of the following: (i) either a discrete group, (ii) a profinite group
or (iii)  a presheaf of groups. Often the framework for (ii) will be where the base object 
$\B$ is a field and the profinite group will be the absolute Galois group  of the field with respect to a chosen algebraic 
closure. In case (iii), when $\bS$ denotes a category of schemes over the base-scheme $\B$,
$\group$ will denote a smooth (affine) group scheme viewed as a presheaf of groups. (Here  
we consider actions of $\group(U)$ on $\Gamma(U, P)$
for $P \eps \PSh^{\group}$. The latter category is defined below.)
\vskip .3cm
Then ${\rm {PSh}}^{\group}({\rm {\bS}})$ (abbreviated to $\PSh^{\group}$) will denote the category of pointed simplicial presheaves on the 
category $\bS$ 
 provided with an action by $\group$.  Next we will fix a 
family, $\W$, of subgroups of
$\group$, so that it has (at least) the following properties 
\be  \begin{enumerate}[ \rm (i)]
        \item it is an inverse  system ordered by inclusion,
        \item if $H \eps \W$, $H_{\group}$ =the core of $H$, i.e. the largest subgroup of $H$ that is normal in $\group$ belongs to $\W$ and
         \item if $H \eps \W$ and $H' \supseteq H$ is a subgroup of $\group$, then $H' \eps \W$. 
\end{enumerate} \ee
In the case $\group$ is a finite
group, $\W$ will denote all subgroups of $\group$ and when $\group$ is profinite, it will denote all subgroups of finite index in 
$\group$. When $\group$ is a presheaf of groups or a smooth affine group-scheme, we will leave $\W$ unspecified, for now. Clearly the family of all closed subgroup-schemes
of a given smooth group-scheme satisfies all of the above properties, so that we may use this as a default choice of $\W$ when nothing 
else is specified.
\vskip .3cm
${\rm {PSh}}^{\group, c}({\rm {\bS}})$  (abbreviated to $\PSh^{\group,c}$) will denote
the full subcategory of such pointed simplicial sheaves where the action of $\group$ is  {\it continuous with respect to $\W$}, i.e.
if $P \eps  \PSh ({\rm {\bS}, \group})$  and $s \eps \Gamma (U, P)$, then the stabilizer of $s$ belongs to $\W$. 
\vskip .3cm
In addition to this we also consider the following alternate {\it relative case}, which has not been so far looked at even in the
non-equivariant case. Assume that $\group$ is a presheaf of groups (or when the underlying category of $\bS$ 
is a category of schemes, $\group$ is a smooth group scheme defined over the scheme $S$.) Then
$\PSh^{\group}/S$ will denote the subcategory of $\PSh^{\group}$ consisting of those pointed simplicial presheaves $P$ which come equipped with  
maps $p_P: P \ra S$ and  $s_P: S \ra P$ in $\PSh^{\group}$ so that $s_P$ is a section to $p_P$: this extra structure will be referred to as {\it the data of pointing by $S$}. Maps between two such simplicial presheaves $f:P \ra Q$ will
be maps of simplicial presheaves that are compatible with extra structure, i.e. those maps $f$ of simplicial presheaves that fit in
a commutative diagram:
\be \begin{equation}
     \label{rel.case}
\xymatrix{&{S} \ar[dl]_{s_P} \ar[dr]^{s_Q} \\
            {P} \ar[rr]_f \ar[dr]_{p_P}&& {Q} \ar[dl]^{p_Q}\\
            & {S}}
\end{equation} \ee
Next if $f: F \ra G$ is a monomorphism of objects in $\PSh^{\group}/S$, we define $G/F$ to be the pushout $P$:
\be \begin{equation}
     \label{S.quotient}
\xymatrix{{F} \ar@<1ex>[r]^f \ar@<1ex>[d]^{p_F} & {G} \ar@<1ex>[d]\\
           {S} \ar@<1ex>[r] &{P}}
    \end{equation} \ee
Since all the other three vertices have compatible maps to $S$, one obtains an induced map $P \ra S$, so that the composition 
$S \ra P \ra S$ is the identity. (In the rare occasions, where there is a chance for confusion, we will
denote the above pushout by $G/^SF$.) The following are the two main results in the unstable setting.
\begin{theorem} 
\label{main.thm.1}Following the terminology as in Definition ~\ref{equiv.model.struct},
letting the {\it generating cofibrations} be of the form 
\vskip .3cm 
$\oI_{\group} = \{ (\group /H )_+ \wedge i \eps \oI, H \eps \W \}$, 
\vskip .3cm 
 the {\it generating trivial cofibrations} be of the form 
\vskip .3cm
$\oJ_{\group} = \{(\group/H )_+ \wedge  j \eps \oJ, 
H \subseteq \group, \quad H \eps \W  \}$ \, and
\vskip .3cm \noindent
the  {\it weak-equivalences} ({\it fibrations}) be maps  $f:P' \ra P$ in
 so that $f^H:{P'}^H \ra P^H$ is a weak-equivalence
(fibration, \res) defines a cofibrantly generated
simplicial model structure on
$\PSh^{\group}$, $\PSh^{\group,c}$,  $\PSh^{\group}/S$ and on $\PSh^{\group,c}/S$ that is proper. 
In addition, the 
smash product of pointed simplicial presheaves defined in ~\eqref{smash.0} and ~\eqref{smash.over.S} makes these symmetric monoidal model categories. 
\vskip .3cm
When $\group$ denotes a finite group or profinite group, the categories $\PSh^{ \group,c }$ and $\PSh/S^{\group,c}$ are locally presentable, i.e. 
there exists a set of
objects, so that every simplicial presheaf $P$ in the above categories is a filtered colimit
of these objects. In particular, they are also combinatorial and tractable model categories.
\end{theorem}
\vskip .3cm \vskip .3cm
 One also considers  the following alternate framework for equivariant unstable motivic homotopy theory.
First one considers the orbit category
${\mathcal O}_{\group}= \{\group/ H \mid  H \eps \W \}$. A morphism $\group/H \ra
\group/K$ corresponds to a $\gamma \eps \group$, so that $\gamma.H \gamma ^{-1}
\subseteq K$. One may next consider the category $\PSh^{{\mathcal
O}_{\group}^o}$ ($\PSh/S^{{\mathcal
O}_{\group}^o}$) of ${\mathcal O}_{\group}^o$ -diagrams with values in ${\rm
{PSh}}$ (${\rm {PSh}}/S$, \res). 
Then the two categories 
$\PSh^{ \group,c}$ and $\PSh^{{\mathcal O}_{\group}^o}$ ($\PSh/S^{ \group,c}$ and $\PSh/S^{{\mathcal O}_{\group}^o}$)
are related by the functors:
\vskip .3cm
$\Phi:\PSh^{\group,c} \ra \PSh^{{\mathcal O}_{\group}^o} (\Phi:\PSh/S^{\group,c} \ra \PSh/S^{{\mathcal O}_{\group}^o}), 
\quad F \mapsto \Phi(F) = \{\Phi(F)(\group/H)=F^H\} \mbox{ and} $
\vskip .3cm
$\Theta: \PSh^{{\mathcal O}_{\group}^o} \ra \PSh^{\group,c} (\Theta: \PSh/S^{{\mathcal O}_{\group}^o} \ra \PSh/S^{\group,c}), \quad M \mapsto \Theta (M) =
{\underset {  \{H | {H \eps \W} \}} \colim } M(\group/H)$.
\vskip .3cm \noindent 
A key result then is the following.
\begin{theorem}
 \label{main.thm.2}
The functors $\Phi$ and $\Theta$ are Quillen-equivalences.
\end{theorem}
\vskip .3cm
Observe that in both the above results we need to consider $\PSh^{{\mathcal O}_{\group}^o}$ ($\PSh/S^{{\mathcal O}_{\group}^o}$) with the 
projective model structure and that we are able to prove that $\PSh^{\group, c}$ and $\PSh/S^{\group,c}$ are locally presentable
only when $\group$ denotes either a finite or profinite group. On the other hand it is possible to prove readily that both the categories
$\PSh^{{\mathcal O}_{\group}^o}$ and $\PSh/S^{{\mathcal O}_{\group}^o}$, when provided with the {\it object-wise model structure}
are locally presentable and therefore combinatorial (and also tractable) model categories in general, i.e. when $\group$ may denote
any one of the three allowed possibilities. Moreover, we observe that both the above categories, when provided with the object-wise model
structures are what are called {\it excellent monoidal model categories} in \cite[Definition A.3.2.16]{Lur}.
\vskip .3cm
Section 3 is entirely devoted to the stable theory.
In view of the above observations, we restrict to  the diagram categories
 $\PSh^{{\mathcal O}_{\group}^o}$ ($\PSh/S^{{\mathcal O}_{\group}^o}$) provided with their object-wise model structures and develop a stable model structure
on them.
We will let $\C$ denote one of the above  categories  provided with the structure
 of object-wise model structure of cofibrantly generated model categories discussed below. (These start with the object-wise model structure 
on $\PSh$ and $\PSh/S$.) The above categories are symmetric monoidal with the smash-product, $\wedge$,  of simplicial presheaves
defined as in ~\eqref{smash.0} or as in ~\eqref{smash.over.S} as the monoidal product. The unit for the monoidal product will be denoted $S^0$. 
 In both cases we perform Bousfield localizations 
either by inverting maps of the form $U \times I \ra U$ when the cite $\bS$ is provided with an interval $I$ (in the sense of \cite[2, (2.3)]{MV})
and certain maps associated to distinguished squares if the topology is
defined by a cd-structure (or maps 
of the form $U_{\bullet} \ra U$ where $U \eps {{\rm Sm}/S}$ and $U_{\bullet} \ra U$ is a 
hypercovering in the given topology,  for a general site.)
\vskip .3cm
Let $\C'$ denote a $\C$-enriched full-subcategory of $\C$ consisting of {\it objects closed 
under the
monoidal product $\wedge$, all of which are assumed to be cofibrant  and containing the unit $S^0$}. 
Let $\C_0'$ denote a $\C$-enriched  sub-category of $\C'$, which may or may not be full, but closed under the 
monoidal product $\wedge$ and containing the unit $S^0$. Then {\it the basic model of equivariant motivic stable homotopy category} will be
the category $[\C_0', \C]$. This is the category whose objects are $\C$-{\it enriched covariant functors
 from $ \C_0'$ to $\C$}: see \cite[2.2]{Dund1}. There are several possible choices for the category $\C_0'$, which are discussed
in detail in Example ~\ref{main.eg}.
\vskip .3cm
 We let $ {\rm Sph}(\C_0')$ denote the $\C$-category defined by taking the 
objects to be the same as the objects of $\C_0'$ and where $\Hom_{ { Sph}(\C_0')}(T_U, T_V) = T_W$ if 
$T_V= T_W \wedge T_U$ and $*$ otherwise. Since $T_W $ is a sub-object of $\Hom_{\C}(T_U, T_V)$, it 
follows that ${\rm Sph}(\C_0')$ is a sub-category of $\C_0'$. Now an enriched functor in 
$[{ {\rm Sph}}(\C_0'), \C]$
is simply given by a collection $\{X(T_V)|T_V \eps {\rm Sph}(\C_0')\}$ provided with a compatible collection of 
maps $T_W \wedge X(T_V) \ra X(T_W \wedge T_V)$. We let $Presp (\C) =[{{\rm Sph}}(\C_0'), \C]$ and call this {\it the
category of pre-spectra with values in $\C$}.
\vskip .3cm
We let $Spectra (\C) $(or rather $Spectra (\C_0', \C)$)  be the category of enriched functors $[\C_0', \C]$.
It is important to observe that the objects of $\C_0'$ and ${{\rm Sph}}(\C_0')$ are the same, but
that the latter is often a strictly smaller sub-category than the former. In particular it will often be
not symmetric monoidal. For example, in the non-equivariant case, i.e. when the group $\group$ is trivial, and
$\C$ is the category of pointed simplicial presheaves on a site,
the subcategory ${\rm Sph}(\C_0')$ will correspond to the usual spheres whereas $\C_0'$ will correspond to
the bigger subcategory of symmetric spheres.
\vskip .3cm
We provide several model structures on $Spectra(\C)$ (and also on $Presp(\C)$: (i) two unstable model structures we call the projective and injective model
  structures and (ii) two stable model structures corresponding to each one of the above unstable model structures.
A map $f: {\mathcal X}' \ra {\mathcal X}$ in $Spectra (\C)$ ($Presp(\C)$) is a 
{\it level equivalence} ({\it level fibration}, {\it level trivial fibration}, {\it level cofibration},
{\it level trivial cofibration) if each $Ev_{T_V}(f)$ is a weak-equivalence (fibration, trivial fibration,
cofibration, trivial cofibration, \res) in $\C$.  Such a map $f$ is a {\it projective cofibration}
if it has the left-lifting property with respect to every level trivial fibration}. It is an {\it injective fibration} if it has the
right lifting property with respect to every trivial level cofibration. The generating projective cofibrations 
(generating projective trivial cofibrations) will be denoted $\oI_{Sp}$ ($\oJ_{Sp}$) for the category $Spectra (\C)$
and the corresponding categories for $Presp(\C)$ will be denoted $\oI_{Presp}$ ($\oJ_{Presp}$, \res). Then we obtain:
\begin{theorem}
 \label{main.thm.3}
The projective cofibrations, the level fibrations and level equivalences define a cofibrantly 
generated model category structure on $Spectra (\C)$  with the generating cofibrations (generating trivial
cofibrations) being $\oI_{Sp}$ ($\oJ_{Sp}$, \res). This model structure (called {\it the 
projective model structure} on
 $Spectra (\C)$) is left-proper (right proper) if
the corresponding model structure on $\C$ is left proper (right proper, \res). It is cellular if
 the corresponding model structure on $\C$ is cellular. The corresponding statements hold for $Presp(\C)$
with $\oI_{Presp}$ ($\oJ_{Presp}$, \res) in the place of $\oI_{Sp}$ ($\oJ_{Sp}$, \res).
\vskip .3cm
The level cofibrations, the injective fibrations and the level weak-equivalences  define a cofibrantly generated
model category structure on $Spectra(\C)$ which will form a  combinatorial (in fact tractable) model category.
\end{theorem}
\vskip .3cm \noindent 
See Corollary ~\ref{level.model.1} for further details. Next we localize the above model structure suitably to define
the stable model structure. A spectrum ${ X} \eps Spectra(\C)$ is an  {\it $\Omega$-spectrum} if if it is level-fibrant and 
each of the natural maps
${ X}(T_V) \ra \Hom_{\C}(T_{W}, { X}(T_V \wedge T_W))$, $T_V, T_W \eps \C_0'$ is a weak-equivalence in $\C$. 
\begin{theorem} (See Proposition ~\ref{stable.model.1} for more details.)
\label{main.thm.4}
(i) The corresponding stable model structure on $Spectra(\C)$ is cofibrantly 
generated, left proper and cellular when $\C$ is assumed to have these properties.
\vskip .3cm 
(ii) The fibrant objects in the stable model structure on $Spectra(\C)$ are the
 $\Omega$-spectra defined above. 
\vskip .3cm 
(iii) $Spectra(\C)$ with the above stable model structures form symmetric monoidal model categories in that the pushout-product
axiom (as in \cite[Definition  3.1]{SSch}) is satisfied with the monoidal structure being the smash-product defined in ~\eqref{smash.products.def}. The 
unit of the monoidal structure is cofibrant in the stable injective model structure.
\end{theorem}
\vskip .3cm
The above results then enable us to define the various equivariant stable motivic homotopy categories that we consider:
these are discussed in Definition ~\ref{mot.spectra}. One may want to note that the terminology above is
clearly non-standard: what we called pre-spectra are often called spectra.
\vskip .3cm
In the last section we  study how ${\mathbb A}^1$-localization behaves with respect to 
 $Z/\ell$-completion, for a fixed prime $\ell$.
Since comparisons are often made with the \'etale homotopy types, completions at a prime also play a major role in our work. 
Unfortunately it is not clear that the Bousfield-Kan
completion as such may not commute with ${\mathbb A}^1$-localization, at least in the general of setting of equivariant spectra considered in this paper. 
This makes it necessary for us to
redefine a Bousfield-Kan completion at a prime $\ell$ that behaves well in the ${\mathbb A}^1$-local setting: this is
a combination of the usual Bousfield-Kan completion and ${\mathbb A}^1$-localization.  Since the resulting
completion is in general different from the traditional Bousfield-Kan completion, we denote our completion functor
by ${\widetilde {Z/\ell}}_{\infty}$.
The main properties of
the resulting completion functor are discussed in ~\ref{completion.mult}.
\vskip .3cm
We have also decided to include brief discussions of some alternate frameworks. One of these, in the unstable setting, is a {\it coarse
model structure} which is considered rather briefly in ~\ref{coarse.model}. Another is a construction of equivariant spectra by applying
the construction of symmetric spectra to any of the equivariant unstable model structures that we discuss. This is  discussed only briefly in Examples
~\ref{main.eg}(iv).
\vskip .3cm
{\bf Acknowledgements}. The authors thank Paul Ostv\ae{}r and Amalendu Krishna for sharing their work \cite{KO} with us and for helpful
discussions. We also thank Jeremiah Heller for helpful discussions. 
\section{\bf The basic framework of  equivariant  motivic homotopy
theory: the unstable theory}
\vskip .3cm
In this section we define a general framework that will specialize readily to equivariant  unstable motivic
homotopy theory, but will be general enough to handle a multitude of other environments. For example, using the
cdh topology and equivariant resolution of singularities, it will be able to handle singular schemes over fields
of characteristic $0$. Moreover, it will be broad enough to handle any of the Grothendieck topologies that arise
commonly in algebraic geometry, like the fppf (flat) topology, the \'etale topology etc.
\vskip .3cm
We recall the framework discussed in the introduction. We will fix a {\it small category} $\bS$ which is closed under all finite sums and finite limits, with a 
terminal object $\B$. $\B$ will be called {\it the base object}
of $\bS$. If $Y$ is any object of $\bS$, 
$\bS/Y$ will denote the comma category of  $Y$-objects in $\bS$ (i.e. objects $X$ in $\bS$ together with a chosen map $X \ra Y$ in $\bS$.) 
We will assume that $S \eps \bS$ is a fixed object and is also provided
with a fixed section $s: \B \ra S$ to the structure map $p_S:S \ra \B$ of $S$.  
\vskip .3cm
A {\it pointed $S$-object} $X$ will be an $S$-object provided with a section $s_X:S \ra X$ to the structure map  $p_X$. Given any object
 $X$ over $S$, one defines the associated pointed object by 
\[X_{+_{S}} = X\sqcup S.\]
Pointed $\B$-objects will have the corresponding meaning. Given an object $X$ (over $\B$), one lets $X_+ = X \sqcup \B$ denote the 
corresponding pointed object (over $\B$).
 Throughout the paper $\group$ will denote one of the following: (i) either a discrete group, (ii) a profinite group
or (iii)  a presheaf of groups. Often the framework for (ii) will be where the base object 
$\B$ is a field and the profinite group will be the absolute Galois group  of the field with respect to a chosen algebraic 
closure. In case (iii), when $\bS$ denotes  a category of schemes over the base-scheme $\B$,
$\group$ will denote a smooth (affine) group scheme viewed as a presheaf of groups. (Here  
we consider actions of $\group(U)$ on $\Gamma(U, P)$
for $P \eps \PSh^{\group}$, $U \eps \bS$.)
\vskip .3cm
$ \bS_{?}$    $({\rm {\bS/S}})_{?}$) will denote the  category  $\bS$ (${\rm {\bS/S}}$, \res) provided with
a Grothendieck 
topology denoted  $?$. We will assume all these sites {\it have enough points}. In the motivic context,  $\bS$ will denote a category schemes and usually provided with one of the big Zariski, Nisnevich, cdh  or \'etale
topologies, but also possibly other topologies. In this case, observe that the morphisms (coverings) in the site $\bS/S_{?}$ will be those morphisms
in the category $\bS/S$ which also belong to the morphisms (coverings) in the topology $?$ on $\bS$ ($\bS/S$, \res), but possibly satisfying further conditions. 

\subsection{\bf Model structures}
Though the main objects considered in this paper are simplicial presheaves, we will
often need to provide them with different model structures (in the sense of
\cite{Qu}). The same holds for the
stable case, where different variants of spectra will be considered.
 Therefore,  the basic framework will be that of  model categories,
often simplicial (symmetric)
 monoidal model categories, i.e. simplicial model categories that have the
structure of monoidal model categories as in \cite{Hov-2}. We will further
assume that these categories have other {\it nice properties}, for example, are
closed under all small limits and colimits, are proper and cellular. (See
\cite{Hov-1} and \cite{Hirsch} for basic material on model categories.)
\vskip .3cm
\subsubsection{\bf The category of simplicial presheaves}
\label{simp.pre}
We fix an object $S \eps \bS$ and consider the category $\PSh = {\rm
{PSh}}( {\bS/S})$  ($\PSh_*= {\rm
{PSh}}_*( {\bS/S})$ of simplicial presheaves (pointed simplicial presheaves, \res) on ${{\bS/S}}$. 
(Recall a pointed simplicial presheaf here has the usual meaning: i.e. $P$ is a pointed simplicial presheaf if there is a unique map $* \ra P$, 
where $*$ denotes the set with one point, which identifies with the terminal object in the category of (unpointed) simplicial presheaves.)
We will need to consider different
Grothendieck topologies
on ${\rm {\bS/S}}$: this is one reason that we prefer to work for the most part with pointed simplicial
presheaves rather than with pointed simplicial sheaves. We will consider pointed
simplicial sheaves only when it becomes absolutely essential to do so. If $?$ denotes a Grothendieck topology,
 (usually either $\acute{e}t$, $Nis$, $cdh$ or
$Zar$ (i.e. \'etale, Nisnevich, cdh  or Zariski)), and $S \eps {\rm {\bS}}$, ${\rm {Sh}}({\rm {\bS/S}}_{?})$ will denote the
category of pointed simplicial sheaves on the corresponding big site.
\vskip .3cm
One may define the structure of a {\it closed simplicial model category} on $\PSh$ and on $\PSh_*$ as follows. To any simplicial set (pointed simplicial set)
one may associate the corresponding constant simplicial presheaf (constant pointed simplicial presheaf). 
Given a simplicial set $K$ and a $P \eps \PSh$ ($P \eps \PSh_*$), we let $K \times P$ ($K \rtimes P$) be the simplicial presheaf (pointed simplicial presheaf)
 given in
degree $n$ by
\be \begin{equation}
\label{half.smash}
    (K \times P)_n = K_n \times P_n,  ( (K \rtimes P )_n = (K_n \times P_n)/K_n \times {*} )
    \end{equation} \ee
\vskip .3cm \noindent
with the structure maps induced from $P$ and $K$ and where $*$ denotes the base point of $P$. This defines a bi-functor
\[\times: (simpl.sets) \times \PSh \ra \PSh, (\rtimes: (simpl.sets) \times \PSh_* \ra \PSh_*)\]
One may define a simplicially enriched Hom in $\PSh$ ($\PSh_*$) by letting it be denoted $Map(P, Q)$, for $P, Q \eps \PSh$ 
and defined by 
\be \begin{equation}
     \label{enriched.hom}
Map(P, Q)_n = Hom(\Delta[n] \times P, Q)  (Map(P, Q)_n = Hom(\Delta[n] \rtimes P, Q), \res.)
    \end{equation} \ee
\vskip .3cm \noindent
One verifies readily that this defines a closed simplicial model structure on the categories $\PSh$ ($\PSh_*$, \res). A similar definition provides 
the structure of a closed simplicial model structure on various subcategories of $\PSh$, $\PSh_*$ and on the categories 
$\PSh^{\O_{\group}^o}$ ($\PSh_*^{\O_{\group}^o}$, \res.)
\vskip .3cm
One defines a {\it symmetric monoidal structure} on the categories $\PSh$ and $\PSh^{\O_{\group}^o}$ by
\be  \begin{equation}
     \label{smash.0} 
P \wedge Q = (P \times Q)/(*\times Q \cup P \times *)
    \end{equation} \ee
\vskip .3cm 

On the other hand, for a fixed smooth scheme $S \eps \bS$, the categories $\PSh/S$ and $\PSh^{\O_{\group}^o}/S$ will have the monoidal structure defined
as follows. Recall that  objects in the above categories consist of  simplicial presheaves $P$ which come equipped 
with  
maps $p_P: P \ra S$ and $s_P: S \ra P$ so that $s_P$ is a section to $p_S$. Maps between two such simplicial presheaves in the category 
$\PSh/S$ ($\PSh^{\O_{\group}^o}/S$) 
$f:P \ra Q$ will
be maps of simplicial presheaves that are compatible with extra structure: see ~\eqref{rel.case}.
 Next if $f: F \ra G$ is a monomorphism of objects in $\PSh/S$, we define $G/F$ to be the pushout $P$:
\be \begin{equation}
     \label{S.quotient}
\xymatrix{{F} \ar@<1ex>[r]^f \ar@<1ex>[d]^{p_F} & {G} \ar@<1ex>[d]\\
           {S} \ar@<1ex>[r] &{P}}
    \end{equation} \ee
Since all the other three vertices have compatible maps to $S$, one obtains an induced map $P \ra S$, so that the composition 
$S \ra P \ra S$ is the identity. Therefore, if $P, Q \eps \PSh/S$, we define their smash-product (over $S$) as:
\be \begin{equation}
\label{smash.over.S}
P \wedge^S Q= (P \times _S Q)/(s_P(S) \times_S Q \cup_{s_P(S) \times_S s_Q(S)} P \times s_Q(S))
\end{equation} \ee
\vskip .3cm \noindent
where the quotient is taken in the sense of ~\eqref{S.quotient}. 
\vskip .3cm 
There are at least three commonly used model structures on categories of diagrams with values in a combinatorial model category like 
$\PSh$, $\PSh_*$ or $\PSh/S$: (i) {\it the object-wise model structure}, (ii) 
{\it the local injective model structure} and (iii) {\it the projective model structure}. The object-wise model structure often has the
advantage that it has more cofibrant objects and the cofibrancy condition is easy to describe, whereas in the projective model
structure one has more fibrant objects and the fibrancy condition is easy to describe. Therefore, we will find that one model
structure is often more advantageous than the other, depending on the application in mind. Moreover, it is possible to
establish important properties of one model structure with the help of the properties of the other. Therefore, we will consider all these
model structures.
\vskip .3cm
By viewing any simplicial presheaf as a diagram with values
in simplicial sets, we obtain the model-structure, where weak-equivalences and
cofibrations are defined {\it object-wise} and fibrations are defined using the
right-lifting property with respect to trivial cofibrations. (i.e. A map of simplicial presheaves $P' \ra P$ is a cofibration
(weak-equivalence) if for each $U \eps {\rm {Sm/S}}$, the induced map $\Gamma (U, P') \ra \Gamma (U, P)$ is
a cofibration (a weak-equivalence, \res).)  Observe as a
consequence, that all objects are cofibrant and cofibrations are 
simply monomorphisms. This is the {\bf object-wise model structure}, often also
called the {\it injective model structure}: but we will refer to this always as the
object-wise model structure.
\vskip .3cm
 All the above sites have enough points: for the \'etale site these are the geometric
points of all the schemes considered, and for the Nisnevich and Zariski sites these are
the usual points (i.e. spectra of residue fields ) of all the schemes
considered. Therefore, it is often convenient to provide
 the simplicial topoi above with the {\bf  local injective model structure} where the
cofibrations
are defined as before, but the weak-equivalences are maps that induce
weak-equivalences on the stalks and fibrations are defined by the lifting
property with respect to trivial cofibrations.
 By providing the appropriate topology $?$ on the category ${\bS/S}$, one
may define a similar  injective model structure on $\PSh$ and on $\PSh/S$. Since this
depends on the topology $?$, we will denote this model  structure by
$\PSh_{?}$ and $\PSh/{S_?}$.  We will often refer to this model structure simply
as the injective model structure.
\vskip .3cm
{\bf The projective model structure} on $\PSh$ is defined by taking fibrations
(weak-equivalences) to be maps that are fibrations (weak-equivalences, \res) object-wise, i.e. of  simplicial
sets on taking sections over each object in the category $\bS _{\B}$. The cofibrations are
defined by the left-lifting property with respect to trivial fibrations. This model structure
is cofibrantly generated, with the generating cofibrations $\oI$ (generating trivial cofibrations $J$)
defined as follows.
\be \begin{align}
     \label{IJ.1}
\oJ&= \{ \Lambda[n] \times X  \ra  \Delta [n] \times X|n > 0\} \mbox{ and }\\
\oI&=  \{ \delta \Delta [n] \times X \ra  \Delta [n] \times X|n \ge 0\} \notag
    \end{align}\ee
\vskip .3cm \noindent
Here $X$ denotes an object of the category $\bS$. 
Replacing $\times$ with $\rtimes$ and $X$ with $h_X$, where $h_X$ denotes the corresponding pointed presheaf represented by $X_+$ defines the projective model structure on 
category of pointed simplicial presheaves, $\PSh_*$. (We use the convention that $ \phi  \rtimes  h_X  = \B$ where $\phi$ denotes the empty set (obtained above
as $\delta \Delta[0]$).) Replacing $\Lambda[n]$ ($\delta \Delta[n]$, $\Delta [n]$, $X$) by the pointed objects $\Lambda[n]_{+_S}$ 
($\delta \Delta[n]_{+_S}$, $\Delta [n]_{+_S}$, $X_{+_S}$) and $\times$ with $\wedge_S$, provides the corresponding projective model structure on 
$\PSh/S$.
\vskip .3cm
\begin{remark} One may see by taking $n=0$ that the objects $X$ ($h_X$) are cofibrant in $\PSh$ ($\PSh_*$, \res.) 
Now one may use ascending induction on $n$ to see that
 the sources (and hence the targets) of the maps in $\oI$ and $\oJ$ are cofibrant.
\end{remark}
We provide the following result that enables one to easily deduce model structures with good properties on $\PSh/S$ starting with model structures
on $\PSh$. We begin by defining a {\it free functor}
\be \begin{equation}
 \label{free.1}
\F: \PSh \ra \PSh/S \mbox{ by } \F(P) = P_{+_{S}} = P \sqcup S
\end{equation} \ee
\vskip .3cm 
Observe that since every object $U \eps \bS$ has a unique (structure) map to $\B$ and all maps in the category $\bS$ are over $\B$, it follows that the 
presheaf corresponding to $\B$ is just the trivial pointed presheaf $*$. Therefore, every $P \eps \PSh$ has a unique map to $\B$ (the latter viewed
as the corresponding presheaf). Since $S$ is provided with a chosen map $\B \ra S$, we define a map $P_{+_{S}} = P \sqcup S \ra S$ by
sending $P \ra \B \ra S$ and $S$ by the identity to $S$. One also defines a map $S \ra P_{+_{S}}$
by mapping $S$ by the identity to the summand $S$. These show that the functor $\F$ takes values in $\PSh/S$. The functor $\F$ has a right adjoint, namely
the underlying functor $U$ that sends a simplicial presheaf in $\PSh/S$ to the same simplicial presheaf forgetting the pointing by $S$.
\begin{proposition}
 \label{passge.to.the.rel.case}
Let $I$ ($J$) denote the generating cofibrations (generating trivial cofibrations) for a cofibrantly generated symmetric monoidal model structure
on $\PSh$. Let $S$ denote a fixed object in $\bS$. Then $\F(I), \F(J)$ defines a cofibrantly generated symmetric monoidal category
structure on $\PSh/S$. Moreover, in the resulting model category on $\PSh/S$, a map $f$ is a cofibration (fibration, weak-equivalence) if and 
only if $U(f)$ is a cofibration (fibration, weak-equivalence, \res) in $\PSh$.
\end{proposition}
\begin{proof} To see that one obtains a cofibrantly generated model category this way, it suffices to prove the following (see
 \cite[Theorem 11.3.2]{Hirsch}). 
\vskip .3cm
(i) $U$ takes relative $\F(J-cells)$ to weak-equivalences and (ii) $\F(I)$ and $\F(J)$ permit small object argument, i.e. the domains of
$\F(I)$ ($\F(J)$) are small relative to $\F(I)$ ($\F(J)$).
\vskip .3cm
One begins by observing that if 
\vskip .3cm
\xymatrix{{\sqcup _{\alpha} \F(A_{\alpha})} \ar@<1ex>[r] \ar@<1ex>[d] & C \ar@<1ex>[d] \\
{\sqcup _{\alpha} \F(B_{\alpha})} \ar@<1ex>[r] & D }
\vskip .3cm \noindent
is a co-cartesian square, then so is the outer square in
\vskip .3cm
\xymatrix{{\sqcup_{\alpha} A_{\alpha}} \ar@<1ex>[r] \ar@<1ex>[d] &{\sqcup _{\alpha} \F(A_{\alpha})} \ar@<1ex>[r] \ar@<1ex>[d] & C \ar@<1ex>[d] \\
{\sqcup B_{\alpha}} \ar@<1ex>[r] &{\sqcup _{\alpha} \F(B_{\alpha})} \ar@<1ex>[r] & D .}
\vskip .3cm \noindent
This is clear since both squares are co-cartesian. It follows therefore, by taking each of the maps $A_{\alpha} \ra B_{\alpha}$ to be in
$I$ ($J$) that $U(\F(I-cell)) \subseteq I-cell$ and $U(\F(J-cells)) \subseteq J-cell$. This readily proves (i) since every $J-cell$ is a weak-equivalence.
Let $A \ra B$ be in $I$ and let $\{C_0 \ra \cdots C_{\alpha}\ra C_{\alpha+1} \ra \cdots\}$ denote a direct system of maps where each
map $C_{\alpha} \ra C_{\alpha+1}$ belongs to $\F(I-cell)$. Assume one is given a map $\F(A) \ra \colim_{\alpha} C_{\alpha}$. Then, by adjunction,
this corresponds to a map $A \ra U(\colim_{\alpha} C_{\alpha}) = \colim_{\alpha} U(C_{\alpha})$. Now each structure map $U(C_{\alpha}) \ra U(C_{\alpha +1})$
 is a map in $I-cell$, by the observation above. Since the domains of $I$ are small relative to $I-cell$, it follows that the  map
$A \ra U(\colim_{\alpha} C_{\alpha}) = \colim_{\alpha} U(C_{\alpha})$ factors through some $U(C_{\alpha_0})$. This proves the domains of $\F(I)$ are  small relative to $\F(I)-cell$ and a similar argument proves
the corresponding statement for $\F(J)$. Therefore, these complete the proof of the existence of a cofibrantly generated model structure on $\PSh/S$.
\vskip .3cm
Clearly a map $f:E \ra B$ is a fibration (weak-equivalence in $\PSh/S$ if and only if $U(f)$ is a fibration (weak-equivalence) in $\PSh$.
Next let $i:C \ra D$ denote a map in $\PSh/S$ so that $U(i)$ is a cofibration and let 
\vskip .3cm
\xymatrix{{C} \ar@<1ex>[r] \ar@<1ex>[d] & {E} \ar@<1ex>[d]\\
{D} \ar@<1ex>[r] & {B}}
\vskip .3cm \noindent
denote a commutative square in $\PSh/S$ with $E \ra B$ a trivial fibration. Since $U(i)$ is a cofibration, one obtains a lifting $U(D) \ra U(E)$.
However, since all the maps in the above square are maps in $\PSh/S$, they preserve the structure of being pointed by $S$. Therefore, the
above lifting also preserves the pointing, so that it is a map in $\PSh/S$. Therefore $i$ is a cofibration $\PSh/S$. Conversely, if $i$ is
a cofibration $\PSh/S$, then it is a retract of an $\F(I-cell)$ and therefore, $U(i)$ is a retract of an $U(\F(I-cell))$. Since every $U(\F(I-cell))$ is
an $I-cell$, $U(i)$ is a retract of an $I-cell$. Therefore $U(i)$ is a cofibration in $\PSh$.
\vskip .3cm
Now it suffices to observe that pushout-product axiom is satisfied. Here the key-observation is that $A_{+_{S}} \wedge _S B_{+_{S}} $ identifies
canonically with $(A \times B) _{+{S}}$, for any two $A, B \eps \PSh$. Therefore, the pushout-product axiom in $\PSh/S$ reduces to the
pushout-product axiom in $\PSh$. 
\end{proof}
\begin{remark} In case $S = \B$, then the model category $\PSh/S$ reduces to $\PSh_*$. (Recall there is one and only one unique map from
any $U \eps \bS$ to $\B$ commuting with the structure maps to $\B$: therefore $\B$ represents the trivial presheaf $*$.)
 \end{remark}
Next we summarize a few useful observations in the following proposition.  
\begin{proposition}
(i) All the above model structures on $\PSh$, $\PSh_*$ and $\PSh/S$ are {\it locally presentable}, {\it combinatorial} and {\it tractable} model categories
which are, in particular, cofibrantly generated. They are also left and right proper. The projective model category structure is also weakly finitely generated in the 
sense of \cite[Definition 3.4]{Dund1}  and cellular.
\vskip .3cm
(ii) If $f:P \ra Q$ is a cofibration in the projective model structure, it is a cofibration in the object-wise 
and local injective model structures.
Any injective map of simplicial presheaves is a cofibration in the object-wise model and local injective model 
structures. Any fibration in the local injective model structure is a fibration in the object-wise
model structure and any fibration in the object-wise model structure is a fibration in the projective model structure. The stalks of
any fibration in the projective model structure are fibrations of pointed simplicial sets.
\vskip .3cm 
(iii) If $P \eps \PSh$, $\PSh_*$ or $P \eps \PSh/S$ is fibrant in the object-wise model
structure or projective model structure, $\Gamma (U, GP)= \holimD \Gamma(U, G^{\bullet}P)$ is a fibrant
pointed simplicial set, where $G^{\bullet}P$ is the cosimplicial object defined
by the Godement resolution and $U$ is any object in the site.
\vskip .3cm
(iv) 
The  functor $a:{\rm {PSh}}(\bS) \ra {\rm {Sh}}(\bS_ {?})$
that sends 
a presheaf to its associated sheaf sends fibrations to maps that are fibrations stalk-wise and
weak-equivalences to maps that are weak-equivalences stalk-wise, when ${\rm {PSh}}(\bS)$
is  provided with the 
projective model structure. The same functor sends cofibrations to injective maps (monomorphisms) of the associated sheaves.
\end{proposition}
\begin{proof}  (i) Every simplicial presheaf is a filtered colimit of finite colimits of simplicial presheaves of 
the form $\Delta [n] \times X$, $X \eps \bS$. This proves that the category $\PSh$ is locally presentable.  That it is 
cofibrantly generated is well-known: see \cite{Dund2} and also \cite[Proposition A.2.8.2]{Lur}. Moreover the choice of the sets $I$ and $J$ show that the domains 
and codomains of the maps there are cofibrant. It follows the above model structures are tractable and therefore also
combinatorial: see \cite{Bar} for the terminology. Very similar arguments apply to the categories $\PSh_*$ and $\PSh/S$.
\vskip .3cm
In fact if $P \eps \PSh/S$, then viewing $P$ as an object in $\PSh$ and
applying the last result shows it is a filtered colimit of finite colimits of presheaves of the form $\Delta[n] \times X$.
Since $P$ is provided with a structure map $p_P:P \ra S$, each of the above $\Delta[n] \times X$ forming the above colimit
also has a structure map to $S$ and that these are also compatible. i.e. A $P \eps \PSh/S$ is a filtered colimit of finite colimits
of objects of the form $(\Delta[n] \times X)_+$ in $\PSh/S$, with the direct system being also in $\PSh/S$, which shows $\PSh/S$ is also
locally presentable. In view of Proposition ~\ref{passge.to.the.rel.case}, now it follows that $\PSh/S$ is
also a combinatorial and also tractable model category.
\vskip .3cm 
  Observe that weak-equivalences in the object-wise and projective model structures
are the same. Since cofibrations are defined object-wise in both the object-wise and
local injective model structures, it follows that a pushout along any such map sends weak-equivalences to weak-equivalences. 
(To see this in the local injective model structure, observe that a cofibration remains a cofibration at each stalk and pushouts commute with
taking stalks.) Therefore,
both the object-wise and injective model structures are left-proper. Since any cofibration in the projective model structure is also
a cofibration in the object-wise model structure and the weak-equivalences in these two model structures are the same, it follows that the
projective model structure is also left-proper. 
\vskip .3cm
In order to prove the remaining assertions in (i), one needs to use the assertions in (ii) which are all clear from the definitions.
The right-properness is clear in the projective model structure. Since
any map that is a fibration in the  object-wise model structure is also a fibration in the projective model structure, and the 
weak-equivalences in both model structures are the same, 
it follows that the object-wise structure is also right proper. Since any fibration in the local injective model structure is a fibration in
the projective model structure and the stalks of a fibration in the projective model structure are fibrations of pointed simplicial sets,
one sees that the local injective model structure is also right proper. The cellularity is clear for the projective model structure from the 
choice of the generating cofibrations and generating trivial cofibrations. See \cite[Definition 3.4]{Dund1} for the
definition of a weakly finitely generated model category. The weak finite generation is clear for the model structures 
we are considering in view of the choice of the sets $\oI$ and $\oJ$ as in ~\eqref{IJ.1}. These prove all the statements in (i) and (ii).
\vskip .3cm
The proof of (iii)
amounts to the observation that if $F$ is a  pointed simplicial presheaf that is
a fibration object-wise, then the
 stalks of $F$ are fibrant pointed simplicial sets. It follows that, therefore, 
$\Gamma (U, G^nF)$ are
all fibrant pointed simplicial sets. (See \cite[Chapter XI, 5.5 Fibration
lemma]{B-K}.) The first statement in (iv) follows from the observation that
fibrations and weak-equivalences are defined object-wise and these are preserved
by the filtered colimits involved in taking stalks.
The second statement in (iv) follows from the observation that a 
 map of sheaves is a monomorphism if and only 
if the induced maps on all the stalks are monomorphisms and the observation that any cofibration of presheaves is an objectwise
monomorphism in both the projective and object-wise model structures.
\end{proof} 
\vskip .4cm
 Let $X$, $Y$ denote two objects in $ \bS$ and let $f: X \ra Y$
denote a  morphism of $\bS$. Then $f$ induces functors: 
\be \begin{align}
\label{maps.of.sites}
f^{-1}: ({\bS/Y}) &\ra ({\bS/X}), f^{-1}(V) = V{\underset {Y} \times}X, \hbox{ and}\\
f_!: ({\bS/X}) &\ra ({\bS/Y}), f_!(U) = U. \notag \end{align} \ee
\vskip .3cm \noindent
The functor $f^{-1}$ now defines functors
\be \begin{equation}
     \label{map.of.topoi.1}
f_*:\PSh_X \ra \PSh_Y \mbox{ and }  f^*: \PSh_Y \ra \PSh_X 
    \end{equation} \ee
\vskip .3cm \noindent
with the former right-adjoint to the latter. The functor $f_*$ is defined by  $P \mapsto P \circ f^{-1}$. 
(One also obtains similar functors
at the level of sheaves.) Observe that the given map $f: X \ra Y$ induces a 
map of presheaves $Y \ra f_*X$ (as may be seen by applying these to a $V \eps \bS/Y$:
$Hom_Y(V, Y) \ra Hom_Y(V, f_*X) =Hom_X(V{\underset Y \times}X, X)$). Therefore, any
$P \eps \PSh/Y$ has a unique map $P \ra Y \ra f_*X$ which by adjunction corresponds
to a map $f^*P \ra X$. 
\vskip .3cm
One first sees that for a representable presheaf $P$, the composition $P \circ f_!$ identifies
with $f^*(P)$: since $f^*$ has a right-adjoint (namely $f_*$) it commutes with colimits and
therefore $P \circ f_!$ identifies with $f^*(P)$. This will show the functor $f^*$ has a left-adjoint
which we denote by $f_{\#}$. Observe that, now $Hom_X(P, X) = Hom_X(P, f^*(Y)) = Hom_Y(f_{\#}(P), Y)$.
 It follows, in view of the above adjunction, that $f_{\#}$ commutes with colimits,
and that when $f$ is a morphism of $\bS$, $f_{\#}(U) = U$ for any $U \eps {\rm {\bS/X}}$.
We see readily that, when $f$ is a morphism of $\bS$ and when $f$ itself has a section,  the functors $f^*$ and $f_{\#}$ define induced functors
\be \begin{equation}
     \label{map.of.topoi.2}
f_{\#}:\PSh/X \ra \PSh/Y \mbox{ and } f^*: \PSh/Y \ra \PSh/X
    \end{equation} \ee
\vskip .3cm \noindent
with $f_{\#}$ left-adjoint to $f^*$. (In fact $f^*$ exists even if $f$ does not have section.)
Since we do not make use of the above functors in ~\eqref{map.of.topoi.2} except those
in ~\eqref{map.of.topoi.1}, we skip any further discussion on the functors in ~\eqref{map.of.topoi.2}  presently.
\vskip .3cm
In case $Y=  \B$ with $f=p_X$ the structure map
of $X$, then $f^*$ is simply the restriction of
a presheaf (sheaf) to  ${\bS/X}$. Therefore if $F \eps {\rm {PSh}}({\bS})$, we will denote this restriction of $F$ to ${\bS/X}$
by $F_{|X}$. In general, neither of the functors $f^*$ nor $f_*$ preserve
any of the model category structures so that one needs to consider the
appropriate derived functors associated to these functors. The left-derived
functor of $f^*$, $Lf^*$, may be defined using representables  as in \cite[Chapter 2]{MV} and
the right derived functor of $f_*$, $Rf_*$, may often be defined using the Godement
resolution in view of the observation that all our sites have enough points. In case $f$ itself is a  morphism of $\bS$,
one can readily show that $f^*$ preserves
cofibrations and weak-equivalences in  all three of the above model
structures. In the case $f$ is again assumed to be a morphism of $\bS$, $Rf_*$
preserves fibrations and weak-equivalences in the  projective
model structures. (To see this one first observes that filtered colimits of fibrations of simplicial sets are fibrations, so that each 
 $f_*G^n$ will preserve object-wise fibrations. Now 
homotopy inverse limits preserve fibrations.)

\subsubsection{\bf  Convention} 
\label{convent.0} By default we will assume {\it either the  projective or the object-wise model structure} with
respect to one
of the given Grothendieck topologies. Objects of $\PSh(\bS)$  will often be referred to
as {\it spaces}. 

\vskip .3cm
\subsection{\bf Equivariant presheaves and sheaves}
As we pointed out earlier we allow the group $\group$ to be one of the following:
\vskip .3cm
(i) a finite group,
\vskip .3cm
(ii) a profinite group or,
\vskip .3cm
(iii)  a presheaf of groups on $\bS$.
\vskip .3cm \noindent
It is often convenient to view a finite or profinite group $\group$ also as a group object in $\bS$ (over $\B$) by replacing $\group$ by
$\group \otimes \B = \sqcup_{\group} \B$, with the group structure on $\group \otimes \B$ induced from the group structure of $\group$. 
In all these cases, one defines a $\group$-equivariant presheaf of sets (pointed
sets) on ${\bS/\B}$ to be a presheaf  $P$ on ${\bS/\B}$ taking
values in the category of sets (pointed sets provided) with an action by $\group$. i.e. One
defines $\group$-equivariant presheaves of simplicial sets, pointed simplicial
 sets and $\group$-equivariant sheaves of simplicial sets and pointed simplicial
sets similarly. 
Henceforth $\PSh^{\group}$ ($\Sh_{?}^{\group}$) will denote the category of
$\group$-equivariant pointed simplicial presheaves (sheaves on the $?$-topology, \res).  
Observe that $ P \eps \PSh^{\group}$ if one is provided with actions of $\group(U) $ on $P(U)$, which are compatible as
$U$ varies in $\bS$. For a fixed $S \eps \bS$, if $\group$ is a group-object defined over $S$, 
$\PSh/S^{\group}$ ($\Sh/S_?^{\group}$) will 
denote the corresponding categories of pointed simplicial presheaves (simplicial sheaves) pointed by $S$. Clearly this 
category is closed under all small limits and colimits.
\vskip .3cm 
\begin{remark} Our discussion through ~\ref{Phi.Theta.props} will only consider the categories $\PSh^{ \group}$ and $\PSh/S^{\group}$
 explicitly though everything that is said here applies equally well to the localized
categories of simplicial presheaves considered in ~\ref{desc:a1} as well.
\end{remark}

\subsubsection{\bf Simplicial sets with action by a finite (or discrete) group $\group$}
\label{simpl.sets}
A related context that comes up, especially in the context of \'etale realization is 
the following: let $(simpl.sets)$ ($(simpl. sets)_*$) denote the category of simplicial sets (pointed simplicial sets)
and let $(simpl. sets, \group)$ ($(simpl. sets, \group )_*$) denote the category of all simplicial sets (pointed simplicial sets)
with actions by the finite (or discrete) group $\group$ with the morphisms in this case being equivariant
maps. By considering the constant presheaves associated to a simplicial set (pointed simplicial set),
we may imbed this category into $\PSh^{\group}$ ($\PSh_*^{\group}$, \res). We let $C$ denote the functor 
associating a simplicial set with its associated constant presheaf.
\vskip .3cm
To obtain an imbedding of $(simpl. sets, \group)$ 
into the category $\PSh/S^{\group}$, we simply take the composition of the
two functors $\F \circ C: (simpl. sets, \group) \ra \PSh/S^{\group}$, where $\F$ denotes the free functor in ~\eqref{free.1}.
\vskip .3cm
At least when $\group$ is trivial, one may readily see that both the functors $C$ and $\F$ are left Quillen functors at the level of
the associated model categories when $\PSh^{\group}$, $\PSh_*^{\group}$ and $\PSh/S^{\group}$ are provided with the object-wise model
structures.
\subsubsection{\bf Simplicial presheaves with continuous action by the group $\group$} 
\label{cont.action}
\vskip .3cm
We let $\W$ denote a family of subgroups of $\group$ so that it has at least the following properties 
\be  \begin{enumerate}[ \rm (i)]
        \item it is an inverse system ordered by inclusion,
        \item if $H \eps \W$, $H_{\group}$ =the core of $H$, i.e. the largest subgroup of $H$ that is normal in $\group$ belongs to $\W$ and
         \item if $H \eps \W$ and $H' \supseteq H$ is a subgroup of $\group$, then $H' \eps \W$. 
\end{enumerate} \ee
In the case $\group$ is a finite
group, $\W$ will denote all subgroups of $\group$ and when $\group$ is profinite, it will denote all subgroups of finite index in 
$\group$. When $\group$ is a presheaf of groups, we will leave $\W$ unspecified for now. When $\bS$ denotes a category of schemes with the 
terminal object $\B$ and $\group$ denotes a smooth affine group-scheme defined over the base-scheme $\B$, we will again leave $\W$ unspecified, for now, 
but nevertheless
 require that each $H \eps\W$ is a closed sub-group scheme of $\group$ and that both $H$ and the homogeneous space $\group/H$ are defined over 
$\B$.
 Clearly the family of all closed subgroup-schemes 
of a given smooth affine group-scheme that are defined over $\B$ (when the base scheme $\B$ is a field) satisfies all of the above properties 
(see \cite[12.2.1 Theorem]{Sp}), so that we may use this as a default choice of $\W$ when nothing 
else is specified.
\vskip .3cm
For the most part, we will only consider $\group$-equivariant
presheaves of pointed simplicial sets $P$ on which the action of $\group$ is {\it
continuous}, i.e. for each object $U \eps \bS/S$ and each 
 section $s \eps \Gamma (U, F)_n$, the stabilizer $Z(s)$ of $s$ belongs to $\W$.
This full category (the subcategory pointed over $S$) will be denoted $\PSh^{\group, c}$ ($\PSh/S^{\group,c}$, \res).
\vskip .3cm
Clearly the above terminology is taken from the case where the group $\group$ is profinite.
 Observe that, in this case the intersection of the conjugates 
$\bigcap _{g} g^{-1}Z(s)g$ of this stabilizer as $g$ varies over a set of coset representatives of 
$Z(s)$ in $\group$ is a normal subgroup of $\group$, contained in $Z(s)$ and  of finite index in $\group$. Therefore $\group$ acts 
continuously on $P$ if and only if
$ \Gamma (U, P) = {\underset {  \{H \mid |\group/H| < \infty, H \lhd \group \}} \colim } \Gamma (U, P)^H$.
\vskip .3cm
For each  subgroup $H \eps\W$, let $P^H$ denote the sub-presheaf
of $P$ of sections fixed by $H$, i.e. $\Gamma (U, P^H) = \Gamma (U, P)^H$. If $H$ is a normal subgroup of $\group$ and 
$\bar H= \group /H$, then
\[\Gamma (U, P)^H = \{s \eps \Gamma (U, P)| \mbox { the action of $\group$ on $s$ factors through } \bar H\}.\]
\vskip .3cm
\subsubsection{The $\group$-equivariant sheaves $\group/H_+$, $H \eps \W$}
\label{GmodH}
We consider this in two distinct contexts, one as an object of $\PSh^{\group}_*$ and then as an object of $\PSh^{\group}/S$ for a fixed $S \eps \bS$.
In the first case if $H \eps \W$, we let $\group/H_+ $ denote the set of orbits of $H$ in $\group$ together with the addition of 
a base point $*$.  Now if $P \eps \PSh_*$, $\group/H_+ \wedge P \eps \PSh^{\group}_*$ is the obvious pointed simplicial presheaf,
where $\group$ acts on the factor $\group/H$. In this case we observe that for a fixed $H \eps \W$, the functor
\[\PSh^{\group}_* \ra \PSh_*, Q \mapsto Q^H \mbox{ has as left adjoint the functor } P \mapsto (\group/H)_+ \wedge P.\]
\vskip .3cm
To define $\PSh/S^{\group}$, we will recall our basic hypotheses on $\group$ and $\W$: that $\group$ is a  group object
defined over $S$
and that for all $H \eps \W$, $H$ is a subgroup object over $S$ with $H$ and $\group/H$ defined over $S$. 
 We let $\group/H \otimes S =\group/H$
as an object over $S$. If, however, $G$ is either a discrete group or a pro-finite group and $H$ is a subgroup, $\group/H$ has no structure as
an object over $S$. Therefore, in this case we let  $\group/H \otimes S = \sqcup_{\group/H}S$.
 This is now an object over $S$ with an obvious action by $\group$. (Observe that, in effect we are replacing $\group$ ($H$) by $\group \otimes S$
($H \otimes S, \res)$ which are group-objects over $S$ and then taking the quotient.)
\vskip .3cm
In both cases, 
 we let $\group/H_{+_{S}} = \F((\group/H) \otimes S)) = (\group/H \otimes S) \sqcup S$ which is an object of $\PSh/S^{\group}$.
 We observe that we obtain the adjunction:
\[\PSh/S^{\group} \ra \PSh/S, Q \mapsto Q^H \mbox{ has as left adjoint the functor } P \mapsto (\group/H)_+ \wedge P.\]
\begin{proposition} 
\label{key.props}
(i) Let $\phi: P' \ra P$ denote a map of simplicial presheaves in $\PSh^{\group}_*$. Then $\phi$ induces a
 map $\phi^H : {P'}^H \ra P^H$ for each subgroup $H \eps \W$. The association 
$\phi \mapsto \phi^H$ is functorial in $\phi$ in the sense that if $\psi: P'' \ra P'$ is another map,
then the composition $(\phi\circ \psi)^H = \phi^H \circ \psi^H$.
 \vskip .3cm \noindent
(ii) Let $\{Q_{\alpha}|\alpha\}$ denote a direct system  of simplicial sub-presheaves of a simplicial
presheaf $Q \eps \PSh^{ \group}_*$ indexed by a small filtered category. If $K$ is any subgroup of $\group$, then
$({\underset {\alpha} \colim} \, Q_{\alpha})^K = {\underset {\alpha} \colim} \, Q_{\alpha}^K$.
\vskip .3cm \noindent
(iii) The full subcategory $\PSh^{\group,c}_*$ of simplicial presheaves with continuous
action by $\group$ is closed under all small colimits, with the small colimits the same as those 
computed in $\PSh^{\group}_*$. 
\vskip .3cm \noindent
(iv) The full subcategory $\PSh^{\group,c}_*$ is also closed under all small limits,
where the limit of a small diagram $\{P_{\alpha}|\alpha\}$ is 
\be \begin{equation}
     \label{lim.cont.action}
{\underset { \{H \eps \W\}} \colim } {\underset {\alpha} \lim } \, P_{\alpha}^H
    \end{equation} \ee
\vskip .3cm \noindent
When the inverse limit above is finite, it commutes with the filtered colimit over $H$, so that
in this case, the inverse limit agrees with the inverse limit computed in $\PSh^{\group}_*$.
\vskip .3cm \noindent
For the remaining statements, we will assume the group $\group$ is {\it profinite}.
\vskip .3cm \noindent
(v) Let $\{P_{\alpha}|\alpha \}$ denote a diagram of objects in $\PSh^{\group, c}_*$ 
indexed by a small filtered category. Let $K$ denote a  subgroup of $\group$ with finite index.
Then $({\underset {\alpha} \colim} \, P_{\alpha})^K = {\underset {\alpha} \colim} \, (P_{\alpha})^K$.
\vskip .3cm \noindent
(vi) The  statements corresponding to the above also hold for the category $\PSh/S^{\group}$ and $\PSh/S^{\group,c}$.
\end{proposition}
\begin{proof} (i) is clear. 
(ii) Observe  that for a simplicial sub-presheaf $Q'$ of $Q$,
 with 
$Q' \eps \PSh^{ \group}_*$, $({Q'})^K = {Q'} \cap Q^K$. Now each $Q_{\alpha}$ maps injectively into $Q$ and the structure maps
of the direct system $\{Q_{\alpha}|\alpha\}$ are all injective maps. Therefore (ii) follows readily.
\vskip .3cm
(iii) Suppose $P= \colim_{\alpha} P_{\alpha}$, where $P_{\alpha} \eps \PSh^{\group,c}_*$. Then one may first replace each $P_{\alpha}$ by
$\bar P_{\alpha} = Image (P_{\alpha} \ra P)$. Then $\group$ acts continuously on each $\bar P_{\alpha}$ and each $\bar P_{\alpha} $ is sub-presheaf of
$P$. Therefore, (ii) applies to show that for each $H \eps \W$, $P^H= (\colim_{\alpha}\bar P_{\alpha})^H = \colim_{\alpha} (\bar P_{\alpha})^H$.
Now $\colim_{H} P^H = \colim_H (\colim_{\alpha}\bar P_{\alpha})^H=
\colim _{\alpha} \colim_H (\bar P_{\alpha})^H = \colim_{\alpha} (\bar P_{\alpha}) = P$. So the action of $\group$ on 
$P$ is continuous.
\vskip .3cm \noindent
(iv) In order to prove (iv), we show that giving  a
 compatible collection of maps $\phi_{\alpha}: P \ra P_{\alpha}$ from a simplicial presheaf $P$ with a continuous action by $\group$ 
 to the given inverse system $\{P_{\alpha}| \alpha \}$ corresponds to giving a map 
$\phi: P \ra {\underset { \{H  \eps \W\}} \colim } {\underset {\alpha} \lim } 
 \, P_{\alpha}^H$.
Observe that the maps $\phi_{\alpha}$ induce a compatible collection of maps
$\{\phi_{\alpha}^H: P^H \ra P_{\alpha}^H\}$ for each fixed $H$. Therefore, one may now take the limit of these
maps over $\{\alpha\}$ to obtain a compatible collection of maps 
$\{\phi^H : P^H \ra {\underset {\alpha } \lim} \, P_{\alpha}^H|H\}$. One may next take the colimit 
over $H$ of this collection to obtain a map
\[\phi: P = {\underset { \{H  \eps \W \}} \colim } \, P^H \ra 
{\underset { \{H \eps \W \}} \colim }{\underset {\alpha } \lim} \, 
P_{\alpha}^H.\]
\vskip .3cm
Clearly the latter maps naturally to  ${\underset {\alpha } \lim} 
{\underset { \{H \eps \W \}} \colim } \, P_{\alpha }^H$.
The latter then projects to 
\vskip .3cm
$P_{\alpha} = {\underset { \{H \eps \W\}} \colim } \, P_{\alpha }^H$. 
Moreover, the 
composition of the above maps $P \ra P_{\alpha}$ identifies with the given map $\phi_{\alpha}$ since the 
corresponding maps identify after applying $(\quad )^H$.
\vskip .3cm 
The action of $\group$
 on ${\underset { \{H \eps \W \}} \colim }   {\underset {\alpha } \lim}  \, P_{\alpha }^H$ is 
continuous. The above arguments show that the latter 
is in fact the inverse limit of $\{P_{\alpha}|\alpha \}$ in the category $\PSh^{ \group,c}_*$. This 
completes the proof of the first statement in (iv). The last statement in (iv) is clear since filtered colimits commute with
finite limits.
\vskip .3cm 
Next we consider (v). Since each $P_{\alpha}$ has a continuous action by $\group$, we first observe that 
\[P_{\alpha} = {\underset { \{H | |\group/H| < \infty, H \lhd \group \}} \colim } P_{\alpha}^H \mbox{ for each } \alpha.\]
Therefore, 
\[{\underset {\alpha} \colim} P_{\alpha} = 
{\underset {\alpha} \colim} {\underset { \{H | |\group/H| < \infty, H \lhd \group\}} \colim } P_{\alpha}^H =
{\underset { \{H | |\group/H| < \infty, H \lhd \group \}} \colim } {\underset {\alpha} \colim} P_{\alpha}^H.\]

Let $Q_H = {\underset {\alpha} \colim} P_{\alpha}^H$; then (since filtered colimits preserve monomorphisms) we see that each $Q_H$ is a sub-simplicial
 presheaf of ${\underset { \{H | |\group/H| < \infty, H \lhd \group \}} \colim } Q_H = 
{\underset {\alpha} \colim} P_{\alpha}$. Then the collection $\{Q_H|H, |\group/H| <\infty, H \lhd \group \}$
satisfies the hypotheses in (iv) so that we obtain:
\[({\underset { \{H | |\group/H| < \infty, H \lhd \group\}} \colim } {\underset {\alpha} \colim} P_{\alpha}^H)^K =
{\underset { \{H | |\group/H| < \infty, H \lhd \group\}} \colim } ({\underset {\alpha} \colim} P_{\alpha}^H)^K.\]
\vskip .3cm
Next observe that, for any simplicial presheaf $Q$ and any fixed normal subgroup $H$ of 
finite index in $\group$, the action of $\group$ on $Q^H$ is through the finite group $\group/H$. Therefore,
$(Q^H)^K = (Q^H)^{\bar K}$, where $K$ is the image of $K$ in $\group/H$. Therefore we obtain the 
isomorphism 
\[({\underset {\alpha} \colim} \, P_{\alpha}^H)^K \cong ({\underset {\alpha} \colim} \, P_{\alpha}^H)^{\bar K} 
 \cong {\underset {\alpha} \colim} \, (P_{\alpha}^H)^{\bar K} \cong {\underset {\alpha} \colim} \, (P_{\alpha}^H)^K.\]
\vskip .3cm \noindent
The second isomorphism follows from the fact that since $\bar K$ is a finite group, taking invariants with respect to $\bar K$ is
a finite inverse limit which commutes with filtered colimits.
Making use of this identification and commuting the two colimits, we therefore obtain the
identification 
\[{\underset { \{H | |\group/H| < \infty, H \lhd \group\}} \colim } ({\underset {\alpha} \colim} \, P_{\alpha}^H)^K =
{\underset {\alpha} \colim} {\underset { \{H | |\group/H| < \infty, H \lhd \group\}} \colim } \, ( P_{\alpha}^H)^K
= {\underset {\alpha} \colim} \, ( P_{\alpha})^K.\]
\vskip .3cm \noindent
The last identification follows from (ii), the assumption that $\group$ acts continuously on each $P_{\alpha}$ and the observation that for each fixed $\alpha$, each 
$P_{\alpha}^H \subseteq P_{\alpha}$ as simplicial presheaves. 
\vskip .3cm
(vi) follows since all of the above arguments are compatible with the structure maps to $S$. 
\end{proof}
\vskip .3cm 
 \subsubsection{\bf Finitely presented objects}
Recall an object $C$ in a category $\C$ is {\it finitely presented} if
$Hom_{\C}(C, \quad)$ commutes with all small filtered colimits in the second
argument. 
\begin{proposition}
\label{finite.pres.1}Let $\group$ denote a profinite group. 
Let $P \eps \PSh^{\group}_*$ be such that
in $\PSh_*$ it is finitely presented and $P =P^H$ for some normal 
subgroup $H$ of finite index in $\group$. Then $P$ is a finitely presented object in
$\PSh^{\group, c}_*$. The same conclusions also hold for a $P \eps \PSh/S^{\group}$.
\end{proposition}
\begin{proof} Let $Hom$ denote the external Hom in the category $\PSh_*$ and
let $Hom_{\group}$ denote the external Hom in the category $\PSh^{\group}_*$.
Then $Hom_{\group}(P, Q) = Hom(P, Q)^{\group}$, where $\group$ acts on the set $Hom(P, Q)$ through its actions on 
$P$ and $Q$. Next suppose $P = P^H$ for some normal subgroup $H$ of $\group$ with finite index.
Then $Hom_{\group}(P, Q) \cong Hom_{\group/H}(P, Q^H) = Hom(P, Q^H)^{\group/H}$.
\vskip .3cm
Next let $\{Q_{\alpha}| \alpha\}$ denote a small collection of objects in 
$\PSh^{\group, c}_*$ indexed by a small filtered category. Then,
\[{\underset {\alpha} \colim} Hom_{\group}(P, Q_{\alpha}) = 
{\underset {\alpha} \colim} Hom(P, Q_{\alpha}^H) ^{\group/H}= 
({\underset {\alpha} \colim} Hom(P, Q_{\alpha}^H) )^{\group/H}\]
\vskip .3cm \noindent
where the last equality follows from the fact that taking invariants with respect to the finite 
group $\group/H$ is a finite inverse limit which commutes with the filtered colimit over $\alpha$.
Next, since $P$ is finitely presented as an object in  $\PSh_*$,
the last term identifies with $(Hom(P, {\underset {\alpha} \colim} Q_{\alpha}^H) )^{\group/H}$.
By Proposition ~\ref{key.props}(v), this then identifies with 
$(Hom(P, ({\underset {\alpha} \colim} Q_{\alpha})^H ))^{\group/H}$. This clearly identifies with
$Hom_{\group}(P, {\underset {\alpha} \colim} Q_{\alpha})$. 
\end{proof}
Next we define the following structure of a cofibrantly generated simplicial model
category on $\PSh^{\group}_*$,  $\PSh^{ \group, c}_*$, $\PSh/S^{\group}$ and on $\PSh/S^{\group, c}$
starting with the projective
model or object-wise model structure on $\PSh_*$ and on $\PSh/S$. Let $\oI$ ($\oJ$) denote the generating cofibrations (generating trivial cofibrations) in
$\PSh_*$ or $\PSh/S$: recall these are given as in ~\eqref{IJ.1}  in the projective model structures. Here we use the convention that the subscript $+$ and
$ \wedge$ denote the usual ones when considering the categories $\PSh^{\group}_*$,  $\PSh^{ \group, c}_*$ while they denote $+_{S}$ and $\wedge _S$
when considering the categories $\PSh/S^{\group}$ and $\PSh/S^{\group,c}$.
\vskip .3cm 
For technical reasons, (as will become apparent, for example, in the proof of Theorem ~\ref{eq.model.struct.1}), when using the projective model
structures on  $\PSh_*$ \mbox{ or } $\PSh/S$,  we will assume that the group 
\be \begin{equation}
     \label{tech.restr.1}
 \group \mbox{ is either finite, profinite or is a constant group-object (over } S \mbox{ ) associated to
a discrete group}. 
 \end{equation} \ee
\vskip .3cm
\begin{definition} (The model structure)
\label{equiv.model.struct}
(i) The {\it generating cofibrations} are of the form 
\[\oI_{\group}=\{(\group/H )_+ \wedge i \mid   i \eps I, \quad H \subseteq \group, \quad H \eps \W \}, \]
\vskip .3cm 
(ii) the {\it generating trivial cofibrations} are of the form 
\[\oJ_{\group} = \{(\group/H ) _+ \wedge  j \mid  j \eps J, \quad H \subseteq \group, \quad H \eps \W \} \, \mbox{ and }\]
\vskip .3cm 
(iii) and the  {\it weak-equivalences} ({\it fibrations}) are maps  $f:P' \ra P$ in
$\PSh^{\group}$ so that $f^H:{P'}^H \ra P^H$ is a weak-equivalence
(fibration, \res) in $\PSh$ for all $H \eps \W$.
\end{definition}
\begin{theorem}  
\label{eq.model.struct.1}
The above structure defines a cofibrantly generated
simplicial model structure on 
\newline \noindent
$\PSh^{\group}_*$, $\PSh^{\group,c}_*$, $ \PSh/S^{\group}$ and on $\PSh/S^{\group,c}$ that is proper. 
\vskip .3cm
In addition, the 
smash product of pointed simplicial presheaves (defined as in ~\eqref{smash.0} and ~\eqref{smash.over.S}) 
make these symmetric monoidal categories satisfy the
pushout-product axiom in both the object-wise and projective model structures. The unit for the smash product
in the object-wise model structure for $\PSh^{\group}_*$ is cofibrant, while the 
unit for the smash-product in $\PSh/S^{\group}$ is cofibrant in both the object-wise and the projective model structures.
\vskip .3cm
When $\group$ denotes a finite or profinite group, the categories $\PSh^{\group,c}_*$ and $\PSh/S^{\group,c}$ 
are locally presentable, i.e. there exists a set of
objects, so that every simplicial presheaf $P \eps \PSh^{\group, c}_*$ is a filtered colimit
of these objects. In particular, they are combinatorial and tractable model categories.
\end{theorem}
\begin{proof} We will not explicitly discuss the case $\PSh^{ \group, c}_*$ or $\PSh/S^{ \group, c}$ since the arguments are the same
as for $\PSh^{ \group}_*$ and $\PSh/S^{ \group}$. The key observation is the following: Let $H \subseteq \group$ denote a
subgroup in $\W$. Then recall that  the functor $P \mapsto P^H,  \, \PSh^{\group}_* \ra \PSh_* \, (\PSh/S^{\group} \ra \PSh/S) $ has as left-adjoint, the functor
\be \begin{equation}
     \label{basic.adj}
Q \mapsto (\group/H)_+ \wedge  Q.
\end{equation} \ee
We now proceed to verify that the hypotheses of 
\cite[Theorem 2.1.19]{Hov-1} are satisfied. It is obvious that the subcategory of
weak-equivalences is closed under composition and retracts and has the two-out-of-three
property.  Next we proceed to verify that the domains of $\oI_{\group}$ are small relative to
$\oI_{\group}$-cell. First let $\{ P _{\alpha}| \alpha \}$ denote a small sub-collection of objects in $\PSh^{\group}_*$ ($\PSh/S^{\group}$)
that are sub-objects of  a  $P \eps \PSh^{\group}_*$ ($P \eps \PSh/S^{\group}$, \res). By Proposition ~\ref{key.props}(ii), one obtains the 
identification 
\[ (\colimalpha P_{\alpha})^H \cong \colimalpha (P_{\alpha})^H. \]
\vskip .3cm
We consider this first in the projective model structure.
Let $ (\group/H )_+ \wedge  (\delta \Delta [n] \rtimes h_X) \ra (\group/H )_+ \wedge  ( \Delta [n] \rtimes h_X)$
denote a generating cofibration in $\oI_{\group}$. Suppose one is given a map $f:  (\group/H )_+ \wedge (\delta \Delta [n] \rtimes h_X) \ra 
\colimalpha P_{\alpha}$.
By the adjunction in ~\eqref{basic.adj}, this map corresponds to a map $f^H:(\delta \Delta [n] \rtimes h_X) \ra (\colimalpha P_{\alpha})^H
= \colimalpha (P_{\alpha})^H$.
Since $(\delta \Delta [n] \rtimes h_X)$ is small in $\PSh_*$ (and in $\PSh/S$), the map
$(\delta \Delta [n] \rtimes h_X) \ra \colimalpha (P_{\alpha}^H)$ factors through some $P_{\alpha_0}^H$. Therefore, its adjoint
$f:(\group/H )_+ \wedge (\delta \Delta [n] \rtimes h_X) \ra \colimalpha P_{\alpha}$ factors through $P_{\alpha_0}$.
This proves the domains
of $\oI_{\group}$ are small relative to $\oI_{\group}$-cell. An entirely similar argument proves that the domains of $\oJ_{\group}$ are small 
relative to $\oJ_{\group}$-cell.
Observe that these two steps make use of the fact that the model structure on $\PSh_*$ and on $\PSh/S$ are indeed the projective
model structures.
\vskip .3cm
Next we consider the object-wise model structure.  Let 
$id \times i: (\group/H )_+ \wedge  A \ra (\group/H )_+ \wedge B$ denote a generating cofibration in $\oI_{\group}$, i.e.
the map $i:A \ra B$ is in $\oI$. Here we make use of the observation that the fixed point functor $Q \ra Q^H$ 
 has the following two properties for simplicial presheaves (see \cite[4.2, 4.3]{Guill}):
\be \begin{multline}
\begin{split}
     \label{pushouts.and.colimits.fixed.points}
\mbox{ (i) The fixed point functor }  Q \ra Q^K \mbox{ preserves pushouts along any map of the form }\\
id \times i: (\group/H)_+ \wedge A \ra (\group/H)_+ \wedge  B, \mbox{ with } i \eps \oI \mbox{ and }
\end{split}
\end{multline} \ee
\be \begin{multline}
     \begin{split}
\mbox{ (ii) if } \{ P _{\alpha}| \alpha \} \mbox{ denotes a small sub-collection of objects in } \PSh_*^{\group} (\PSh/S^{\group}) 
\mbox{ that are sub-objects of  a  } \notag \\  
P \eps \PSh_*^{\group} (\PSh/S^{\group}), \mbox{ one obtains the identification }
(\colimalpha P_{\alpha})^H \cong \colimalpha (P_{\alpha})^H. \notag
\end{split} \end{multline} \ee
\vskip .3cm \noindent
One may verify the first property as follows. First observe that  taking pushout of a diagram  in $\PSh/S^{\group}$ is the same as taking
pushout of the corresponding diagram in $\PSh^{\group}$. Secondly pushouts of pointed simplicial presheaves may be computed section-wise, i.e.
$(A \sqcup_B C)(U) = A(U) \sqcup_{B(U)} C(U)$, when $A, B, C \eps \PSh_*$ and $ U \eps \bS_S$. Therefore, for any subgroup $K \eps \W$,
\[ ((\group/H)_+ \wedge B \sqcup _{(\group/H)_+ \wedge A} P )^K(U) = (((\group/H)_+(U)^{K(U)} \wedge B(U)) \sqcup _{(\group/H)_+(U)^{K(U)} \wedge A(U)}  P(U)^{K(U)}\]
where $P \eps \PSh^{\group}_*$ or $P \eps \PSh/S^{\group}$ and $ K \eps \W$. The last equality shows that one obtains:
\[((\group/H)_+ \wedge B \sqcup _{(\group/H)_+ \wedge A} P)^K = (\group/H)_+^K \wedge B \sqcup_{(\group/H)_+^K \wedge A} P^K.\]
\vskip .3cm \noindent
Moreover observe that in the object-wise model structure, every object in $\PSh_*$ ($\PSh/S$) is cofibrant: therefore, $(\group/H)_+^K$ is cofibrant in 
$\PSh_*$ ($\PSh/S$, \res). While using the projective model structure on $\PSh_*$, $(\group/H)_+^K$ is still cofibrant in view of our restrictive hypothesis
~\eqref{tech.restr.1}. 
Now the monoidal axiom in $\PSh_*$ ($\PSh/S$) shows that 
 $(\group/H)_+^K \wedge A {\overset {(\group/H)_+^K \wedge i} \ra} (\group/H)_+^K \wedge B$ is a cofibration in $I$ (trivial cofibration)
if $i \eps \oI$ ($i \eps \oJ$, \res). 
\vskip .3cm
 Next suppose 
\vskip .3cm
\be \begin{equation}
      \label{fixed.pts.pres.cofs.0}
\xymatrix{{\vee _{\alpha} (\group/K_{\alpha})_+ \wedge A_{\alpha} } \ar@<1ex>[r] \ar@<1ex>[d]^{\vee_{is \wedge i_{\alpha}}} & {Q_{\alpha}} \ar@<1ex>[d]\\
{\vee _{\alpha} (\group/K_{\alpha})_+ \wedge B_{\alpha} } \ar@<1ex>[r] & {Q_{\alpha+1}} }
    \end{equation} \ee
is a pushout-square with each $i_{\alpha}: (\group/K_{\alpha})_+ \wedge A_{\alpha} \ra (\group/K_{\alpha})_+ \wedge B_{\alpha}$ belonging to
$\oI_{\group}$ ($\oJ_{\group}$, \res). Then for each $H \eps \W$, the above observations show that one obtains a pushout-square on taking the $H$-fixed points
and that therefore, the induced map $Q_{\alpha}^H \ra Q_{\alpha+1}^H$ is a cofibration (trivial cofibration, \res).
\vskip .3cm
Now suppose $(\group/H)_+ \wedge A$ is the domain of a map in $\oI$ ($\oJ$, \res) and one is provided with a map $(\group/H)_+ \wedge A \ra \colim_{\alpha} P_{\alpha}$,
which is an $\oI_{\group}$-cell (or an $\oJ_{\group}$-cell).  By adjunction, this corresponds to a map $A \ra (colim_{\alpha} P_{\alpha})^H = \colim_{\alpha}(P_{\alpha})^H$
with the last equality holding in view of ~\eqref{pushouts.and.colimits.fixed.points}. The above observations now show that each of the maps
$P_{\alpha}^H \ra P_{\alpha+1}^H$ is a cofibration (trivial cofibration, \res) in $\PSh_*$ or $\PSh/S$.  
At this point, \cite[Proposition 2.1.16]{Hov-1} shows that  the domains 
of maps in $\oI$
 ($\oJ$) are small relative to $\oI$-cofibrations ($\oJ$-cofibrations (i.e. $\oI$-trivial cofibrations), \res), so that the above map factors through
some $P_{\alpha_0}^H$. Taking adjoints, it follows the original map $(\group/H)_+ \wedge A \ra \colim_{\alpha} P_{\alpha}$ factors through
$P_{\alpha_0}$.
 This proves  that the domains of the maps in $\oI_{\group}$ ($\oJ_{\group}$) are
also small relative $\oI_{\group}-cell$ ($\oJ_{\group}-cell$, \res.)
\vskip .3cm
The first property in ~\eqref{pushouts.and.colimits.fixed.points} and the above arguments also show that if $f$ is obtained by co-base change from a map 
$j \eps \oJ_{\group}$, then $f^K$ is a $\oJ$-cofibration (i.e. a trivial cofibration) in $\PSh_*$ ($\PSh/S$) for any $K \eps \W$.
Therefore, any transfinite composition of maps of the form $f$ is a weak-equivalence, i.e. each map in $\oJ_{\group}$-cell is a weak-equivalence.
One may now observe using the adjunction  that the fibrations defined above identify with $\oJ_{\group}-inj$ and that the trivial 
fibrations (i.e.
the fibrations that are also weak-equivalences) identify with $\oI_{\group}-inj$. Recall from \cite{Hov-1} that $\oI_{\group}-cof$ (i.e. the $\oI_{\group}$-cofibrations)
are the maps $(\oI_{\group}-inj)-proj$, i.e. those maps that have the left lifting property with respect to every trivial fibration.  
Next we proceed to show that any map in $\oJ_{\group}$-cell is in $\oI_{\group}-cof$.
Therefore, suppose we are given a commutative square in $\PSh_*^{\group}$ or in $\PSh/S^{\group}$:
\[\xymatrix{{(\group /H)_+ \wedge A} \ar@<1ex>[r] \ar@<1ex>[d] ^{id \times j} & {X} \ar@<1ex>[d]^p\\
            {(\group/H)_+ \wedge B} \ar@<1ex>[r] & Y}
\]
with $j \eps \oJ$ and $p \eps \oI_{\group}-inj$. Then, by adjunction this corresponds to the commutative square:
\[\xymatrix{{ A} \ar@<1ex>[r] \ar@<1ex>[d] ^{ j} & {X^H} \ar@<1ex>[d]^{p^H}\\
            { B} \ar@<1ex>[r] & Y^H}
\]
in $\PSh_*$ or $\PSh/S$. Now $p^H$ is a fibration and $j$ is a generating trivial cofibration, so that one obtains a 
lifting: $B \ra X^H$ making the two triangles commute. By adjunction this lift corresponds to a lift $(\group/H )_+ \wedge B \ra X$ in the
first diagram. This proves any map in $\oJ_{\group}$ is in $\oI_{\group}-cof$. One may readily see that $\oI_{\group}-cof$ is closed under co-base change,
pointed unions and transfinite compositions so that any map in $\oJ_{\group}-cell$  also belongs to $\oI_{\group}-cof$.
\vskip .3cm
Since $\oI_{\group}-inj$ corresponds to trivial fibrations, every map in $\oI_{\group}-inj$ is a weak-equivalence and  it is in 
 $\oJ_{\group}-inj$ (which  denote the  fibrations). Since $\oJ_{\group}-inj$ denotes fibrations, it is clear that any map that is in
$\oJ_{\group}-inj$ and is also a weak-equivalence is also in $\oI_{\group}-inj$ (which denotes trivial fibrations.)
Therefore, we have verified all the hypotheses in \cite[Theorem 2.1.19]{Hov-1} and therefore the  statement that the structures
in Definition ~\ref{equiv.model.struct} define a cofibrantly generated model category in the theorem is proved.
\vskip .3cm
The left-properness may be established using the property that the fixed point functors preserve pushout along the generating cofibrations: see
~\eqref{pushouts.and.colimits.fixed.points}(i) considered above. (One may also want to observe that weak-equivalences are closed under transfinite 
compositions, which reduces to the corresponding property for simplicial sets and pointed simplicial sets.)
The right properness is clear since
the fixed point functor preserves pull-backs. The smash products are the ones defined in ~\eqref{smash.0} and ~\eqref{smash.over.S}.
\vskip .3cm
Next we prove that the model structures in Definition ~\ref{equiv.model.struct} define a symmetric monoidal model category structure on $\PSh_*^{\group}$
and $\PSh/S^{\group}$ with respect to the monoidal structures considered in ~\eqref{smash.0} and ~\eqref{smash.over.S}. Observe from 
\cite[Corollary 4.2.5]{Hov-2} that in order to prove the pushout-product axiom holds in general, it suffices to prove that the pushout product of two generating cofibrations is a cofibration and
that this pushout-product is also a weak-equivalence when one of the arguments is a generating trivial cofibration.
\vskip .3cm
We will first consider this in the projective model structure.
Therefore, let $(\group/H)_+ \wedge i: (\group/H)_+ \wedge A \ra (\group/H)_+ \wedge B$ and let $(\group/K)_+ \wedge j: (\group/K)_+ \wedge X \ra
(\group/K)_+ \wedge Y$ denote two generating cofibrations in $\PSh_*^{\group}$ or $\PSh/S^{\group}$. Then
a key observation is that $(\group/H)_+ \wedge (\group/K)_+ \cong \vee_{\alpha} (\group/K_{\alpha})_+$ where the $\vee$ is over the 
orbits of $\group$ for the diagonal action of $\group$ on $\group/H \times \group/K$, each orbit being of the form $\group/K_{\alpha}$: this is
possible only because of the restrictive hypotheses on the group $\group$ as in ~\eqref{tech.restr.1}.
(We skip the verification that $K_{\alpha} \eps \W$, which needs to be verified separately in each of the  cases we consider.)
Therefore, it suffices to prove that for $E$
a fibrant object in $\PSh_*^{\group}$ ($\PSh/S^{\group}$), the induced map
\be \begin{multline}
     \begin{split}
  \Hom((\group/( K_{\alpha}))_+ \wedge B \wedge Y, E) \ra\\
    \Hom((\group/( K_{\alpha}))_+ \wedge A \wedge Y, E) {\underset {\Hom((\group/( K_{\alpha}))_+ \wedge A \wedge X, E)} \times} \Hom((\group/(K_{\alpha}))_+ \wedge B \wedge X, E)
     \end{split}
\end{multline} \ee
is a fibration in $\PSh_*^{\group}$($\PSh/S^{\group}$, \res) which is a weak-equivalence if $i$ or $j$ is also weak-equivalence and where $\Hom$
denotes the appropriate internal hom.
The above map now identifies with
\[ \Hom( B \wedge Y, E^{ K_{\alpha}}) \ra \Hom( A \wedge Y, E^{ K_{\alpha}}) {\underset {\Hom( A \wedge X, E^{ K_{\alpha}})} \times} \Hom( B \wedge X, E^{ K_{\alpha}}) \]
where $\Hom$ now denotes the internal hom in $\PSh_*$ ($\PSh/S$, \res). Therefore, the fact that the above map is a fibration 
and that it is a trivial fibration if $i$ or $j$ is also a weak-equivalence follows from the fact that the pushout-product axiom
holds in $\PSh_*$ ($\PSh/S$).
\vskip .3cm
In the object-wise model structure the corresponding proof is much easier and holds more generally for any $\group$ satisfying our basic hypotheses in view
of the following observations: (i) one may identify $(\group/K)_+ \wedge X \wedge (\group/H)_+\wedge Y$ with $((\group/K)_+ \wedge (\group/H)_+) \wedge X \wedge Y$
and therefore $((\group/K)_+ \wedge X \wedge (\group/H)_+\wedge Y)^L \cong ((\group/K)_+ \wedge (\group/H)_+)^L \wedge X \wedge Y$ for any subgroup $L \eps \W$.
(ii) The observation in ~\eqref{pushouts.and.colimits.fixed.points} and (iii) in the object-wise model structure, every object is cofibrant and every 
monomorphism is a cofibration. In view of these, one reduces the pushout-product property in $\PSh_*^{\group}$ ($\PSh/S^{\group}$) to the corresponding
property in $\PSh_*$ ($\PSh/S$, \res).
\vskip .3cm
One may observe that the unit of the smash-product (see ~\eqref{smash.0}) in $\PSh_*^{\group}$ is just the usual $0$-sphere $S^0$ which is 
cofibrant in  the
object-wise  model structure. The unit for the smash-product (see ~\eqref{smash.over.S}) in $\PSh/S^{\group}$ is $S_+$ and this is
cofibrant in both the object-wise and projective model structures on $\PSh/S^{\group}$. 
\vskip .3cm
Next we proceed to observe that the above categories are combinatorial and tractable when $\group$ denotes either a finite or pro-finite group. 
First observe that for any $\group$, the objects of the form $(\group/H)_+ \wedge (\Delta [n] \rtimes h_U)$ as $U \eps \bS_S$, $H \eps \W$ and $n \ge 0$
 vary form a set of generators for $\PSh_*^{\group, c}$ and $\PSh/S^{\group, c}$. Now Proposition ~\ref{finite.pres.1} shows that these objects
are finitely presented as objects in $\PSh_*^{\group}$ and $\PSh/S^{\group}$ when $\group$ is a profinite group. Moreover, since
we are considering simplicial presheaves, it is easy to see that the above generators form a family of strong generators in the sense
of \cite[Definition 4.5.3]{Bor1}.
Therefore, \cite[Lemma 5.2.5]{Bor2} proves that every object in the above categories is a filtered colimits of objects $\{G_{\alpha}|\alpha\}$ 
 obtained as finite colimits of the above generators and that the objects $G_{\alpha}$ themselves are finitely presentable. 
Therefore, when $\group$ is finite or profinite, the above categories are locally presentable. (One may observe that the same conclusions also
 hold for the diagram categories $\PSh_*^{{\mathcal O}^o_{\group}}$ and $\PSh/S^{{\mathcal O}^o_{\group}}$ in general, i.e. without the restriction that 
the group $\group$ be finite or profinite: see below.)
Since these are already shown to be cofibrantly generated model categories in both the projective and object-wise model structures, it follows 
they are in fact combinatorial model categories.
\vskip .3cm
 In the projective model structure, it is clear from the choices of the sets $\oI_{\group}$ and $\oJ_{\group}$ that
every object $(\group/H)_+ \wedge  h_X$, $X \eps \bS/S$ is cofibrant. (These follow from the fact that each object $h_X$ is cofibrant in
the projective model structure on $\PSh_*$ and on $\PSh/S$.) In the object-wise model structure, all monomorphisms are cofibrations, 
so that the
domains of the sets $\oI_{\group}$ and $\oJ_{\group}$ are cofibrant. Therefore, it follows that these model structures are also tractable.
\end{proof}
\begin{remark}
 In case $\group$ is finite or pro-finite, it is also true that $(\group/H)_+^K \wedge i$ is a disjoint finite union
of copies of $i$ and hence belongs to $\oI$-cell ($\oJ$-cell) if $i \eps \oI$ ($j \eps \oJ$, \res). But this fails in general when $\group$ denotes a 
group-scheme. In view of this, one cannot call the first property in ~\eqref{pushouts.and.colimits.fixed.points} the {\it cellularity} of the fixed point 
functors in general.
\end{remark}
\vskip .3cm
\subsection{\bf Alternate approach via diagram categories}
\label{diagram.cats}
The following is an alternative approach to providing a model structure on the
category of pointed simplicial presheaves with continuous action by $\group$. We will start with either the object-wise or the projective model
structures on $\PSh_*$ and $\PSh/S$. Next one considers  the 
orbit category
${\mathcal O}_{\group}= \{\group/H \mid H \eps \W \}$. A morphism $\group/H \ra
\group/K$ corresponds to $\gamma \eps \group$, so that $\gamma.H \gamma ^{-1}
\subseteq K$. One may next consider the categories $\PSh_*^{{\mathcal
O}_{\group}^o}$ of ${\mathcal O}_{\group}^o$ -diagrams with values in ${\rm
{PSh}}_*$ ($\PSh/S^{{\mathcal
O}_{\group}^o}$ of ${\mathcal O}_{\group}^o$ -diagrams with values in ${\rm
{PSh}}/S$, \res). This category, being a category of diagrams with values in
$\PSh_*$ ($\PSh/S$) readily inherits the structure of
 a cofibrantly generated model category by providing it with the projective model structure. 
(i.e. A map $\{f(\group/H): A(\group/H) \ra B(\group/H)| H \eps \W\}$ in $\PSh_*^{{\mathcal
O}_{\group}^o}$ ($\PSh/S^{{\mathcal
O}_{\group}^o}$) is a weak-equivalence (fibration) if each $f(\group/H)$ is a weak-equivalence (fibration, \res) in
$\PSh_*$ ($\PSh/S$, \res). ) It is observed in \cite[Proposition 11.6.3]{Hirsch} that the cofibrations in the above model
structure are also object-wise cofibrations. Therefore, it is readily shown that 
the above model structure on the diagram category  is left-proper (right-proper, cellular, simplicial)
since the model category $\PSh_*$ ($\PSh/S$) is. It is also combinatorial and tractable since since the model category $\PSh_*$ ($\PSh/S$) is: see the next paragraph.
\vskip .3cm
 Recall that the generating cofibrations of this diagram category are defined as follows.
\[\oI_{{\mathcal O}^{{\group}^o}} = \{(\group/H)_+ \wedge i \mid  i \eps I,  H \eps \W \}.\] 
 The corresponding generating trivial cofibrations are
\[\oJ_{{\mathcal O}^{{\group}^o} }= \{(\group/H)_+ \wedge j \mid  j \eps J,  H \eps \W \}.\] 
The diagram $(\group/H)_+ \wedge i$ is the
${{\mathcal O}_{\group}^o}$-diagram defined by 
\[((\group /H)_+ \wedge i)(\group/K) = Hom_{{\mathcal O}_{\group}^o}(\group/H, \group/K)_+ \wedge i =
Hom_{\group}(\group/K, \group/H)_+ \wedge i = (\group/H)^K_+ \wedge i.\] 
\vskip .3cm
In fact we can define the {\it constant functor} 
\be \begin{equation}
     \label{free.diagram}
{\mathfrak C}:\PSh_* \ra \PSh_*^{{\mathcal O}^{{\group}^o} } \,  ({\mathfrak C}:\PSh/S \ra \PSh/S^{{\mathcal O}^{{\group}^o} })
\end{equation} \ee
by 
\[{\mathfrak C}(A)(\group/K) = A, \mbox{ for all } K \eps \W. \] 
This functor will be used
later on.
\vskip .3cm
Then the two categories 
$\PSh_*^{\group,c}$ and ${\rm {PSh}}_*^{{\mathcal O}_{\group}^o}$
are related by the functors:
\[\Phi:\PSh_*^{\group, c} \ra {\rm {PSh}}_*^{{\mathcal O}_{\group}^o},
\quad F \mapsto \Phi(F) = \{\Phi(F)(\group/H)=F^H\} \mbox{ and} \]
\[\Theta: {\rm {PSh}}_*^{{\mathcal O}_{\group}^o} \ra \PSh_*^{\group,c}, \quad M \mapsto \Theta (M) =
{\underset { \{H \mid {H \eps \W} \}} \colim } M(\group/H).\]
The same functors are defined at the level of the categories $\PSh/S^{\group,c}$ and $\PSh/S^{{\mathcal O}_{\group}^o}$ also.
\vskip .3cm \noindent 
Recall that $\W$ is an inverse system of subgroups of $\group$. Therefore, given two subgroups $H, H' \eps \W$, there
is a subgroup $H'' \eps \W$ contained in $H \cap H'$. Now the core of $H''$, $H''_{\group}$ is contained in $H \cap H'$, belongs to $\W$ and
is a normal subgroup of $\group$.  In case the group $\group$ is profinite, $H''_{\group}$ will have finite index in $\group$.
 The obvious quotient map 
$\group/H''_{\group} \ra \group/H$ induces a map
 $M(\group/H) \ra M(\group/H''_{\group})$ so that $\{M(\group/K)|K \mbox{ normal in } \group \mbox{ and } \eps \W\}$
is cofinal in the direct system used in the above colimit. Now $\group/K$ acts on $\group/K$ by translation
and this induces an action by $\group/K$ on $M(\group/K)$. Therefore, the above colimit has a natural 
action by $\group$ which is clearly continuous.
The  natural transformation $\Theta \circ \Phi \ra id $ may be shown to be an 
 isomorphism readily. Moreover $\Theta$ is left-adjoint to $\Phi$. It follows,
therefore, that the functor $\Phi$ is full and faithful, so that $\Phi$ is an
imbedding of the
category $\PSh_*^{\group,c}$ in ${\rm {PSh}}_*^{{\mathcal
O}_{\group}^o}$. The same holds for the categories $\PSh/S^{\group,c}$ and ${\rm {PSh}}/S^{{\mathcal
O}_{\group}^o}$.
\begin{proposition} 
\label{Phi.Theta.props}
Assume the above situation. 
Then the following hold 
\vskip .3cm
(i) If $P \eps {\rm {PSh}}_*^{{\mathcal O}_{\group}^o}$ ($\PSh/S^{{\mathcal O}_{\group}^o}$) is cofibrant, the natural map $\eta: P \ra \Phi \Theta (P)$ 
is an isomorphism.
\vskip .3cm
(ii) The two functors $\Phi$ and $\Theta$ are
 Quillen-equivalences.
\vskip .3cm
(iii) The categories $\PSh_*^{{\mathcal O}_{\group}^o}$ and $\PSh/S^{{\mathcal O}_{\group}^o}$ are combinatorial and tractable model categories
for both the projective and object-wise model structures on $\PSh_*$.
\vskip .3cm
(iv) The categories  $\PSh_*^{{\mathcal O}_{\group}^o}$,  $\PSh/S^{{\mathcal O}_{\group}^o}$, $\PSh_*^{\group, c}$ and $\PSh/S^{\group, c}$
have sets of small homotopy generating sets that are cofibrant. i.e. Each of the above categories has a set of cofibrant objects $\{C_{\alpha}|\alpha\}$, so that 
every object in the above categories is a homotopy colimit of the $C_{\alpha}$. (See \cite[Definition 1.3]{Bar}.)
\vskip .3cm
(v) Assume $\group$ is profinite. Let $P \eps \PSh_*^{\group}$ ($\PSh/S^{\group}$) be such that
in $\PSh_*$ ($\PSh/S$, \res) it is finitely presented and $P =P^H$ for some normal 
subgroup $H$ of finite index in $\group$. Then $\Phi(P)$ is a finitely presented object in
${\rm {PSh}}_*^{{\mathcal O}_{\group}^o}$ ($\PSh/S^{{\mathcal O}_{\group}^o}$, \res).
\end{proposition}
\begin{proof} We will explicitly consider only the case of $\PSh_*^{\group,c}$ and $\PSh_*^{{\mathcal O}_{\group}^o}$, since the
case of $\PSh/S^{\group,c}$ and $\PSh/S^{{\mathcal O}_{\group}^o}$ is similar.  
A key observation is that the the functor $\Theta$ sends the generating
cofibrations, i.e. diagrams of the form $(\group/H)_+ \wedge i$ to $(\group/H)_+ \wedge i \eps \oI_{\group}$
and similarly sends the generating trivial cofibrations, i.e. diagrams of the form $(\group/H)_+  \wedge j$
 to $(\group/H)_+ \wedge   j \eps \oJ_{\group}$. $\Phi$, on the other hand, sends the generating cofibrations
in $\oI_{\group}$ of the form $(\group/H)_+ \wedge i$ (generating trivial cofibrations in $\oJ_{\group}$ of the form $(\group/H)_+ \wedge j$) 
to the corresponding generating cofibration $(\group/H)_+ \wedge i$ in 
$\oI_{{\mathcal O}^{{\group}^o} }$ (generating trivial cofibration $(\group/H)_+ \wedge j$ in $\oJ_{{\mathcal O}^{{\group}^o} }$, \res). Since any cofibrant 
object
 in the model category $\PSh_*^{{\mathcal O}_{\group}^o}$ is a retract of an 
$\oI_{{\mathcal O}_{\group}^o}$-cell, and both $\Theta$ and $\Phi$ preserve retractions, it suffices to
prove (i) when $P$ is an $\oI_{{\mathcal O}_{\group}^o}$-cell. 
Recall that  $\Theta$ obviously preserves colimits and $\Phi$ also preserves filtered colimits of sub-simplicial presheaves of 
a simplicial presheaf as proven in Lemma
~\ref{key.props}(iii). Moreover $\Theta$ obviously preserves pushouts while $\Phi$ also preserves pushouts along the generating cofibrations as 
observed in ~\eqref{pushouts.and.colimits.fixed.points}. Therefore, it suffices to consider the case when $P= (\group/H)_+ \wedge B$, 
where $i:A \ra B$ is generating cofibration in $\PSh_*$. 
This has  already been observed to be true, thereby proving  the first statement. 
\vskip .3cm
Observe that it suffices to prove the following in order to establish the second statement.
Let $X \eps \PSh^{{\mathcal O}_{\group}^o}$ be cofibrant and let $Y \eps \PSh^{\group,c}$
be fibrant. Then a morphism $f: \Theta (X) \ra Y$ in $\PSh^{ \group,c}$ is a 
weak-equivalence if 
and only if the corresponding map $g: X \ra \Phi(Y)$ in $\PSh^{{\mathcal O}_{\group}^o}$ 
is a weak-equivalence. Now the induced map $g: X \ra \Phi(Y)$ induced by $f$ by adjunction, factors
 as the composition $X {\overset {\eta} \ra} \Phi(\Theta (X)) {\overset {\Phi(f)} \ra } \Phi(Y)$. 
The first map is
an isomorphism since $X$ is cofibrant, so that the map $g:X \ra \Phi(Y)$ is a weak-equivalence if and
only if the map $\Phi(f)$ is a weak-equivalence. But the map $\Phi(f)$ is a weak-equivalence if and only if
each of the maps $f^H: X^H=(\Theta(X))^H \ra Y^H$ is a weak-equivalence, which is equivalent to 
$f$ being a weak-equivalence. This proves (ii). 
\vskip .3cm
(iii) \cite[Proposition A.2.8.2]{Lur} proves that all the model categories here are combinatorial.  The sources of the generating cofibrations and
 trivial cofibrations in $\PSh_*$ are cofibrant for both the projective and object-wise model structures: therefore, the same hold for the 
sources of the generating cofibrations and generating trivial cofibrations in the above diagram categories. These observations prove that the
model structures are tractable.
\vskip .3cm
(iv) Since $\PSh_*^{{\mathcal O}_{\group}^o}$ identifies with the category of pointed simplicial presheaves on the category ${{\mathcal O}_{\group}} \times \bS/S$,
and every object of the form $(\group/H)_+ \wedge h_U$, $U \eps \bS/S$, is cofibrant, we readily see that the statement (iv) is true for
$\PSh^{{\mathcal O}_{\group}^o}$. Since $\Theta \circ \Phi=id$, it follows that any object in $\PSh^{\group,c}$ is the image
of some object of $\PSh^{{\mathcal O}_{\group}^o}$. Since $\Theta$ commutes with colimits and cofibrations, and preserves
simplicial objects with an extra degeneracy, it follows any $F \eps \PSh^{ \group,c}$ is a homotopy colimit of cofibrant objects
of the form $(\group/H )_+\wedge h_U$, $U \eps \bS/S$. (See \cite[Lemma 2.7, Proposition 2.8]{Dug}.) This proves (iii).
\vskip .3cm
The proof of (v) is similar to the proof of
Proposition ~\ref{finite.pres.1} and is therefore skipped. (One needs to first observe that $P=P^{H'}$ for all 
subgroups $H'$ of $\group$ for which $H' \subseteq H$.)
\end{proof}
\vskip .3cm
\begin{remarks} (i) One may observe that the functor $\Phi$ evidently commutes with all small limits
while the functor $\Theta$ evidently commutes with all small colimits. Lemma ~\ref{key.props} (v) 
shows that, when $\group$ is profinite, the functor $\Phi$ also commutes with all small colimits.
\vskip .3cm
(ii) The above proposition proves that, instead of $\PSh_*^{ \group, c}$ ($\PSh/S^{\group, c}$), it suffices to
consider the diagram category $\PSh_*^{{\mathcal O}_{\group}^o}$ ($\PSh/S^{{\mathcal O}_{\group}^o}$, \res) which is often easier to handle.
\end{remarks}

\subsubsection{\bf Object-wise  model structures on $\PSh_*^{{\mathcal O}_{\group}^o} $ and on $\PSh/S^{{\mathcal O}_{\group}^o} $}
\label{objectwise.diagram.model.structs}
Even if we started with an object-wise model structure on $\PSh_*$ or $\PSh/S$, the model structures we produced so far on
$\PSh_*^{{\mathcal O}_{\group}^o} $ and on $\PSh/S^{{\mathcal O}_{\group}^o} $ have been projective model structures derived from
the original model structure on $\PSh_*$ and $\PSh/S$. However, since these are diagram categories one can provide a different {\it object-wise}
model structure on these categories which we proceed to discuss briefly and which will be also used in the construction of model
structures on the category of spectra. We will see in the next section  that this structure has certain advantages as far as constructing a category of
spectra with reasonable properties.
\vskip .3cm
We start with the object-wise model structures on $\PSh_*$ and $\PSh/S$. 
For the corresponding  model structures on $\PSh_*^{{\mathcal O}_{\group}^o} $ and on $\PSh/S^{{\mathcal O}_{\group}^o}$, a map $\{f_{\group/H}: X(\group/H) \ra Y(\group/H)|H \eps \W\}$ will be called a cofibration (weak-equivalence)
if each $f_{\group/H}$ is a cofibration (weak-equivalence) in $\PSh_*$ (or $\PSh/S$) and fibrations are defined by the right-lifting property
with respect to trivial cofibrations. Now \cite[Theorem 1.19]{Bar} and \cite[Proposition A.2.8.2]{Lur} show that the above structure defines
a combinatorial model structure on $\PSh_*^{{\mathcal O}_{\group}^o} $ and on $\PSh/S^{{\mathcal O}_{\group}^o} $. Clearly the domains
and codomains of the generating cofibrations and generating trivial cofibrations are all cofibrant, so that these are in fact
{\it tractable model structures}. 
\vskip .3cm
Given two diagrams $\{P(\group/H)|H \eps \W\}$ and $\{Q(\group/H)|H \eps \W\}$, one defines the smash-product
\be \begin{equation}
\label{smash.diagrams}
     (P \wedge Q)(\group/H) = P(\group/H) \wedge Q(\group/H).
    \end{equation} \ee
\vskip .3cm \noindent
With this tensor structure, $\PSh_*^{{\mathcal O}_{\group}^o} $ and on $\PSh/S^{{\mathcal O}_{\group}^o}$ are symmetric monoidal categories.
The pushout-product axiom is readily verified since it holds in $\PSh_*$ and $\PSh/S$. Moreover every object is cofibrant in this
model structure so that both $\PSh_*^{{\mathcal O}_{\group}^o} $ and on $\PSh/S^{{\mathcal O}_{\group}^o}$ are symmetric monoidal
model categories satisfying the monoidal axiom.
\begin{proposition} 
\label{diagram.cats.excellent}
With the above structure, both $\PSh_*^{{\mathcal O}_{\group}^o}$ and $\PSh/S^{{\mathcal O}_{\group}^o}$ are
 {\it excellent model categories} in the sense of \cite[Definition A.3.2.16]{Lur}.
\end{proposition}
\begin{proof} The proof consists in showing that these categories satisfy the axioms (A1) through (A5) in \cite[Definition A.3.2.16]{Lur}.
We already observed that  these categories are combinatorial, which verifies the axiom (A1).  Since every monomorphism in $\PSh_*$ and
$\PSh/S$ is a cofibration and the cofibrations in the above diagram categories are defined object-wise, it follows that very monomorphism 
is a cofibration. In fact cofibrations in both the categories identify with monomorphisms. Therefore, cofibrations  in both 
these categories are closed under products. This verifies axiom (A2). 
\vskip .3cm
Next observe that in the category of pointed simplicial sets, the domains and codomains of the generating cofibrations and trivial cofibrations
are finitely presented.  It follows (see \cite[Definition 3.4, Lemma 3.5]{Dund1} that the model category of pointed simplicial sets
is weakly finitely generated and hence that weak-equivalences are closed under filtered colimits. Since weak-equivalences in
$\PSh_*$, $\PSh/S$, $\PSh_*^{{\mathcal O}_{\group}^o}$ and $\PSh/S^{{\mathcal O}_{\group}^o}$ are defined object-wise, it follows
the same holds for the class of weak-equivalences in these categories, thereby proving axiom (A3). We already know that the above diagram categories are 
monoidal model
categories which verifies the axiom (A4). Therefore, it remains to verify the invertibility axiom (A5). 
\vskip .3cm
For this we make use of \cite[Lemma A.3.2.20]{Lur} as follows. First the functor sending
a pointed simplicial set to the associated constant simplicial presheaf in $\PSh_*$ and $\PSh/S$ 
as defined in ~\ref{simpl.sets}
is a monoidal left Quillen functor. So also is the constant-diagram functor considered in ~\eqref{free.diagram}, ${\mathfrak C}: \PSh_* \ra \PSh^{{\mathcal O}_{\group}^o}$
and ${\mathfrak C}:\PSh/S \ra \PSh/S^{{\mathcal O}_{\group}^o}$. (To see that these are left-Quillen functors, observe that their right adjoint
is the functor sending a diagram $\{P(\group/K)|K \eps \W \} \mapsto P(\{e\})$, where $e$ denotes the identity element of $\group$. Since any fibration in
the object-wise model structure on $\PSh_*^{{\mathcal O}_{\group}^o}$ and $\PSh/S^{{\mathcal O}_{\group}^o}$ is a fibration in the projective model structure,
it follows that the above functor preserves fibrations.) The functor ${\mathfrak C}$ is also clearly a monoidal functor as may be seen from the 
definition of the monoidal structure on $\PSh_*^{{\mathcal O}_{\group}^o}$ and $\PSh/S^{{\mathcal O}_{\group}^o}$ defined in ~\eqref{smash.diagrams}.
Therefore, so is their composition and the monoidal category of pointed simplicial sets 
satisfies the axioms (A1) through (A4) in \cite[Definition A.3.2.16]{Lur}. This  proves the proposition.
\end{proof}
\subsection{\bf Localization with respect to an interval and simplicial sheaves up to homotopy}
\label{desc:a1}
Presently we will assume that our site $\bS$ comes provided with the structure of a site with an interval $\bI$
in the sense of \cite[2, 2.3]{MV}.
One may apply localization with respect to $\bI$ to  each of the above model structures
as in \cite{MV}. The resulting localized model categories will be denoted with the subscript $\bI$. 
\vskip .3cm
The fibrant 
objects $F$ in these categories are characterized by the property that they are
fibrant in the underlying  model structure and that $Map(P \times \bI, F)
\simeq Map(P, F)$, for all $P \eps {\rm {PSh}}_*^{\group}$ ($P \eps \PSh/S^{\group}$). Such fibrant objects will
be called {\it $\bI$-local}. 

\vskip .3cm
\begin{definition}
 \label{motivic.fibrance}
Next, there are two means of passing to {\it sheaves or rather sheaves up to homotopy}. When the topology is 
defined by a complete cd-structure (see \cite{Voev-cd}) as in 
the case of the 
Nisnevich topology, one defines a presheaf $P \eps {\rm {PSh}}_*^{\group}$ to be $\bI$-local or
{\it $\bI$-fibrant} if (i) $P$ is fibrant in $\PSh_*^{\group}$, (ii) $\Gamma (\phi, P) $ is contractible (where $\phi$ denotes the empty 
scheme), (iii) sends a {\it
distinguished square} as in \cite{MV} to a homotopy cartesian square and (iv) the obvious pull-back 
$\Gamma (U, P) \ra \Gamma (U \times \bI, P)$ is a weak-equivalence. 
\end{definition}
\begin{example}
 Of course one of of the main examples of the above framework is when $\bS$ is a category of schemes over a given base scheme
provided with a completely decomposed topology, for example, the Nisnevich topology for smooth schemes or the cd-topology
 for general schemes. In this case the interval $\bI$ is the affine line ${\mathbb A}^1$. Again there are finer variants of
the above: for example, $\bS$ could be all smooth schemes of finite type over a given base scheme with the Nisnevich topology,
or $\bS$ could be the subcategory of smooth schemes provided with an action by a smooth group scheme defined over the base scheme
with morphisms being equivariant maps. It could also be cdh-analogue of the above sites.
\end{example}

Then a map $f: A \ra B$ 
 in  $\PSh_*^{\group}$ is an $\bI$-local weak-equivalence if the induced map $Map(f, P)$ is
a weak-equivalence for every $\bI$-local object $P$, with $Map$ denoting the simplicial
mapping space. One then localizes such weak-equivalences.  The resulting model structure will be denoted 
 ${\rm {PSh}}^{\group}_{*,\bI}$. The above definitions also apply to $\PSh/S$ and the resulting model category will be
denoted $\PSh/S_{\bI}^{\group}$.
\vskip .3cm
It may be important to specify the generating trivial cofibrations for the localized category, which we 
proceed to do now: see 
\cite[Definition 2.14]{Dund1}. For any  distinguished square
\[Q= \xymatrix{{P} \ar@<1ex>[r] \ar@<-1ex>[d] & Y \ar@<1ex>[d]^{\phi}\\
U \ar@<1ex>[r]_{\psi} & X}\]
\vskip .3cm \noindent
we factor the induced map $h_P \ra h_Y$ as a cofibration (using the simplicial mapping cylinder)
$h_P \ra C $ followed by a simplicial homotopy equivalence $C \ra h_Y$ and similarly factor the induced
map $sq=h_U{\underset {h_P} \sqcup} C \ra h_X$ as a cofibration $sq \stackrel{q}{\ra} tq$ followed
by a simplicial homotopy equivalence $tq \ra h_X$. Similarly we factor
the obvious map  $h_{U \times \bI} \ra h_U$ into a cofibration $u:h_{U \times \bI} \ra 
C_u$ followed by a simplicial homotopy equivalence $C_u \ra h_U$. 
 Let 
\vskip .3cm
\be \begin{equation} 
     \label{tilde.J}
      \begin{split}
 {\oJ}' =\{(\group/H)_+ \wedge *_+ \ra h_{\phi})|H \} \cup \{(\group/H)_+ \wedge u:(\group/H)_+ \wedge  h_{U \times \bI} \ra (\group/H )_+ \wedge C_u \mid U \eps 
\bS/S\}\\
 \cup \{(\group/H )_+ \wedge q: (\group/H )_+ \wedge sq\ra (\group/H )_+\wedge tq \mid q \mbox { is an elementary distinguished square } \}  \end{split} 
\end{equation} \ee
\vskip .3cm \noindent
Then one adds the set $\oJ'$ to the set of generating trivial cofibrations in $\PSh_*^{\group}$ to obtain a set of generating trivial cofibrations for
the localized model structure. One may perform corresponding localizations on the diagram categories $\PSh_*^{{\mathcal
O}_{\group}^o}$, $\PSh/S^{{\mathcal
O}_{\group}^o}$ as well as  $\PSh_*^{\group,c }$ and $\PSh/S^{\group,c }$ by a very similar process. The resulting categories will be denoted 
$\PSh_{*,{\bI}}^{{\mathcal O}_{\group}^o}$, $\PSh/S_{\bI}^{{\mathcal O}_{\group}^o}$, $\PSh^{\group, c}_{*, \bI}$ and $\PSh/S^{\group, c}_{\bI}$.
\vskip .3cm
When $\group$ is a finite or profinite group,  both the object-wise and projective model structures on $\PSh_*^{\group,c }$ and $\PSh/S^{\group. c}$ 
are 
tractable simplicial model categories (see \cite[Chapter 4]{Hirsch} and \cite{Bar} for  basic results on localization) and localization preserves these
properties. In the same situation, the projective model structures on 
 $\PSh_*^{\group}$ and $\PSh/S^{\group}$ are also cellular model categories and the localization preserves this property. 
\vskip .3cm
An alternate approach that applies in general is to localize by inverting hypercovers as in \cite{DHI}.
Following \cite{DHI}, a simplicial presheaf has the {\it descent property}
for all hypercovers if for $U$ in $(\bS/S)_?$, and all hypercoverings $U_{\bullet} \ra U$,
the induced map $P(U) \ra \holimD \{\Gamma (U_n, P)|n\}$ is a weak-equivalence.
By localizing with respect to
maps of the form $U_{\bullet} \ra U$ where $U_{\bullet}$ is a 
hypercovering of $U$ and also maps of the form $U \times \bI \ra U$, it is proven in
 \cite[Theorem 8.1]{Dug} and \cite[Example A. 10]{DHI} that we obtain
a model category which is Quillen equivalent to the Voevodsky-Morel model category of simplicial 
sheaves on $(\bS/S)_?$ as in \cite{MV}. Though the resulting localized category
is cellular and left-proper (see \cite[Chapters 12 and 13]{Hirsch}) it is unlikely to be weakly finitely generated:
the main issue  is that the hypercoverings, being simplicial objects, need not be small. Nevertheless
this seems to be the only alternative available in the \'etale setting.
This localized category of simplicial presheaves will  be
 denoted by $\PSh_{*,{des}}^{\group, c}$. (In this case one adds 
$\{ (\group/H )_+ \wedge h_{U_{\bullet}} \ra (\group/H )_+\wedge Cyl(h_{U_{\bullet}} \ra h_U)|U\} \cup
\{(\group/H)_+  \wedge Cyl(u):(\group/H)_+ \wedge  h_{U \times \bI} \ra Cyl(u)|u:h_{U \times \bI} \ra U\}$ to the 
generating trivial cofibrations in $\PSh_*^{\group}$ to obtain
a set of generating trivial cofibrations for the localized category.  Here $Cyl$ denotes the obvious mapping cylinder.) 
Similar definitions apply to $\PSh/S$, $\PSh_*^{{\mathcal
O}_{\group}^o}$, $\PSh/S^{{\mathcal
O}_{\group}^o}$. The resulting localized model category corresponding to $\PSh/S$ ($\PSh_*^{{\mathcal
O}_{\group}^o}$, $\PSh/S^{{\mathcal
O}_{\group}^o}$) will be
denoted $\PSh/S_{des}^{\group, c}$ ($\PSh_{*,{des}}^{{\mathcal
O}_{\group}^o}$, $\PSh/S_{des}^{{\mathcal
O}_{\group}^o}$, \res).
\vskip .3cm

\begin{proposition}
 \label{localization.1}
(i) The localized model categories obtained from $\PSh_*^{{\mathcal O}_{\group}^o}$,  $\PSh/S^{{\mathcal O}_{\group}^o}$ are tractable model categories.
The same holds for the localized model categories obtained  from $\PSh_*^{\group, c}$ and $\PSh/S^{\group, c}$ when $\group$ is profinite or finite.
The above localized model categories are left-proper when the original categories are left proper.
\vskip.3cm \noindent
(ii) These are also
symmetric monoidal model categories with $\wedge$ (or $\wedge^S$) as the monoidal structure. 
\vskip .3cm \noindent
(iii) The localized diagram categories $\PSh_*^{{\mathcal O}_{\group}^o}$ and $\PSh/S^{{\mathcal O}_{\group}^o}$ with the object-wise model 
structures (and when $\PSh_*$, $\PSh/S$ are provided with the object-wise model structures)
are excellent model
categories.
\end{proposition}
\begin{proof} The first statement follows from the observation that the original categories before localization are all tractable model categories
and localization preserves these properties. (See \cite[Theorem 2.15]{Bar}.) It also preserves left-properness. 
Proposition ~\ref{Phi.Theta.props}(iv) shows
all the categories have sets of homotopy generators that are cofibrant. (Recall that objects of the form $(G/H)_+  \wedge (\Delta[n] \rtimes h_U)$ are
cofibrant in both the object-wise and projective model structures on $\PSh_*^{\group, c}$.) Therefore, \cite[Proposition 3.19]{Bar} applies
to complete the proof of the second statement. (We are invoking \cite[Proposition 3.19]{Bar} with ${\bf V}$ denoting the category of
pointed simplicial sets, so that the ${\bf V}$-enriched hom is simply the usual mapping space $Map$. ${\bf H}$ corresponds to the set
$\{ U \times \bI \ra U | U \}$ so that the ${\bf H}/{\bf V}$-local objects considered in \cite[Proposition 3.19]{Bar} are simply the 
$\bI$-local pointed simplicial presheaves.) These prove both statements (i) and (ii). 
\vskip .3cm
They also show that the diagram
categories in (iii) satisfy the axioms (A1), (A2) and (A4) in \cite[Definition A.3.2.16]{Lur}. 
We proceed to verify that weak-equivalences are stable
by filtered colimits. Let $\{f_{\alpha}: A_{\alpha} \ra B_{\alpha}|\alpha \eps I\}$ denote filtered direct system of maps in one of the above model categories.
By considering the model category of $I$-diagrams provided with the projective model structure, one may readily see that the functor $\colim_{\alpha}$ is 
a left Quillen functor and therefore preserves trivial cofibrations. Now one may functorially factor $f=\{f_{\alpha}| \alpha\}$ into the composition of
a trivial cofibration $i=\{i_{\alpha}|\alpha \}$ and a fibration $p = \{p_{\alpha}|\alpha\}$. Since $f$ is a weak-equivalence, it follows that
so is $p$. Therefore, it suffices to show that $\colim_{\alpha} p_{\alpha}$ is a trivial fibration in the given model category. i.e. We reduce
to showing that if each $f_{\alpha}$ is a trivial fibration, then so is $\colim_{\alpha} f_{\alpha}$. But the trivial fibrations in the model category
obtained by left-Bousfield localization from another model category are the same as in the model category before localization: see 
\cite[Proposition 3.3.3 (1)(b)]{Hirsch}.
The corresponding property clearly holds in $\PSh_*^{{\mathcal O}_{\group}^o}$ and in $\PSh/S^{{\mathcal O}_{\group}^o}$ since the same property holds
in the model category of pointed simplicial sets (which is weakly finitely generated.) These verify that the axiom (A3) holds in the localized model categories
considered above.
\vskip .3cm
Next we proceed to verify that the axiom (A5) in \cite[Definition A.3.2.16]{Lur} also holds.
For this, we make use of Proposition ~\ref{diagram.cats.excellent} which showed that the diagram categories
$\PSh_*^{{\mathcal O}_{\group}^{o}}$ and $\PSh/S^{{\mathcal O}_{\group}^{o}}$ are excellent model categories.
One can readily verify that the identity functor that sends the last model  categories to the corresponding $\bI$-localized model 
categories are also
 left Quillen monoidal functors. Therefore, these verify the hypotheses of 
\cite[Lemma A.3.2.20]{Lur} proving that the localized diagram 
categories in (iii) satisfy axiom (A5)
in \cite[Definition A.3.2.16]{Lur} thereby also completing the proof of the proposition.
\end{proof}
\begin{remark} Localization with respect to an interval is better discussed in detail in the
stable setting, which we consider
below. Moreover, since the etale homotopy type of affine spaces is trivial after
completion away from the 
 characteristic of the base field, ${\mathbb A}^1$-localization in the \'etale
setting is often simpler than the
corresponding motivic version as shown below. 
\end{remark}
\vskip .3cm \noindent
\subsubsection{\bf Equivariant topologies}
\label{equiv.covers}
So far we let the category $\bS$ be fairly general and the action of $\group$ was only on the simplicial presheaves on the category $\bS$.
However, as the example of equivariant G-theory, K-theory and equivariant cycle theories show there are important simplicial 
presheaves that are defined only on objects provided with a group-action, i.e. one needs to restrict to objects in $\bS$
 provided with actions by the given group.  We proceed to consider various sites that arise in 
this equivariant context.
\begin{definition}(Equivariant topologies) 
\label{equiv.tops}
Let $\bS$ denote a site as before and let $\group$ denote a presheaf of groups
 defined on $\bS$. Then we let $\bS^{ \group}$ denote the subcategory of objects 
$Y \eps \bS$ provided with an action by $\group$, with
morphisms being $\group$-equivariant maps in $\bS$. (i.e. Here we are regarding $Y$ as the presheaf represented by the object $Y$ and assuming
that the presheaf $Y$ has an action by the presheaf $\group$.) 
\vskip .3cm
When  $\bS$ is provided with a Grothendieck topology $?$, one can define a Grothendieck topology on $\bS^{\group}$ as follows. 
Coverings $\{U_i \ra X|i\}$ of a given object $X \eps \bS^{\group}$ will be a set of morphisms $\{U_i \ra X|i\}$ in $\bS^{ \group}$ so that 
(i) it is a covering
in the given topology $?$ on $\bS$ and (ii) it satisfies possibly other conditions so that such coverings define a Grothendieck topology
on $\bS^{\group}$. These will be referred to as equivariant sites.
\end{definition}
\begin{examples} One obvious example of  an equivariant site is to simply let the coverings in $\bS^{\group}$ be
families of morphisms $\{U_i \ra X|i \eps I\}$ in $\bS^{\group}$ so that $\{U_i \ra X|i \eps I\}$ is a covering in
the given topology on $\bS$. However, very often one has to add additional hypotheses so that the coverings in
$\bS^{ \group}$ form a smaller family, in general. 
\vskip .3cm
For example, $\bS$ could be a category schemes of finite type over a given base-scheme $\B$ and $\group$ is defined by an affine group scheme
over $\B$. Now $\bS$ could be the Nisnevich or \'etale sites of smooth schemes over $\B$. Then one could let $\bS^{\group}$ denote
the isovariant Nisnevich or \'etale sites: the isovariant \'etale sites were considered in \cite{T2} and extended to stacks in \cite{J03}, while
the isovariant Nisnevich sites were considered in \cite{Serp}. One can also consider the equivariant Nisnevich sites considered in
\cite{KO} (or the $H$-Nisnevich topology considered in \cite{Her}) which have more coverings, in general, than the isovariant Nisnevich site.
 \end{examples}
Since we assume that $\bS^{\group}_{?}$ is a Grothendieck topology, hypercoverings of objects in this site may be defined as 
usual. The category of hypercoverings in $\bS^{\group}_{?}$ for a given $X$ will be denoted $HR(X, \group)$: since this category is not filtered
in general, we take its associated homotopy category, which will be denoted ${\overline {HR}}(X, \group)$. We now obtain the
following results, which follow from standard arguments: see \cite[Chapter 24, section 9]{StPr}, for example.
\begin{proposition} (Cohomology from equivariant hypercoverings)
 \label{coh.equiv.hypercover}
Let $P \eps \AbPsh(\bS^{ \group}_?)$ be an additive presheaf (i.e. one that takes disjoint unions of objects in $\bS$ to products) and let $aP$ denote the associated abelian sheaf. 
Given any $X \eps \bS^{\group}_?$, we obtain the isomorphism
\[H^*_?(X, aP) \cong {\underset  {U_{\bullet} \eps {\overline {HRR}}(X, G)} \colim } H^*(\Gamma(U_{\bullet}, aP)) 
\cong
{\underset  {U_{\bullet} \eps {\overline {HRR}}(X, G)} \colim } H^*(\Gamma(U_{\bullet}, P))\]
\vskip .3cm \noindent
that is functorial in $P$.
\end{proposition}
\begin{corollary} 
\label{coh.via.HRR.1}
Let $P \eps \AbPsh(\bS^{\group}_?)$ be an additive presheaf and let $X, Y \eps \bS^{\group}$ be fixed
  objects. Then one also obtains an isomorphism
\[Ext^*_G(Z(X_+ \wedge S^s \wedge Y), aP) \cong 
{\underset  {U_{\bullet} \eps {\overline {HRR}}(X, G)} \colim } H^*(Hom_G(Z(U_{\bullet, +} \wedge S^s \wedge Y), P)),\]
where $Hom_G$ denotes Hom in the category $\AbPsh(\bS/^{\group}, ?)$ and $Ext_G$ denotes the derived
functor of the corresponding $Hom_G$ for $G$-equivariant abelian sheaves.
 \end{corollary}
\begin{proof} One defines a new abelian presheaf $\bar P$  by $\Gamma (U, \bar P) = Hom(Z(U_+ \wedge S^s \wedge Y), P)=
 Hom(Z(U_+) \otimes Z(S^s) \otimes Z(Y), P)$. (The identification
$Z(U_+ \wedge S^s \wedge Y) = Z(U_+) \otimes Z(S^s) \otimes Z(Y)$ is clear, see \cite[p. 6]{Wei} for example.)
Then one may observe that the functor $P \ra \bar P$ is
exact, sends an additive abelian presheaf to an additive abelian presheaf and commutes with the functor $a$. Therefore, the conclusion
 follows by applying the last proposition to the abelian presheaf $\bar P$.
\end{proof}
\begin{remark}
As shown in \cite[Theorem 1.3]{DHI}, one way to consider simplicial presheaves that are {\it simplicial 
sheaves up to homotopy} is to consider a localization of the category of simplicial presheaves by
inverting hypercoverings. Recall a simplicial presheaf $P$ on a site has {\it descent} if for every object $X$ in the site
and every hypercovering $U_{\bullet} \ra X$ in the site, the obvious augmentation 
\[P(X) \ra \holim\{P(U_n)|n\}\]
is a weak-equivalence. 
\end{remark}
\subsection{\bf Comparison of descent properties of simplicial presheaves}
Assume the equivariant framework considered in Definition ~\ref{equiv.tops}. One clearly has a map of sites $\epsilon: \bS \ra \bS^{\group}$
where the corresponding underlying functor sends a covering in $\bS^{\group}$ to the same object, but viewed as a covering in $\bS$. 
This defines a pushforward: 
\[\epsilon_*: \PSh^{\group}(\bS_?) \ra \PSh^{\group}(\bS^{\group}).\]
\begin{proposition} If $P \eps \PSh^{\group}(\bS?)$ is objectwise fibrant and  has descent on the site $\bS_?$, then $\epsilon_*(P)$ has descent on the site
 $\bS^{\group}_?$. 
\end{proposition}
\begin{proof} This is clear since every hypercovering of any object $X \eps \bS^{\group}_?$ is a hypercovering of the same 
 object $X$ viewed as an object of $\bS_?$. Now one invokes \cite[Theorem 1.3]{DHI}.
\end{proof}
\begin{example} As a typical example of the last proposition, one may let $\bS$ denote the Nisnevich site of all smooth schemes
 over a given base scheme $\B$ and let $\bS^{\group}$ denote the corresponding equivariant Nisnevich site, or the isovariant Nisnevich
site. Then, the last proposition shows that any equivariant simplicial presheaf on the Nisnevich site has descent on restriction to
the last two sites. Nevertheless, this example does not apply to important equivariant simplicial presheaves like those defining
equivariant G-theory or K-theory which are only defined on the equivariant or isovariant site.
\end{example}
\subsection{Sheaves with transfer and correspondences}
 Next let $\group$ denote a profinite group. We let $(\bS, \group)$ denote category whose
objects are smooth schemes of finite type over $S$ provided with an action by some finite quotient of $\group$. The morphisms in this
category will be $\group$-equivariant maps. For each topology $?$, we let the coverings of a scheme $X \eps (\bS, \group)$ be given by
surjective $\group$-equivariant maps $U \ra X$ in $(\bS, \group)$ so that $U \ra X$ is in the given topology on forgetting the $\group$-action. Given two smooth schemes
$X$, $Y$ in $(\bS, \group)$, let
\be \begin{equation*}
Cor_{\group}(X, Y) = \{Z \subseteq X \times Y\mid \mbox{closed, integral so that
the projection $Z \ra X$ is finite}\}. 
\end{equation*} \ee
Observe that since $X$ and $Y$ are assumed to have actions through some  finite quotient of 
$\group$, there is natural induced action of $\group$ on $Cor_{\group}(X, Y)$: hence the
presence of the subscript $\group$. One may also define the category of $\group$-equivariant 
correspondences on $\B$, by letting the objects be the smooth schemes
 in $({\rm {Sm_{\B}}, \group})$ and where morphisms are elements of $Cor_{\group}(X,
Y)$. This category will be denoted ${\mathbf {Cor}}_{\group}$.  An abelian
$\group$-equivariant presheaf with transfers is a contravariant functor 
${\mathbf {Cor}}_{\group} \ra (abelian \quad groups, \group)$, where the category on the
right consists of abelian groups
with a continuous action by $\group$, continuous in the sense that the stabilizers
are all subgroups with
finite quotients.
\vskip .3cm
An {\it abelian  $\group$-equivariant  sheaf with transfers} is an abelian
$\group$-equivariant presheaf with transfers which is a sheaf. This is a sheaf on the site $({\rm
{Sm_{\B}}}, \group)_{Nis}$ with the extra property of having transfer maps for finite $\group$-equivariant correspondences. 
This category will be denoted by 
${\rm {\bf ASh}}_{tr}(\B, \group)$.
\vskip .3cm
One obtains an imbedding of the category $(\bS, \group)$ into ${\mathbf 
{Cor}}_{\group}$ by sending a scheme to itself
and a $\group$-equivariant map of schemes $f: X \ra Y$ to its graph $\Gamma _f$. Given a scheme $X$
in $(\bS, \group)$, ${\mathbb Z}_{tr, \group}(X)$ will denote the sheaf with
transfers defined by $\Gamma (U, {\mathbb Z}_{tr,\group}(X)) = Cor_{\group}(U, X)$.
One extends this to define the {\it $\group$-equivariant motive of $X$} as the
complex associated to the simplicial abelian sheaf $n \mapsto Cor_{\group}( -
\times \Delta [n], X)$, where $-$ takes any object in the big Nisnevich site on $\bS$ as
argument.
We will denote this by ${\mathcal M}_{\group}(X)$.
\vskip .3cm
It is important to realize that ${\mathcal M}_{\group}(X)$ is a complex of {\it
sheaves} on the site $({\rm {Sm_{\B}}}, \group)_{Nis}$. The global sections of this
complex, i.e. sections over $ \B$ will be denoted $M_{\group}(X)$.
\vskip .3cm 
\subsubsection{\bf A key property}
\label{univ.sh.tr.0}
Let $F$ denote an abelian $\group$-equivariant sheaf with transfers. Then 
\newline \noindent
$\Gamma (X, F) \cong Hom_{{\rm {\bf ASh}}_{tr}(\S, \group)}(\M_{\group}(X), F)$. 
\vskip .3cm
\subsection{\bf The motivic framework}
So far we tried to present the unstable theory in as broad a setting as possible. Now we point out how to specialize this to the
motivic framework. The category $\bS$ and the topology $?$ on it could be any one of the following:
\vskip .3cm
a) $\bS$ denotes the category of all smooth schemes of finite type over a fixed base scheme $\B$, with the topology $?$ denoting
either the Zariski, Nisnevich, or \'etale topologies. The interval $\bI$ will be the affine line ${\mathbb A}^1$ over the base-scheme.
The resulting category of pointed simplicial presheaves will be denoted $\PSh_*^{\group}(\Sm_?)_{{\mathbb A}^1}$, where $\group$ denotes
a presheaf of groups on $\Sm$.
\vskip .3cm
$a)_{eq}$ Assume that $\group$ denotes a presheaf of groups on $\bS$ as above. Then $\bS$ may be replaced by $\bS^{\group}$ provided
with a corresponding equivariant topology as in Definition ~\ref{equiv.tops}. The interval $\bI$ will again be ${\mathbb A}^1$.
The resulting category of pointed simplicial presheaves will be denoted $\PSh_*^{\group}(\Sm_?^{\group})_{{\mathbb A}^1}$. 
\vskip .3cm
b) The above simplicial presheaves are pointed in the usual sense. However, if $S$ denotes a fixed object of $\bS$ with trivial action by $\group$
, one may also consider
the categories of simplicial presheaves that are pointed by $S$: these will be denoted 
$\PSh^{\group}/S(\Sm_?^{\group})_{{\mathbb A}^1}$ and $\PSh^{\group}/S(\Sm_?^{\group})_{{\mathbb A}^1}$.
\vskip .3cm
c) One may replace the category $\Sm$ with the category of all schemes of finite type over the base scheme: in this case
one would also replace the topology by the {\it cdh} topology. The corresponding categories of simplicial presheaves will be denoted
$\PSh^{\group}(\Sch_{cdh})_{*,{\mathbb A}^1}$, $\PSh_*^{\group}(\Sch_{cdh}^{\group})_{*,{\mathbb A}^1}$,
$\PSh^{\group}/S(\Sch_{cdh})_{{\mathbb A}^1}$ and $\PSh^{\group}/S(\Sch_{cdh}^{\group})_{{\mathbb A}^1}$.
\vskip .3cm
When there is no chance for confusion, any of these categories of simplicial presheaves pointed by $*$ (by $S$) will be denoted
$\PSh^{\group}_{*,mot}$ ($\PSh^{\group}/S_{mot}$, \res).
\subsection{The coarse model structure on $\PSh_{*, mot}^{\group}$.}
\label{coarse.model}
We end this section with a brief discussion of a {\it coarse model structure} on $\PSh_{*,mot}^{\group}$. Let $U: \PSh_{*, mot}^{\group} \ra \PSh_{*,mot}$
denote the functor forgetting the group action. Then we  define the coarse model structure on $\PSh_{*, mot}^{\group}$ by letting
 the fibrations (weak-equivalences) in the {\it coarse model structure} be those maps
$f: A \ra B$ so that $U(f)$ is a fibration (weak-equivalence, \res) in $\PSh_{*, mot}$. The cofibrations will be defined by the lifting property 
with respect to trivial fibrations. The underlying functor $U$ has a left-adjoint, namely the free functor ${\mathcal F}: \PSh_{*, mot} \ra \PSh_{*, mot}^{\group}$
that sends a pointed simplicial presheaf $P$ to $ U \ra \group(U) \wedge P(U)$, where $\group(U)$ is pointed by the identity element of $\group$. The functors
$U$ and ${\mathcal F}$ define a triple which provides a cofibrant replacement of any simplicial presheaf $P$. (In fact one may identify this
with $E\group \wedge P$.)
 \vfill \eject
\section{\bf The basic framework of equivariant  motivic homotopy
theory: the stable theory}
The theory below is a variation of the theory of enriched functors and spectra as in \cite{Dund1}, 
\cite{Dund2} and also \cite{Hov-3}  modified so as to handle equivariant spectra for the action of a group $\group$ where
 $\group$ denotes a group as before.
\subsection{\bf Enriched functors and spectra}
\label{en_fun} Let $\C$ denote a symmetric monoidal cofibrantly generated model category: 
recall this means $\C$ is a cofibrantly generated model category, which also has the additional
structure of a symmetric monoidal category,  compatible with the model structure as in 
\cite{Hov-2}. We will assume that $\C$ is pointed, i.e. the initial object
 identifies with the terminal object, which will be denoted $*$ in general. The monoidal product will be 
denoted $\wedge$;
the unit of the monoidal structure will be denoted $S^0$ and this is assumed to be 
cofibrant. 
\begin{remark}
In the relative case, where for example, $\C= \PSh/S^{\group, c}$, $*=S$ and $S^0=S_+ = S \sqcup S$.
\end{remark}
We will let $\oI$ ($\oJ$) denote the generating cofibrations (generating trivial 
cofibrations, \res). We will make the following additional assumptions almost always. (The only exception to this is
when we discuss symmetric spectra (see Examples ~\ref{main.eg}(iv)), where the hypothesis (ii) will not be required.):
\be \begin{enumerate}[ \rm (i)]
     \item{The model structure on $\C$ is {\it left-proper}. We will assume the model structure 
 is {\it combinatorial (i.e. locally presentable)} (which together with the conditions in (ii) imply it is in fact {\it tractable}). 
We do not always require $\C$ to be  cellular. In practice the cellularity hypothesis means the domains and
co-domains of both $\oI$ and $\oJ$ are small with respect to $\oI$ and that the cofibrations are effective 
monomorphisms.  }
\item{We will further assume that all monomorphisms in $\C$ are cofibrations. It follows that all the objects of $\C$, and hence in particular, 
the domains and co-domains of the maps in $\oI$ are cofibrant. We will also assume that the cofibrations in $\C$ are stable under products.}
\item {We will also assume that the weak-equivalences in $\C$ are stable under filtered colimits.}
\item{Though this assumption is not strictly necessary, we will further assume that 
$\C$ is simplicial or at least pseudo-simplicial so that there is a bi-functor 
$Map: \C^{op} \times \C \ra (pointed. simpl. sets)$.
We will also assume that there exists an internal Hom functor $\Hom_{\C}: \C ^{op} \times \C \ra \C$,
so that for any finitely presented object $C \eps \C$, one obtains:
\be \begin{equation}
     \label{Omega.0}
Map (C, \Hom_{\C}(K, M)) \cong Map(C \wedge K, M)
\end{equation} \ee
(In fact if $Map (K, N)_n = Hom_{\C} (K \otimes \Delta [n], N)$ where $K \otimes \Delta [n]$ denotes
an object in $\C$ defined using a pairing $\otimes: \C \times (simpl. sets) \ra \C$, then it suffices
to assume that $(K \otimes \Delta [n])_k$ is finitely presented for all $k \ge 0$ and that the
analogue of the isomorphism in ~\eqref{Omega.0} holds with $Map$ replaced by $Hom_{\C}$.)}
\item{Let $\C'$ denote a small symmetric monoidal $\C$-enriched sub-category of $\C$, which may not be a full enriched sub-category. (In particular, 
this implies that the objects of $\C'$ are all objects in $\C$, that the monoidal product on $\C'$
 is the restriction of the monoidal product  $\wedge$ on $\C$ and that the unit in $\C'$ is the same
as the unit $S^0$ of $\C$. But for two objects $A, B \eps \C'$, $\Hom_{\C'}(A, B) $ will in general be distinct from $\Hom_{\C}(A, B)$.)
Let $\C_0'$ denote a $\C$-enriched  sub-category of  $\C'$, which may or may not be a full enriched sub-category, but closed under the 
monoidal product $\wedge$ and containing the unit $S^0$. We will further assume that all the objects in $\C_0'$ are {\it finitely presented} in 
$\C$ and are {\it cofibrant} as objects in $\C$.
(In \cite[section 4]{Dund1}, there are a series
of additional hypotheses put on the category $\C'$ to obtain stronger conclusions. These are not needed
for the basic conclusions as in \cite[Theorem 4.2]{Dund1} and therefore, we skip them for the time 
being.) In view of
the motivating example considered below, we will denote the objects of $\C_0'$ as $\{T_V\}$. }
\end{enumerate} \ee
\vskip .3cm
\begin{remark}
 The hypothesis in (ii) that all the objects of $\C$ are cofibrant is used only in ensuring that the internal hom $\Hom_{\C}$ between
any two objects of $\C'_0$ (assuming this is a full subcategory of $\C$) is always cofibrant. However, this means that we are forced to restrict to the diagram categories with the
object-wise model structures as in ~\ref{objectwise.diagram.model.structs} as candidates for the model category $\C$. We hope
to consider at a future date constructions of equivariant motivic stable homotopy using some of the other unstable model categories 
considered there.
\vskip .3cm
When hypothesis (ii) is not assumed and $\C_0'$ is not assumed to be a full sub-category of $\C$, $\Hom_{\C_0'}$ will be defined 
in such a manner so that $\Hom_{\C_0'}$ between any two objects of $\C_0'$ is cofibrant. See for example, the case of pre-spectra (see 
Definition ~\ref{basic.defs}(ii)) or symmetric spectra (see Examples ~\ref{main.eg}(iv).) (Therefore, when the hypothesis (ii) is not assumed
we may also use the projective model structures on $\PSh_*$ ($\PSh/S$).)
\end{remark}
\vskip .3cm \indent
Observe that $T_W \cong \Hom(S^0, T_W)$ and that $S^0 \ra \Hom_{\C_0'}(T_V, T_V)$. Now the pairing 
$\Hom_{\C_0'}(T_V, T_V) \wedge 
\Hom_{\C_0'}(S^0, T_W ) \ra \Hom_{\C_0'}(T_V, T_V \wedge T_W)$, pre-composed with
$S^0 \wedge \Hom_{\C_0'}(S^0, T_W) \ra \Hom_{\C_0'}(T_V, T_V) \wedge \Hom_{\C_0'}(S^0, T_W)$,
sends $T_W \cong \Hom_{\C_0'}(S^0, T_W)$  to a sub-object of $\Hom_{\C_0'}(T_V, T_V \wedge T_W)$. 
\begin{definitions} 
\label{basic.defs}
(i) $[\C_0', \C]$ will denote the category whose objects are $\C$-{\bf enriched covariant functors
 from $ \C_0'$ to $\C$}: see \cite[2.2]{Dund1}. 
 An  enriched functor ${\mathcal X}$ sends 
$\Hom_{\C_0'}(T_V, T_W) \mapsto \Hom_{\C}({\mathcal X}(T_V), {\mathcal X}(T_W))$ for objects $T_V$ and 
$T_W \eps \C_0'$, with 
this map being 
functorial in $T_V$ and $T_W$. 
\vskip .3cm
The 
adjunction between $\wedge $ and $\Hom_{\C}$ shows that in this case,
one is provided with a compatible family of pairings $\Hom_{\C_0'}(T_V, T_W) \wedge {\mathcal X}(T_V) \ra 
{\mathcal X}(T_W)$.
On replacing $T_W$ with $T_V \wedge T_W$ and pre-composing the above pairing with 
$S^0 \wedge \Hom_{\C_0'}(S^0, T_W) \ra \Hom_{\C_0'}(T_V, T_V) \wedge \Hom_{\C_0'}(S^0, T_W) \ra
\Hom_{\C_0'}(T_V, T_V \wedge T_W)$, one sees that, an enriched functor ${\mathcal X}$ comes provided with 
the structure maps 
$T_{W} \wedge {\mathcal X}(T_V) \ra {\mathcal X}(T_V \wedge T_W)$ that are compatible as $T_W$ and $T_V$ vary in
$\C_0'$.) 
\vskip .3cm
A {\it morphism} $\phi: {\mathcal X}' \ra {\mathcal X}$ between two such enriched functors  is given by a 
$\C$-natural transformation, which means that one is provided with a 
compatible collection of maps $\{{\mathcal X}'(T_V) \ra {\mathcal X}(T_V) \mid T_V \eps \C_0'\}$,
compatible with the pairings $\Hom_{\C_0'}(T_V, T_W) \wedge {\mathcal X}(T_V) \ra X(T_W)$ and 
$\Hom_{\C_0'}(T_V, T_W) \wedge {\mathcal X}'(T_V) \ra {\mathcal X}'(T_W)$. We call such enriched functors {\bf spectra} (or
{\bf spectra} with values in $\C$). This category will be denoted {\bf $Spectra (\C)$} or {\bf $Spectra(\C_0', \C)$}
 if one wants to emphasize
the role of the subcategory $\C'_0$. The obvious {\it inclusion} functor $\C_0' \ra \C$ will be denoted $\Sigma$ for reasons
that will become clear once the smash product of spectra is defined.
\vskip .3cm
(ii) ${\bf Presp (\C)}$. We let ${\rm Sph}(\C_0')$ denote the $\C$-category defined by taking the 
objects to be the same as the objects of $\C_0'$ and where $\Hom_{Sph(\C_0')}(T_U, T_V) = T_W$ if 
$T_V= T_W \wedge T_U$ and $*$ otherwise. Since $T_W $ maps naturally to  $\Hom_{\C_0'}(T_U, T_V)$, if $T_V=T_W \wedge T_U$, it 
follows that there is an obvious functor from the category ${\rm Sph}(\C_0')$ to the category $\C_0'$. Now an enriched functor in 
$[{\rm Sph}(\C_0'), \C]$
is simply given by a collection $\{X(T_V)|T_V \eps {\rm Sph}(\C_0')\}$ provided with a compatible collection of 
maps $T_W \wedge X(T_V) \ra X(T_W \wedge T_V)$. We let $Presp (\C) =[{\rm Sph}(\C_0'), \C]$. i.e. Such enriched
functors will be called {\bf pre-spectra}. {\it Caution}: the terminology here and in (i) are non-standard. For us, it is
the category of spectra that will be important and the category of pre-spectra merely plays a background role.
\vskip .3cm
(iii) For each fixed $T_V \eps \C_0'$, we associate the (enriched) {\it free $\C$-functor} ${\mathcal F}_{T_V}: \C \ra [\C_0', \C]$
 defined by ${\mathcal F}_{T_V}({ X}) = { X} \wedge \Hom_{\C_0'}(T_V, \quad )$ and the 
free $\C$-functor $F_{T_V}: \C \ra Presp (\C)$
that sends each $X \eps \C$ to the pre-spectrum $F_{T_V}(X)$ defined by 
\be \begin{align}
\label{sus.spectrum}
F_{T_V}(X)(T_W) &= X \wedge T_U, \mbox{ if }  {T_W=T_U \wedge T_V} \mbox{ and }\\
&=* \mbox{ otherwise} \notag \end{align} \ee
\vskip .3cm
(iv) For each $T_V \eps \C_0'$, we let $\Omega_{T_V}: \C \ra \C$ denote the functor defined
by $\Omega_{T_V}(X) = \Hom_{\C}(T_V, X)$.
\vskip .3cm
(v) For each $T_V \eps \C_0'$, we let ${\mathcal R}_{T_V}: \C \ra [\C_0', \C]$ denote  the $\C$-enriched functor
${\mathcal R}_{T_V}(P)$ defined by ${\mathcal R}_{T_V}(P) (T_W) = \Hom_{\C}(\Hom_{\C_0'}(T_W, T_V), P)$. Similarly 
we let  $R_{T_V}:\C \ra Presp (\C)$ denote the functor 
defined by $R_{T_V}(P)(T_W) = \Hom_{\C}(T_{U}, P)$ if $T_U \wedge T_W = T_V$ and $S^0$ otherwise.
\vskip .3cm
Let ${\mathcal {E}}val_{T_V}: Spectra(\C)=[\C_0', \C] \ra \C$ denote the $\C$-enriched functor sending
${\mathcal X} \mapsto {\mathcal X}(T_V)$. Similarly, let $Eval_{T_V}: Presp (\C) =[Sph(\C'_0), \C] \ra \C$ denote
 the $\C$-enriched functor that sends
a pre-spectrum ${\mathcal X} \eps Presp (\C)$ to ${\mathcal X}(T_V)$.
\end{definitions}
\vskip .3cm
\begin{lemma}
 Now ${\mathcal R}_{T_V}$ is right-adjoint to ${\mathcal E}val_{T_V}$ while 
${\mathcal F}_{T_V}$ is left-adjoint to ${\mathcal E}val_{T_V}$. Similarly,
 ${ R}_{T_V}$ is right-adjoint to ${ E}val_{T_V}$ while 
${ F}_{T_V}$ is left-adjoint to ${ E}val_{T_V}$.
\end{lemma}
\begin{example}
 \label{main.eg}
The {\it main examples} of the above are the following. Let $\bS$ denote a category of schemes (with the terminal object denoted $\B$) 
as in section 2. Let $\C$ denote one of the categories
${\rm {PSh}}_*^{{\mathcal O}_{\group}^o}$, ${\rm {PSh}}/S^{{\mathcal O}_{\group}^o}$ or 
the corresponding ${\mathbb A}^1$-localized categories of 
simplicial sheaves up to homotopy, all provided with the object-wise model structures obtained by starting with 
the object-wise model structures on $\PSh_*$ and $\PSh/S$ in all but the last example. In the last example, one may also start with the
projective model structure.
 One has the following choices for the category $\C'=\C_0'$: in all, but the example considered in (iv), $\C'$ will be
assumed to be a full sub-category of $\C$. 
\vskip .3cm
Just for the following discussion, we will adopt the
following convention (mainly to simplify the discussion): while considering the categories $\PSh_*$ and $\PSh_*^{{\mathcal O}_{\group}^o}$, 
we will let $\S= \B$. While considering the categories $\PSh/S$ and $\PSh/S^{{\mathcal O}_{\group}^o}$, $\S$ will denote $S$.
\vskip .3cm
(i) Let $V$ denote an affine space over $\S$. We let $Cone( u)$ denote the simplicial mapping cone  of $u$, where $u:V-0 \ra V$ is the obvious map.
We let $T_V $ denote any finitely presented object in $\PSh_*$ ($\PSh/S$) for which there is a trivial cofibration $Cone(u) \ra T_V$. Observe as
a consequence that $T_V$ is a cofibrant object of $\PSh_*$ ($\PSh/S$).
In case $V$ is zero dimensional, i.e. $V= \S$, then we use the convention that $T_V=\S_{+_{\S}} = \S \sqcup \S$.
We let $(Thm.sps)$ denote the full subcategory of $\PSh_*$ (or $\PSh/S$) whose objects are $\{T_V\}$. Then we let $\C'_0=\C'$ denote
the category $(Thm.sps)^{{\mathcal O}_{\group}^o}$ of diagrams with values in  $(Thm.sps)$ and indexed by
${{\mathcal O}_{\group}^o}$. This is the {\it canonical choice} of $\C'_0=\C'$. An object in the above category may then 
be denoted $\{T_{V}(\group/H)|H \eps \W\}$: we will abbreviate this to just $T_V$. Observe that $\C'$ now contains the 
$0$-dimensional sphere $S^0 =\S_{+_{\S}}$. 
\vskip .3cm
(i)' As an approximation to the above, one may also define the categories $\C'_0=\C'$ as follows.
Let $V$ denote an affine space over $\S$ 
and provided with a linear action by $\group$ so that $V^H = V$ for some $H \eps \W$.
We let $Cone( u)$ denote the simplicial mapping cone  of $u$, where $u:V-0 \ra V$ is the obvious map.
We let $T_V $ denote any finitely presented object in $\PSh_*$ ($\PSh/S$) for which there is a trivial cofibration $Cone(u) \ra T_V$.
In case $V$ is zero dimensional, i.e. $V= \S$, then we use the convention that $T_V=\S_{+_{\S}} = \S \sqcup \S$.
\vskip .3cm
We let $\C'_0=\C'$ denote the collection $\{T_V| V\}$ as one varies over all such (i.e. continuous) representations of $\group$. 
(By sending $\{T_V|V\}$ to $\{T_{V^H}|H \eps \W, V\}$ one obtains an imbedding of the current category into the one considered in (i).)
 Observe that $\C'$ now contains the $0$-dimensional sphere $S^0 =\S_{+_{\S}}$ 
(by taking  $V$ to be the $0$-dimensional vector space) and is closed under smash products.  It is important to observe
that $\Hom_{\C}(T_V, T_W)$ is often larger than the corresponding Hom in the subcategory $Sph(\C_0')$ though the objects
of $\C_0'$ and $Sph(\C_0')$ are the same. The resulting category of spectra will be denoted $Spectra (\C)$. A typical example
is that of symmetric spectra. Here the group $\group$ is trivial, but nevertheless, the sphere $T_V$ is supposed to have
an action by the symmetric group $\Sigma_{dim(V)}$. 
\vskip .3cm
(i)'' A variant of the above framework is the following: let $F: \C \ra \C$ denote a $\C$-enriched functor that 
 is compatible with the 
 monoidal product $\wedge$ (i.e. there is a natural map $F(A) \wedge F(B) \ra F(A \wedge B)$) and with all small colimits. We may now 
let $\C'_0=\C'_F = $ the full subcategory
of $\C$ generated by $\{F(A) |A \eps \C'\}$ together with $S^0$. Examples of such a framework appear elsewhere. The resulting category of spectra will be denoted $Spectra(\C, F)$. 
\vskip .3cm
(i) and (i') seem to be the only general
framework available for $\group$ denoting any one of the three choices above.
\vskip .3cm
(ii) One may also adopt the following alternate definition of $\C'_0=\C'$, when $\group$ denotes either a finite or profinite group.
Let  $\T= {\mathbb P}^1= {\mathbb P}^1_{\S}$. 
Let $K$ denote a  normal subgroup of $\group$ with finite index. 
Let $\T_K$ denote 
the $\wedge$ of $\group/K$ copies of $\T$ 
with $\group/K$ acting by permuting the various factors above. 
We let $\C'$ denote the full sub-category of
$\C$ generated by the objects $\{\T_K| |\group/K| <\infty\} $ under finite applications of $\wedge$ as $K$ is allowed to 
vary subject to the above constraints.  {\it This sub-category of $[\C', \C]$ will be denoted 
$[\C', C]_{\T}$ and the corresponding category of spectra will be denoted 
$Spectra(\C, \T)$.} If we fix the normal subgroup $K$ of $\group$, 
the resulting category of $\T$-spectra will be denoted 
$Spectra (\C,  \group/K, \T)$: here the subcategory $\C'$ will denote the full sub-category of
$\C$ generated by these objects under finite applications of $\wedge$.
\vskip .3cm
More generally, given any object $P \eps \C$ together with an action by  some finite quotient
 group $\group/K$  of $\group$, one may define the categories $[\C', C]_P$ and the category, $Spectra(\C, P)$ of 
{\it $P$-spectra} similarly by replacing
$\T$ above by $P$. (For example, $P$ could be $F(\T)$ for a functor $F$ as 
in (i).)
 In case the base-scheme $\S$ is a field of characteristic $0$, the regular representation of any finite group breaks up into the sum 
of irreducible representations and contains among the summands all  irreducible representations. 
{\it Therefore, in case $\S$ is a field of characteristic  $0$, and the group is finite, there is no loss of generality in adopting this framework to that of
(i). In arbitrary characteristics, and also for general base schemes and general groups, the framework of (i) is clearly more general.}
\vskip .3cm
(iii) One may also let $\T=S^m$ (for some
fixed positive integer $m$) denote the
usual simplicial $m$-sphere. Let $\T_K$ denote the $\wedge$ of $\group/K$ copies of $\T$ 
with $\group/K$ acting by permuting the various factors. We let $\C'=\C'_0$ denote the full sub-category of
$\C$ generated by the objects $\{\T_K| |\group/K| <\infty\} $ under finite applications of $\wedge$ as $K$ is allowed to 
vary subject to the above constraints. In case $K=\group$ and $m=1$, we obtain a spectrum 
in the usual sense and indexed by the non-negative integers. Such spectra will be called 
{\it ordinary spectra}:
when the constituent simplicial presheaves are all simplicial abelian presheaves, such spectra will be
called {\it ordinary abelian group spectra}.
\vskip .3cm
(iv) Let $\T$ denote a fixed object in $\C$. Let $\C'=\C'_0$ denote the symmetric monoidal category
whose objects are iterated powers of $\T$ under the monoidal product in $\C$ along with the unit element of $\C$.
Now one defines the morphisms between two objects $\T^{\wedge m}$ and $\T^{\wedge n}$ exactly as in \cite[2.6]{Dund1}: one may readily see
that this is a $\C$-enriched subcategory of $\C$, but definitely not a full enriched sub-category. The
resulting category of enriched functors $[\C', \C]$ will be equivalent to the category of symmetric spectra
in $\C$ as shown in \cite[Proposition 2.15]{Dund1}. This example will be important for some future applications
when one takes for $\C$ one of the following categories: ${\rm {PSh}}_*^{{\mathcal O}_{\group}^o}$, 
${\rm {PSh}}/S^{{\mathcal O}_{\group}^o}$, $\PSh_*^{\group, c}$, $\PSh/S^{\group, c}$ or 
the corresponding ${\mathbb A}^1$-localized categories of 
simplicial sheaves up to homotopy all provided with one of the  model structures considered earlier. 
Since we no longer have to require that all the objects in the above categories be cofibrant, it is possible to use
the projective model structures on $\PSh_*$ ($\PSh/S$) here.
We do not discuss the details
of this construction in the present paper as it is already discussed amply in \cite{Hov-3} or \cite[2.6]{Dund1}.
\vskip .3cm
A variant of this situation is to replace $\T^{\wedge ^n}$ by $F(\T)^{\wedge ^n}$ where $F: \C \ra \C$ is a monoidal functor as before
and then do the same construction as in the paragraph above.
The weakness of the corresponding stable theory is clear: the spheres which form the suspension coordinates are not indexed
by representations of the group $\group$. In fact, in this setting, the suspension coordinates have no action by $\group$. 
\end{example} 
\begin{definition} {\bf Motivic and \'etale spectra, $\T$-motivic and \'etale spectra}.
\label{mot.spectra}
(i) If $\C= {\rm {PSh}}_{*,{mot}}^{{\mathcal O}_{\group}^o}$  and  $\C_0'$  chosen as 
in Example ~\ref{main.eg}(i), the resulting category of spectra, $Spectra(\C)$, with the stable model structure
 discussed below (see ~\ref{stable.model})  will  be called  {\it motivic spectra} and denoted $\Spt_{{mot}}^{\group}$. 
If instead $\C=  {\rm {PSh}}/S_{mot}^{{\mathcal O}_{\group}^o}$  and  $\C_0'$  chosen as 
in Example ~\ref{main.eg}(i), the resulting category of spectra, will be denoted $\Spt/S_{mot}^{\group}$.
The unit of the corresponding monoidal structure is denoted
$\Sigma_{mot}$. 
\vskip .3cm
If $\C=  {\rm {PSh}}_{*,{des}}^{{\mathcal O}_{\group}^o}$  
with the \'etale topology and  $\C_0'$  chosen as 
in Example ~\ref{main.eg}(i), with the corresponding stable model structure as in ~\ref{stable.model}, the resulting category of spectra, 
$Spectra(\C)$ with 
  the stable model structure  will  be called  {\it \'etale  spectra} and denoted $\Spt_{{et}}^{\group}$. If instead 
$\C= \PSh/S_{des}^{{\mathcal O}_{\group}^o}$, the resulting 
category of spectra will be denoted $\Spt/S_{et}^{\group}$. The unit of the corresponding monoidal structure is denoted
$\Sigma_{et}$. 
\vskip .3cm
In case $F: \C \ra \C$ is a functor
as in Example ~\ref{main.eg} (i)'', $\Spt_{{mot}}^{ \group, F}$ ($\Spt_{{et}}^{\group, F}$) will denote the 
corresponding category, again with the stable model structure. Observe from Example ~\ref{main.eg}(i), that there are several possible choices for the 
sub-categories $\C'_0$. The unit of the corresponding monoidal structure will be denoted $\Sigma_{mot, F}$ 
($\Sigma_{et, F}$, \res.)
\vskip .3cm
(ii)  If $\C= {\rm {PSh}}_{*,{mot}}^{{\mathcal O}_{\group}^o}$ and  $\C= \C_0'$  chosen as 
in Example ~\ref{main.eg}(ii), the resulting category of spectra, $Spectra(\C)$, with the stable model structure
  will  be called  {\it $\T$-motivic spectra} and denoted 
$\Spt_{{mot}}^{\group, \T}$. 
If $\C=  {\rm {PSh}}_{*,des}^{{\mathcal O}_{\group}^o}$ with the \'etale topology and  $\C_0'$  chosen as 
in Example ~\ref{main.eg}(ii), the resulting category of spectra, $Spectra(\C)$ with the stable model structure
  will  be called  {\it $\T$-\'etale  spectra} and denoted 
$\Spt_{{et}}^{\group, \T}$. 
In case $F: \C \ra \C$ is a functor
as in Example ~\ref{main.eg} (i)'', $\Spt_{{mot}}^{\group, F(\T)}$ ($\Spt_{{et}}^{\group, F(\T)}$) 
will denote the 
corresponding category.  Observe again from Example ~\ref{main.eg}(ii), that there are several possible choices 
for the sub-categories $\C'_0$. $\Spt/S_{mot}^{\group, \T}$, $\Spt/S_{et}^{\group, \T}$, $\Spt/S_{mot}^{\group, F(\T)}$, $\Spt/S_{et}^{\group, F(\T)}$
will denote the corresponding categories of spectra which are all pointed by  $S$.
\vskip .3cm
If $S^m$ for some fixed positive integer $m$ is used in the place of $\T$ above, the 
resulting categories 
will be denoted $\Spt_{{mot}}^{\group, S^m}$ and $\Spt_{et}^{\group, S^m}$, etc.
\vskip .3cm
The group $\group$ will
be suppressed
 when we consider spectra with trivial action by $\group$. 
\end{definition}

\begin{definition} 
\label{ho.mot.spec} 
(i) ${\mathbf {HSpt}}_{{mot}}^{\group}$ (${\mathbf {HSpt}}_{{et}}^{\group}$) will denote the 
stable homotopy category of motivic (\'etale) spectra while ${\mathbf {HSpt}}/{S_{mot}}^{\group}$ (${\mathbf {HSpt}}/{S_{et}}^{\group}$) will denote the 
stable homotopy category of motivic (\'etale) spectra pointed by $S$.
\vskip .3cm
(ii) ${\mathbf {HSpt}}_{{mot}}^ {\group, \T}$ 
(${\mathbf {HSpt}}_{{et}}^{\group, \T}$,  ${\mathbf {HSpt}}/{S_{mot}}^ {\group, \T}$ 
 and ${\mathbf {HSpt}}/{S_{et}}^{\group, \T}$ )  will denote the corresponding 
stable homotopy category of  $\T$-motivic (\'etale) spectra.
\vskip .3cm
The  group $\group$ will
be suppressed
 when we consider spectra with trivial action by $\group$. 
\end{definition}

\subsection{\bf Smash products of spectra and ring spectra}
\label{smash.products.def}
First we recall the construction of smash products of enriched functors from \cite[2.3]{Dund1}. 
Given $F', F \eps [\C'_0, \C]$, their {\it smash product} $F' \wedge F$ is defined as the left 
Kan-extension along the 
monoidal product $\wedge: \C'_0 \times \C'_0 \ra \C'_0$ of the $\C$-enriched functor $F' \times F \eps [\C'_0 \times \C_0', \C]$. Given pre-spectra ${\mathcal X}$ and
${\mathcal Y}$ in $[Sph(\C'_0), \C]$, this also defines their smash product ${\mathcal X} \wedge {\mathcal Y}$. 
 One also defines the {\it derived smash product} ${\mathcal X} {\overset L \wedge} {\mathcal Y}$ by  $C({\mathcal X}) \wedge {\mathcal Y}$, where $C({\mathcal X})$ is a cofibrant replacement of ${\mathcal X}$ in the stable model structure on
 $Spectra (\C)$. Now one may observe that $Spectra (\C_0', \C)=[\C_0', \C]$
is itself symmetric monoidal with respect to the smash product and that, when $\C'_0$ is a sub-category of $\C$
 the inclusion functor $\Sigma: \C_0' \ra \C$ is the
unit for the smash product. (See \cite[Theorem 2.6]{Dund1} and also \cite{Day}.)
(For the situation considered in Example ~\ref{main.eg}(iv), it is shown in \cite[p. 424]{Dund1} that there is
an obvious functor $\C' \ra \C$ that is monoidal, which then becomes the unit for the smash product.)
 We call this the {\it sphere spectrum} and denote it as 
$\Sigma$. 
\label{ring.struct}
An {\it algebra} in $[\C'_0, \C]$ is an enriched functor ${\mathcal X}$ provided with an
associative and unital pairing $\mu: {\mathcal X} \wedge {\mathcal X} \ra {\mathcal X}$, i.e.
for $T_V, T_W , T_K\eps \C_0'$, one is given a pairing, 
$\Hom_{\C}(T_V \wedge T_W, T_K) \wedge {\mathcal X}(T_V) \wedge {\mathcal X}(T_W) \ra {\mathcal X}(T_K)$
 which is compatible as $T_W, T_V$ and $T_K$ vary and is also associative and unital. 
\vskip .3cm
A {\it ring spectrum} in $Spectra(\C)$ is an algebra in $[\C_0', \C]$ for some choice of $\C_0'$
satisfying the above hypotheses. A map of ring-spectra is defined as follows.
If ${\mathcal X}$ is an algebra in $[\C_0', \C]$ and ${\mathcal Y}$ is an algebra in $[\D_0', \C]$
 for some choice of $\C$-enriched monoidal categories $\C_0'$ and $\D_0'$ (with both monoidal structures denoted $\wedge$), then a map $\phi: {\mathcal X} \ra {\mathcal Y}$
of ring-spectra is given by the following data: (i) a $\C$- enriched covariant functor $\phi: \C_0' \ra \D_0'$ compatible
with $\wedge$ and (ii) a map of spectra ${\mathcal X} \ra \phi_*({\mathcal Y})$ compatible with the ring-structures.
(Here $\phi_*({\mathcal Y})$ is the  spectrum in $Spectra(\C, \C_0')$ defined by $\phi_*({\mathcal Y}) (T_V) = {\mathcal Y}(\phi(T_V))$.
\vskip .3cm
Given a ring spectrum $A$, a left module spectrum $M$ over $A$ is a spectrum $M$ provided with a pairing
$\mu: A \wedge M \ra M$ which is associative and unital. One defines right module spectra over $A$ similarly.
Given a left (right) module spectrum $M$ ($N$, \res) over $A$, one defines 
\be \begin{equation}
\label{smash.over.A}
M {\underset A \wedge } N = coequalizer (M \wedge A \wedge N {\stackrel \ra \ra} M \wedge N)
\end{equation} \ee
\vskip .3cm \noindent
where the coequalizer is taken in the category of spectra and the two maps correspond to the module structures on $M$ and $N$, \res. 
\vskip .3cm
Let $Mod(A)$ denote the category of left module spectra over $A$ with morphisms being maps of left module spectra over
 $A$.  Then the underlying functor $U: Mod(A) \ra Spectra (\C)$ has a left-adjoint given by the functor $F_A(N) = A \wedge N$. The composition $T=F_A \circ U$ defines a triple and we let $TM = \hocolimD \{T^nM|n\}$. Since a map $f:M' \ra M$ in
$Mod(A)$ is a weak-equivalence of spectra if and only if $U(f)$ is, one may observe readily that $TM$ is weakly-equivalent to $M$. Therefore, one defines
\be \begin{equation}
\label{der.smash.over.A}
M {\overset L { {\underset A \wedge }}} N = TM {\underset A \wedge } N = coequalizer (TM \wedge A \wedge N {\stackrel \ra \ra} TM \wedge N)
\end{equation} \ee
\vskip .3cm \noindent
\begin{examples}
 \begin{enumerate}[\rm (i)]
  \item{ Assume the situation of ~\ref{main.eg} (i). Let $F :\C \ra \C$ denote a functor commuting with
$\wedge$ and with colimits and provided with a natural augmentation $\epsilon : id \ra F$ also
compatible with $\wedge$.  Then $\Sigma _{mot, F}$ is a ring spectrum in $Spectra (\C)$ and
$\epsilon$ induces a map of ring spectra $\Sigma _{mot} \ra \Sigma _{mot, F}$. Whether or not $F$ is 
provided with the augmentation,  $\Sigma_{mot, F}$ is a ring spectrum in $Spectra(\C, F)$: in fact it is
the unit of this symmetric monoidal category. .
\vskip .3cm
Let $\phi: F \ra G$ denote a natural transformation
between two functors $\C \ra \C$ that commute with the tensor structure $\wedge$ and with all 
small colimits. Let $\C'$ denote a subcategory of $\C$ as before and let $\C'_F$, $\C'_G$ denote 
the corresponding subcategories. Then $\phi$ induces a functor $\C'_F \ra \C'_G$ compatible with
$\wedge$. 
Let $\Sigma _{mot, F}$ and $\Sigma _{mot, G}$ denote the motivic sphere spectra defined there. Then
these are ring-spectra and $\phi$ induces a map commuting with the ring structures.}
\item{ Assume the situation of Example ~\ref{main.eg} (ii). Let $\phi: F \ra G$ be as above. 
Then for any $P \eps \C$,
$\Sigma_{F(P)} $ is a ring-spectrum in $Spectra (\C, F(P))$ and $\Sigma _{G(P)}$ is a ring-spectrum
 in $Spectra (\C,G(P))$ with $\phi$ inducing a map of these ring-spectra. }
 \end{enumerate}
\end{examples}
\subsection{\bf Model structures on $Spectra (\C)$ and $Presp(\C)$}
\vskip .3cm
We will begin with the level-wise model structure on $Spectra (\C)$ which will be defined as
follows.  
\begin{definition}
\label{levelwise}
(The level model structure on $Spectra (\C)$ and $Presp(\C)$.)
 \label{level.struct} A map $f: {\mathcal X}' \ra {\mathcal X}$ in $Spectra (\C)$  is a 
{\it level equivalence} ({\it level fibration}, {\it level trivial fibration}, {\it level cofibration},
{\it level trivial cofibration}) if each ${\mathcal E}val_{T_V}(f)$ is a weak-equivalence (fibration, trivial fibration,
cofibration, trivial cofibration, \res) in $\C$. 
 A map $f: {\mathcal X}' \ra {\mathcal X}$ in $Presp(\C)$  is a 
{\it level equivalence} ({\it level fibration}, {\it level trivial fibration}, {\it level cofibration},
{\it level trivial cofibration}) if each $Eval_{T_V}(f)$ is a weak-equivalence (fibration, trivial fibration,
cofibration, trivial cofibration, \res) in $\C$.  
\vskip .3cm
Such a map $f$ is a {\it projective cofibration}
if it has the left-lifting property with respect to every level trivial fibration.
\end{definition}
\vskip .3cm
Let $\oI$ ($\oJ$) denote the generating cofibrations (generating trivial cofibrations, \res ) of
the model category $\C$. 
We define the generating cofibrations $\oI_{Sp}$ ($\oI_{Presp}$) to be 
\[\bigcup_{T_V \eps \C_0'} \{\F_{T_V}(i)
\mid i \eps \oI \} (\mbox{ and } \bigcup_{T_V \eps \C_0'} \{F_{T_V}(i)
\mid i \eps \oI \}, \mbox{ \res }) \mbox { and }\]
 the generating trivial cofibrations $\oJ_{Sp}$ ($\oJ_{Presp}$) to be 
\[\bigcup_{T_V \eps \C_0'}
\{\F_{T_V}(j)|j \eps \oJ\} \, \,  (\bigcup_{T_V \eps \C_0'}
\{F_{T_V}(j)|j \eps \oJ\}, \mbox{ \res }).\]
\begin{proposition} 
\label{key.props.unstable}
(i) If $A \eps \C$ is small relative to $\oI$ (or the cofibrations
 in $\C$), then $\F_{T_V}(A)$ ($F_{T_V}(A)$) is small relative $\oI_{Sp}$ ($\oI_{Presp}$, \res).
The corresponding assertions hold also for trivial cofibrations with $\oI_{Sp}$ ($\oI_{Presp}$) replaced by
$\oJ_{Sp}$ ($\oJ_{Presp}$, \res). 
\vskip .3cm
(ii) A map $f$ in $Spectra (\C)$ (in $Presp(\C)$) is a level cofibration if and only if it has the left lifting property
 with respect to all maps of the form $\R_{T_V}(g)$ ($R_{T_V}(g)$, \res) where $g$ is a trivial fibration in $\C$.
A map $f$ in $Spectra (\C)$ (in $Presp(\C)$) is a level trivial cofibration if and only if it has the left lifting property
 with respect to all maps of the form $\R_{T_V}(g)$ ($R_{T_V}(g)$, \res) where $g$ is a fibration in $\C$.
\vskip .3cm 
(iii) Every map in $\oI_{Sp}-cof$ ($\oI_{Presp}-cof$) is a level cofibration and every map in $\oJ_{Sp}-cof$ ($\oJ_{Presp}-cof$)
is a level trivial cofibration. The projective cofibrations (projective trivial cofibrations) identify with $\oI_{Sp}-cof$ ($\oJ_{Sp}-cof$, \res).
\vskip .3cm
(iv) The domains of $\oI_{Sp}$ ($\oJ_{Sp}$) are small relative to $\oI_{Sp}-cell$ ($\oJ_{Sp}-cell$, \res). Similarly
the domains of $\oI_{Presp}$ ($\oJ_{Presp}$) are small relative to $\oI_{Presp}-cell$ ($\oJ_{Presp}-cell$, \res).
\end{proposition}
\begin{proof} We will only consider the case of spectra, since the corresponding statements for pre-spectra are similar. 
(i) The main observations here are that the functor ${\mathcal E}val_{T_V}$ being left adjoint to $\R_{T_V}$, ${\mathcal E}val_{T_V}$ commutes 
with all small colimits and that all the maps in $\oI_{Sp}-cell$ are level cofibrations, i.e. cofibrations in $\C$ after applying any 
${\mathcal E}val_{T_V}$. This proves the case if $A$ is small 
relative to the cofibrations in $\C$. If $A$ is small relative to $\oI$, one may
invoke \cite[Proposition 2.1.16]{Hov-1} to observe that $A$ is then small relative to the cofibrations in $\C$.
 \vskip .3cm
(ii) Since $\R_{T_V}$ is right adjoint to ${\mathcal E}val_{T_V}$, ($R_{T_V}$ is right adjoint to $Eval_{T_V}$) $f$ has the left lifting property with 
respect to
$\R_{T_V}(g)$ ($R_{T_V}(g)$) if and only if ${\mathcal E}val_{T_V}(f)$ ($Eval_{T_V}(f)$) has the left-lifting property with respect to $g$. (ii) follows
readily from this observation.
\vskip .3cm
(iii) Recall every object of $\C$ is assumed to be cofibrant. Therefore, smashing with $\Hom_{\C_0'}(T_V, T_W)$ or any $T_V$, with 
$T_V, T_W \eps \C'_0$
preserves cofibrations of $\C$ when $\C_0'$ is a full enriched sub-category of $\C$. The only case we consider where $\C_0'$ is not an enriched full
sub-category of $\C$ is the case in Examples ~\ref{main.eg}(iv) where our definition ensures that $\Hom_{\C_0'}(T_V, T_W)$ is 
cofibrant in $\C$.
Therefore, every map in $\oI_{Sp}$ is a level cofibration. By (ii) this means
 $\R_{T_V}(g) \eps \oI_{Sp}-inj$ for all trivial fibrations $g$ in $\C$. Recall every map in $\oI_{Sp}-cof$ has
the left lifting property with respect to every map in $\oI_{Sp}-inj$ and in particular with 
respect to every 
map $\R_{T_V}(g)$, with $g$ a trivial fibration in $\C$. Now the adjunction between ${\mathcal E}val_{T_V}$ and $\R_{T_V}$
completes the proof for $\oI_{Sp}-cof$. The proof for $\oJ_{Sp}-cof$ and for pre-spectra is similar. These prove the first statement in (iii).
The second statement in (iii) follows readily by observing that the trivial level fibrations (level fibrations) identify with $\oI_{Sp}-inj$
($\oJ_{Sp}-inj$,\res).
\vskip .3cm
(iv) follows readily in view of the adjunction between the free functor $\F_{T_V}$ and ${\mathcal E}val_{T_V}$. 
\end{proof}
\vskip .3cm
We now obtain the following corollary.
\vskip .3cm
\begin{corollary} 
\label{level.model.1}
 The projective cofibrations, the level fibrations and level equivalences define a cofibrantly 
generated model category structure on $Spectra (\C)$ with the generating cofibrations (generating trivial
cofibrations) being $\oI_{Sp}$ ($\oJ_{Sp}$, \res). This model structure (called {\it the 
projective model structure} on
 $Spectra (\C)$) has the following properties:
\be \begin{enumerate}[\rm (i)]
\item{It is left-proper. It is also right proper if
the corresponding model structure on $\C$ is right proper. It is cellular if
 the corresponding model structure on $\C$ is cellular.}
\item{The objects in 
$\bigcup_{T_V \eps \C_0'}\{\F_{T_V}(\C_0')\}$ are all finitely presented. The category $Spectra(\C)$
is symmetric monoidal with the pairing defined as follows.
One first considers the following product for $\chi, \chi' \eps[\C_0', \C]$ 
\[\chi \bar \wedge \chi': \C_0' \times \C_0' \ra \C, \chi \bar \wedge \chi' (T_V , T_W) = \chi(T_V) \wedge \chi'(T_W).\]
Next one defines $\chi \wedge \chi'$ as the left Kan extension of $\chi \bar \wedge \chi'$ along the monoidal product
$\wedge : \C_0' \times \C_0' \ra \C_0'.$
Now the unit of the above monoidal structure is the (enriched) inclusion functor
$$i:  \C_0' \ra \C. $$}
\item{The projective model structure on $Spectra (\C)$ is weakly finitely generated when the given model 
structure on $\C$ is weakly finitely generated.}
\item{The category $Spectra (\C)$ is locally presentable and hence  $Spectra(\C)$ is a tractable (and hence a combinatorial)
model category.}
\item{With the above structure, $Spectra(\C)$ is a symmetric monoidal model category satisfying the monoidal axiom.}
\end{enumerate} \ee
\end{corollary}
\begin{proof} Recall from Proposition ~\ref{diagram.cats.excellent} and Proposition ~\ref{localization.1}(iii) that
the category $\C$ is an excellent model category. Therefore, one may deduce the existence of the above projective model structure readily by 
invoking \cite[Proposition A.3.3.2]{Lur}. However, we will provide some details for the convenience of the reader.
\vskip .3cm
The retract and two out of three axioms (see \cite[Definition 1.1.3]{Hov-1}) are immediate, as is the lifting axiom
for a projective cofibration and a level trivial fibration. Clearly a map is a level trivial
fibration if and only if it is in $\oI_{Sp}-inj$. Therefore, a map is a projective cofibration if and only if it is in
$\oI_{Sp}-cof$. Now Proposition ~\ref{key.props.unstable}(iv) show that \cite[Theorem 2.1.14]{Hov-2} applied to $\oI_{Sp}$ then produces
a functorial factorization of a map as the composition of a projective cofibration followed by a level trivial fibration.
\vskip .3cm
By adjunction, a map is a level fibration if and only if it is in $\oJ_{Sp}-inj$. Proposition ~\ref{key.props.unstable}(iii) shows that
every map in $\oJ_{sp}-cof$ is a level equivalence. Such maps have left-lifting property with respect to all level   fibrations and hence
with respect to all level trivial fibrations. Now Proposition ~\ref{key.props.unstable}(iv) shows that \cite[Theorem 2.1.14]{Hov-2} applied to 
$\oJ_{Sp}$ then produces
a functorial factorization of a map as the composition of a projective  cofibration which is also a level equivalence 
followed by a level fibration. 
\vskip .3cm
Next we show that any map $f$ which is a projective cofibration and a level equivalence is
in $\oJ_{Sp}-cof$, and hence has the left lifting property with respect to level fibrations.
To see this, we factor $ f = pi$ where $i$ is in $\oJ_{Sp}-cof$ and $p$ is in $\oJ_{Sp}-inj$. Then $p$ is a
level fibration. Since $f$ and $i$ are both level equivalences, so is $p$. Therefore $p$ is a level trivial fibration and $f$ has the
left lifting property with respect to $p$. This shows $f$ is a retract of $i$: see, for example, \cite[Lemma 1.1.9]{Hov-1}.  In particular 
$f$ belongs to  $\oJ_{Sp}-cof$. Using the adjunction between the functor $\F_{T_V}$ and $Eval_{T_V}$, one may now readily check any of the
remaining hypotheses in \cite[Proposition 2.1.19]{Hov-1}. These prove the existence of the projective model structure on $Spectra(\C)$.
\vskip .3cm
Since pushouts and pullbacks  in $Spectra(\C)$ are taken level-wise, the statements in (i) are clear.
The first assertion in (ii) is clear since the objects in the subcategory $\C_0'$ are assumed to be finitely presented in $\C$. 
The assertions in (ii) on the monoidal structure follow from 
a theorem of Day: see \cite{Day}. See also \cite[Theorem 2.6]{Dund1}. Statement (iii) follows from \cite[Theorem 4.2]{Dund1}.
 Statement (iv) follows from \cite[Proposition A.3.3.2]{Lur}. The last statement follows from \cite[Theorem 4.4]{Dund1}.
\end{proof}
\vskip .3cm
\subsubsection{\bf The injective model structure on $Spectra(\C)$.}
Here we define a map $f:\chi' \ra \chi$ of spectra to be an {\it injective cofibration} (an {\it injective weak-equivalence})
if it is a level cofibration (a level equivalence, \res). The {\it injective fibrations} are defined by the lifting property with
respect to trivial cofibrations.  
\begin{proposition}
 \label{inj.model.struct}
This defines a combinatorial (in fact tractable) simplicial monoidal model structure on $Spectra (\C)$ that is left proper if $\C$ is. Every projective cofibration is
an injective (i.e. level) cofibration and every injective fibration is a level fibration. The cofibrations are the monomorphisms. 
The unit of the monoidal structure on $Spectra(\C)$ and in fact every object in $Spectra(\C)$ is cofibrant in this
model structure.
\end{proposition}
\begin{proof} We start with the observation that the category $\C$ is a simplicially enriched tractable simplicial model category.
 The left-properness is obvious, since the cofibrations and weak-equivalences are defined level-wise.
 The first conclusion follows now from \cite[Proposition A.3.3.2]{Lur}: observe that the pushout-product axiom holds
since cofibrations (weak-equivalences) are object-wise cofibrations (weak-equivalences, \res) and the pushout-product
 axiom holds in the monoidal model category $\C$. The second statement follows from Proposition ~\ref{key.props.unstable}(iii).
Recall the unit of $Spectra(\C)$ is the inclusion functor
$\C'\ra \C$. We will denote this by $\Sigma$. To prove it is cofibrant, all one has to observe is that 
$\Sigma(T_V)=T_V$ which is cofibrant in $\C$ for every $T_V \eps \C'$.
\end{proof}
\begin{remark} One may observe that the unit of $Spectra(\C)$ may not be cofibrant in the projective model structure, whereas
it is cofibrant in the above injective model structure. This makes it advantageous to consider the injective model structure.
\end{remark}
\vskip .3cm
Recall a set of homotopy generators in a model category is a set of objects so that
every object in the model category is a (filtered) homotopy colimit of these objects.

\begin{proposition}
 \label{homotopy.generators}
 In the projective model structure on $Spectra(\C)$, every object is the homotopy colimit of a simplicial object in $Spectra(\C)$ which is cofibrant in
each simplicial degree.
\end{proposition}
\begin{proof}
 Since enriched functors $[\C_0', \C]$ are functors on $\C_0'$ with values in $\C$, one can find a simplicial resolution of
any enriched functor by a simplicial object of enriched functors, each term of which is a representable functor. More precisely,
let $(\C_0')^{disc}$ denote the category which has the same objects as $\C_0'$, but only identity maps. Then one may restrict any enriched functor $\C_0' \ra \C$ to
$(\C_0')^{disc}$: this restriction functor will be denoted $U:[\C_0', \C] \ra [(\C_0')^{disc}, \C]$. Observe that an enriched functor
$(\C_0')^{disc} \ra \C$ is given by specifying its values on the objects of $\C_0'$. Therefore, $U$  has a left adjoint defined
by sending $\bar \chi$ to $ \sqcup_{T_V \eps \C_0'} \Hom_{\C}(T_V, \quad) \wedge \bar \chi(T_V) = \sqcup_{T_V \eps \C_0'}\F_{T_V} (\bar \chi (T_V))$.
We denote this enriched functor by $\F$. Now $\F\circ U$ defines a triple and produces on iteration a simplicial object 
\vskip .3cm
$\F U_{\bullet}(\chi)$ together with an augmentation to $\chi$, for any enriched functor $\chi: \C_0' \ra \C$. 
\vskip .3cm 
To see that this is a 
simplicial resolution of $\chi$, one observes that it suffices to prove this after applying $U$ one more time to the above
simplicial object: then it will have an extra degeneracy, which proves the above simplicial object is in fact a resolution of $\chi$. 
\end{proof}
\subsubsection{\bf The stable model structures on $Spectra(\C)$.}
\label{stable.model}
We proceed to define the stable model structure on $Spectra(\C)$ by applying a suitable Bousfield
localization to the projective and injective model structures on $Spectra(\C)$. This follows the approach in
\cite[section 3]{Hov-3}. The corresponding model structures will be called the {\it projective stable model structure} and the {\it the injective stable
 model structure} on $Spectra (\C)$. (One may observe that the domains and co-domains of objects in $\oI$ are 
cofibrant, so that there is no need for a cofibrant replacement functor $Q$ as in \cite[section 3]{Hov-3}.)
\begin{definition} ($\Omega$-spectra)
 A spectrum ${ \chi} \eps Spectra(\C)$ is an  {\it $\Omega$-spectrum} if it is level-fibrant and 
each of the natural maps
${ \chi}(T_V) \ra \Hom_{\C}(T_{W}, { \chi}(T_V \wedge T_W))$, $T_V, T_W \eps \C_0'$ is a weak-equivalence in $\C$.
\end{definition}
\vskip .3cm
 Observe that giving a map 
${\mathcal F}_{T_V \wedge T_W}(T_W \wedge C) \ra {\mathcal F}_{T_V}(C)$ corresponds by adjunction to giving a map
$C \wedge T_W  \ra ({\mathcal F}_{T_V}(C))(T_V \wedge T_W) = C \wedge \Hom_{\C}(T_V, T_V \wedge T_W)$.   Since $T_W $ maps naturally to $\Hom_{\C}(T_V, T_V \wedge T_W)$,
there is a natural map $C \wedge T_W  \ra ({\mathcal F}_{T_V}(C))(T_V \wedge T_W)$.
Therefore 
we let
\be \begin{equation}
\label{S}     
{\mathbf S} \mbox{ denote the morphisms } \{{\mathcal F}_{T_V \wedge T_W}(C \wedge T_W ) \ra {\mathcal F}_{T_V}(C) \mid C \eps
 \mbox{ Domains or Co-domains of } \oI, T_V, T_W \eps \C_0' \}
\end{equation} \ee
\vskip .3cm \noindent
corresponding to the above  maps $C \wedge T_W  \ra C \wedge \Hom_{\C}(T_V, T_V \wedge T_W)$ by adjunction. 
\vskip .3cm
The {\it stable model structure} on $Spectra(\C)$ is obtained by localizing the projective or injective model
structure on $Spectra(\C)$ with respect to the maps in ${\mathbf S}$. The reason for considering such
maps is as follows: let $C \eps \C$ be an object as above and let $\chi \eps Spectra(\C)$ be fibrant in the projective or injective model structures. Then 
\[Map(C, \chi(T_V)) =Map(C, {\mathcal E}val_{T_V}(\chi)) \simeq Map(\F_{T_V}(C), \chi) \mbox{ and }\]
\[ Map(C, \Hom_{\C}(T_W, \chi(T_V \wedge T_W))) = Map(C, \Hom_{\C} (T_W, {\mathcal E}val_{T_V \wedge T_W}(\chi))) \simeq Map(\F_{T_V \wedge T_W}(C \wedge T_W ), \chi).\]
 Therefore to convert
$\chi$ into an $\Omega$ -spectrum, it suffices to invert the maps in ${\mathbf S}$. (See \cite[Proposition 3.2]{Hov-3} that shows it suffices to consider
the objects $C$ that form the domains and codomains of the generating cofibrations in $\C$.)
The ${\mathbf S}$-local weak-equivalences
(${\mathbf S}$-local fibrations) will be referred to as the {\it stable equivalences} 
({\it stable fibrations}, \res). The cofibrations in the localized model structure are the 
cofibrations in the projective or injective model structures on $Spectra(\C)$.

\begin{proposition} 
\label{stable.model.1}
(i) The corresponding stable model structure on $Spectra(\C)$ is cofibrantly 
generated and left proper when $\C$ is assumed to have these properties. It is also
locally presentable, and hence combinatorial (tractable).  The projective stable model structure on $Spectra(\C)$ is also
cellular if $\C$ is.
\vskip .3cm 
(ii) The fibrant objects in the stable model structure on $Spectra(\C)$ are the
 $\Omega$-spectra defined above. 
\vskip .3cm
(iii) The category $Spectra (\C)$ is a symmetric monoidal model category 
(i.e. satisfies the pushout-product axiom: see \cite[Definition  3.1]{SSch}) in both the projective and injective stable model structures with
the monoidal structure being the  same in both the
 model structures. In the injective model structure, the unit is cofibrant and the monoidal axiom (see \cite[Definition 3.3]{SSch})
is also satisfied.
\vskip .3cm
(iv) The stable model structure on $Spectra(\C)$ is weakly finitely generated when the given model structure
on $\C$ is and one starts with the projective unstable model structure on $Spectra(\C)$.
\end{proposition}
\begin{proof} The proof of the first statement in (i) is entirely similar to the proof of \cite[Theorem 3.4]{Hov-3} and is 
therefore skipped. If $\C$ is left proper, so are both the projective and injective unstable model structures on $Spectra (\C)$ 
as proved above. In case $\C$ is cellular, it was shown above that the projective unstable model structure is also cellular. 
It is shown in \cite[Proposition 3.4.4 and Theorem 4.1.1]{Hirsch} that then the localization of the unstable model structures preserves
these properties. The fact it is locally presentable and hence combinatorial follows from the corresponding property of the
unstable model categories: see \cite[Theorem 3.18]{Bar}. Our hypotheses show that the domains of the maps in $I$ and $J$ are cofibrant. These prove all
the statements in (i).
\vskip .3cm
(ii) is clear. 
Clearly, since the monoidal structure is the same as in the unstable setting, the category $Spectra(\C)$ is symmetric monoidal. Now to
prove it is a monoidal model category, it suffices to prove that the pushout-product axiom holds. This may be proven exactly as in the proof of 
\cite[Proposition 3.19]{Bar}. i.e. It suffices to prove the following: Let $i:X \ra Y$ denote a trivial cofibration in the above localized
model category of spectra and let $f:A \ra B$ denote a generating cofibration in the given model structure on $Spectra(\C)$. Then it suffices to
show that the pushout-product $i \square f$ is a weak-equivalence. This will hold, if for any fibrant object $Z$ in the localized category of
spectra, the diagram
\[\xymatrix{{R\Hom(Y, Hom(B, Z))} \ar@<1ex>[r] \ar@<1ex>[d] & {R\Hom(X, \Hom(B, Z))} \ar@<1ex>[d] \\
{R\Hom(Y, Hom(A, Z))} \ar@<1ex>[r]  & {R\Hom(X, \Hom(A, Z))}}\]
is homotopy cartesian, where $\Hom$ denotes the internal hom in the category $Spectra(\C)$ and $\RHom$ denotes its obvious 
derived functor. Now observe that the sources and targets of the generating cofibrations are themselves cofibrant in the projective model
structure (and trivially so in the injective model structure) on $Spectra (\C)$. Therefore, both $\Hom(A, Z)$ and $\Hom(B, Z)$ are fibrant objects
in the localized model structures, which proves both the horizontal maps are weak-equivalences. Therefore, the above square is homotopy cartesian.
\vskip .3cm
It is clear that the unit is cofibrant in the injective model structure. Moreover, since an enriched functor 
$F:\C_0' \ra \C$ is cofibrant if ${F}(T_V)$ is cofibrant in $\C$ for every $T_V \eps \C_0'$ and every object of $\C$ is assumed to be 
cofibrant, it follows that
every such functor is cofibrant in the injective stable model structure. Therefore, the monoidal axiom is also satisfied. This proves (iii).
\vskip .3cm
 The last statement is obvious from the corresponding properties of the 
projective model structure.
 \end{proof}   

\vskip .3cm
\subsubsection{\bf Alternate stabilization for symmetric spectra}
\label{alt.stab.symm.sp}
As we observed earlier (see ~\ref{main.eg}(iv)), when the group $\group$ is trivial, the categories of spectra we obtain 
include the category of
symmetric spectra. If $T$ denotes the starting sphere, then one may obtain a fibrant spectrum by the usual process.
i.e. Let $\X =\{\X_n = \X(T^{\wedge ^n})|n \ge 0\}$ denote a spectrum. Then one may define the symmetric spectrum $R\X$
by $R\X_n = \colimk \Omega_T^k(\X(T^{\wedge ^{n+k}}))$ with the induced action of the symmetric group $\Sigma _n$.
Then one may show readily that $R\X = \{R\X_n|n \ge0\}$ is a fibrant symmetric spectrum.
\vskip .3cm

\subsubsection{\bf An alternate construction of the category of spectra}
\label{spectra.construct.2}
In the construction of the stable category of spectra considered above, we start with
unstable category of pointed simplicial presheaves provided with the object-wise model structure, then localize this to 
invert ${\mathbb A}^1$-equivalences,
 then stabilize to obtain the category of spectra with the unstable projective or injective model structures which are then 
localized to obtain the category of spectra with the corresponding stable model structures. 
Instead of these steps,  one may adopt the
following alternate series of steps to construct the category of spectra with its stable model
structure, where all the localization is carried out at the end. 
\vskip .3cm
One starts with one of the categories of pointed simplicial presheaves $\C= \PSh^{\O^{o}_{\group}}$ or $\PSh/S^{\O^{o}_{\group}}$
 with the  object-wise model structures in general. Next we 
apply the stabilization construction as discussed above and obtain $Spectra(\C)$ with its {\it unstable}
 projective or injective model structure considered above. 
This is left proper and a tractable model category which is cellular and weakly finitely generated in the projective case. 
Assume first that $? = Nis$. One defines a presheaf $P$ in any of the above categories to be 
{\it motivically fibrant} if 
\begin{enumerate}[\rm(i)]
 \item 
 $P$ is fibrant in $Spectra(\C)$, 
\item
$\Gamma (\phi, P) $ is contractible (where $\phi$ denotes the empty 
scheme), 
\item sends an {\it elementary 
distinguished square} as in \cite{MV} to a homotopy cartesian square on taking sections and 
\item the obvious pull-back 
$\Gamma (U, P) \ra \Gamma (U_+ \wedge {\mathbb A}^1_+, P)$ is a weak-equivalence.
\end{enumerate}
 Then a map 
$f: A \ra B$ 
 in  $Spectra(\C)$ is a {\it motivic weak-equivalence} if the induced map $Map(f, P)$ is
a weak-equivalence for every motivically fibrant object $P$, with $Map$ denoting the simplicial
mapping space.
\vskip .3cm
Next we localize this by inverting maps that belong to any one of the following
classes: (i) $f$ is a motivic weak-equivalence  and (ii) 
$\{{\mathcal F}_{T_V \wedge T_W}(T_V \wedge C) \ra {\mathcal F}_{T_W}(C) \mid C \eps
 \mbox{ Domains and Co-domains of } \oI, T_V, T_W \eps \C_0' \}$. (Here $\oI$ is the set of generating cofibrations of $\C$.)
 The cofibrations in the localized
model structure will be the same as those of
$Spectra (\C)$. The fibrations 
are defined by the lifting property. By \cite[Theorem 4.1.2]{Hirsch}, it follows
that this is a left-proper model category which is cellular when one sticks with the projective model structures throughout. The resulting stable model structure identifies with
the stable model structures on motivic spectra and $\T$-motivic spectra obtained above.
\vskip 3cm
One may in fact carry out this process in two stages. First one only inverts maps 
of the following form: $\{{\mathcal F}_{T_V \wedge T_W}(T_V \wedge C) \ra {\mathcal F}_{T_W}(C) \mid C \eps
 \mbox{ Domains and Co-domains of } \oI, T_V, T_W \eps \C_0' \}$. (Here $\oI$ is the set of generating cofibrations of $\C$.
This produces a stable model structure on $Spectra (\C)$ where the fibrant objects are $\Omega$-spectra. This stable model
category will be denoted $Spectra_{st}(\C)$. Then one inverts the maps $f$ which are motivic weak-equivalences to obtain
the ${\mathbb A}^1$-localized stable category of spectra.

One obtains a 
similar alternate construction of 
\'etale spectra as well.  
\subsection{\bf Key properties of $\Spt_{ mot}^{\group}$, $\Spt_{et}^{\group}$, $\Spt/S_{mot}^{\group}$ and 
$\Spt/S_{et}^{\group}$.}
\label{keyprops.spectra}
We summarize the following properties which have already been established above.
\label{key.props.mot.stab.cat}
\begin{enumerate}[\rm(i)]
\item{{\it Weak-equivalences and fibration sequences.} Suppose $A$ and $B$ are $\Omega$-spectra. Then a map $f: A \ra B$ in any one of the above
 categories of $\group$-equivariant spectra is a weak-equivalence if and only if the induced map
 $f(T_V)(\group/H):A(T_V)(\group/H) \ra B(T_V)(\group/H)$ is a weak-equivalence of pointed simplicial sets for all subgroups $H \eps \W$ and all $T_V$. 
A diagram $F \ra E \ra B $ of $\Omega$-spectra is a fibration sequence in 
any one of the above categories
of equivariant spectra only if the induced diagrams 
\be \begin{equation}
     \label{fib.seqs}
 F(T_V)(\group/H) \ra E(T_V) (\group/H) \ra B(T_V)(\group/H)  
\end{equation} \ee
\vskip .3cm \noindent
 are all fibration sequences of spectra for all subgroups $H \eps \W$  and all $ T_V$. 
These follow readily from the basic properties of left Bousfield localization: see \cite[Proposition 3.4.8 and Theorem 4.3.6]{Hirsch}. }
 \item{{\it Stably fibrant objects}. A spectrum $E$
is fibrant in the above stable
model structure if and only if each $E(T_{V})$ is fibrant in $\C$= the appropriate unstable category of simplicial
presheaves and  the induced map
$E(T_V) \ra \Hom_{\C}(T_W, (E(T_V\wedge T_W)))$ is a weak-equivalence of $\group$-equivariant  spaces
for all $T_V, T_W \eps \C_0'$.}
\item{{\it Finite sums}. Given a finite collection $\{E_{\alpha}|\alpha \}$ of
objects, the finite sum ${\rm V}_{\alpha} E_{\alpha} $
identifies with the
product $\Pi_{\alpha} E_{\alpha}$ up to stable weak-equivalence.}
\item{{\it Additive structure}. The corresponding stable homotopy categories  are additive
categories.} 
\item{{\it Shifts}. Each $T_V \eps \C_0'$ defines  a shift-functor $E \ra E[T_V]$, where $E[T_V](T_W) =
E(T_V \wedge T_W)$. This is an 
automorphism of  the corresponding stable homotopy category.}
\item{{\it Combinatorial (Tractable) left proper simplicial model category structure}. All of the above
categories of spectra  have the structure of {\it combinatorial (in fact tractable) left proper simplicial model categories}.}
\item{{\it Symmetric monoidal model structure}. There is a symmetric monoidal model
structure on all the above categories of spectra, 
 where the product is
denoted $\wedge$, i.e. the pushout-product and monoid axioms are satisfied. The  sphere spectrum, (i.e. the inclusion of $\C_0'$ into $\C$) is the unit in this
symmetric monoidal structure. In the stable injective model structure, the sphere spectrum is cofibrant.
Given a fixed spectrum $E$, the functor $F \mapsto E \wedge F$ has a right adjoint which
will be denoted $\Hom$. This is the internal hom in the above categories of spectra. 
The derived functor of this $\Hom$ denoted $\RHom$ may be defined as follows.
$\RHom(F, E) = \Hom(C(F), Q(E))$, where $C(F)$ is a cofibrant replacement of $F$
and $Q(E)$ is a fibrant replacement of $E$.}
\item{ {\it Ring spectra}. Algebras in the above categories of spectra will be called 
{\it $\group$-equivariant ring spectra} (and {\it ring-spectra} for short.)}
\item{{\it Change of groups}. Let $\phi:\group' \ra \group$ denote a homomorphism and let $\W'$ ($\W$) be a chosen collection of subgroups of $\group'$ ($\group$) so that
they satisfy the hypotheses in ~\ref{cont.action} and so that $\phi^{-1}(H) \eps \W'$ for all $H \eps \W$.Then restricting the group action from $\group$ to
 $\group'$ along the homomorphism $\phi$ defines the restriction functor sending categories of $\group$-equivariant spectra to $\group'$-equivariant spectra.
In special cases, the restriction functor will have an adjoint called induction. This will be explored elsewhere.}
\end{enumerate}  
\begin{definition} (Stable homotopy groups).  
Let  $E$ denote a fibrant spectrum in any of the above categories of spectra
and let $H$ denote any subgroup of 
$\group$ so that $ H \eps \W$ and so that $H$ acts trivially on the affine space $V$ on which $\group$ acts linearly. 
Recall that since $E$ is fibrant, the obvious map $E(T_V) \ra \Omega_{T_W}(E(T_V \wedge T_W))=
\Hom_{\C}(T_W, E(T_V \wedge T_W))$
is a weak-equivalence. Therefore, the $k$-th iteration of  the above map,  $E(T_V) \ra 
\Omega^k_{T_W}(E(T_V \wedge T_W^{\wedge ^k}))$ is also a weak-equivalence.
 Therefore, 
 we observe for $U \eps \bS$ or $U \eps {\rm {Sm/ S}}$ with a continuous $\group$-action (i.e. where $U$ is stabilized by some
$H \eps \W$)
\vskip .3cm
 $\pi_{s+2|T_V|, T_V}^H(E(U)) = [ \Sigma ^{s+2|T_V|, T_V}  U_+, E^H_{|U}] \cong
\colimn [T_V \wedge T_{W}^{\wedge ^n} \wedge S^s \wedge U_+,  E^H_{|U}( T_W^{\wedge ^n})]$.
\vskip .3cm \noindent
Here $[A, B]$ denotes $Hom$ in the homotopy category associated to the corresponding category of $\group$-equivariant spectra
 and where $H$ also acts trivially on $W$. We have used the notation $E^H$ for $E(\group/H)$.
\vskip .3cm
Next assume that $\group$  denotes a profinite group and 
let $H_{\group}$ denote the largest normal subgroup of $\group$ contained in $H$. (This is
 the {\it core} of $H$ we considered earlier and has finite index in $\group$.) Then, for $U \eps \bS$ or $U \eps {\rm {Sm/ S}}$ with a continuous $\group$-action, we let:
\vskip .3cm
$\H^{s+2|T_V|, T_V, H}(U_+, E) = \Hom_{Sp, \group/H_{\group}}( \Sigma ^{s+2|T_V|, T_V}  U_+, E^{H_{\group}}_{|U})=  $
\vskip .3cm
$ \cong \colimn \Hom_{Sp, , \group/H_{\group}} (T_V \wedge T_{W}^{\wedge ^n} \wedge S^s \wedge U_+,   E^{H_{\group}}_{|U}( T_W^{\wedge ^n}))$
\vskip .3cm \noindent
where $\Hom_{Sp}$
denotes the internal hom-functor defined in ~\ref{mapping.spectra} for $\group/H_{\group}$-equivariant 
objects and $?$ denotes either the Nisnevich or the \'etale sites. We will use  the notation
\[\H^{s+2|T_V|, T_V, H}_{mot}(U_+, E) (\H^{s+2|T_V|, T_V, H}_{et}(U_+, E))\] 
to denote
the corresponding hypercohomology when $E \eps \Spt_{S, mot}^{\group}, \Spt/S_{mot}^{\group}$ ($\Spt_{S,et}^{\group}, 
\Spt/S_{et}^{\group}$, \res). 
\end{definition}
\vskip .3cm
Observe that if we restrict to ${\mathbb P}^1$-motivic spectra, then a general 
$T_W \eps \C_0'$ is a finite iterated smash product of terms of the form 
${\underset {\group/K} \wedge} {\mathbb P}^1$, for some normal subgroup $K$ of $\group$ with finite 
index. When the group actions are ignored (or trivial), a general $T_W \eps \C_0'$ is a finite
 iterated smash product of ${\mathbb P}^1$, so that in this case the above stable homotopy groups
may be just indexed by two integers. 

\subsection{Assorted results on spectra}
In this subsection, we will briefly consider several useful constructions and results on spectra and pre-spectra.
\subsubsection{\bf Spectra indexed by the non-negative integers} 
\label{usual.spectra}
One may construct spectra indexed by the non-negative integers from 
 $Spectra(\C)$ as follows. Let $T_{V}$ denote a fixed element of $\C'_0$. 
 One may now consider the full symmetric monoidal subcategory
of $\C'_0$ generated by $T_V$ under iterated smash products along with the unit $S^0$. Any
spectrum ${\mathcal X}$ in $Spectra (\C)$ may be restricted to this sub-category and provides
a spectrum ${\overline {\mathcal X}}$ in the usual sense by defining ${\overline {\mathcal X}}_n
= {\mathcal X}(\wedge ^{n} T_V)$.  It is, however, important to observe that for the
full subcategory $\C_0'$ to be symmetric monoidal, $\Hom_{\C_0'}(T_V^n, T_V^{n+k})$ is not just $T_V^k$,
but something bigger. 
\vskip .3cm
For example, when $\group$ is trivial, this category corresponds to that of {\it symmetric
spectra with values in $\C$} as \cite[2.6]{Dund1}, so that $\Hom_{\C_0'}(T_V^n, T_V^{n+k}) = \vee_{\alpha \eps Inj(n, n+k)}T_V^k$
where $Inj(n, n+k)$ denotes the set of all injective maps $n \ra n+k$.  In particular,
the symmetric monoidal subcategory generated by the spheres $T_V$s under iterated smash products already contains
the symmetric spheres $\vee_{\alpha \eps Inj(k, k)}T_V^k$.
\vskip .3cm
Similar observations apply to pre-spectra as well.
Given any pre-spectrum ${\mathcal X}$, one obtains this way, the pre-spectrum 
${\overline {\mathcal X}}$ indexed the non-negative integers where
 ${\overline {\mathcal X}}_n = {\mathcal X}(T_V^{\wedge ^n})$.
\vskip .3cm
\begin{examples} {\bf Suspension spectra}
\label{sus.spectra}
(i) Let $\C'_0=\C'=\{T_V\}$ as $V$ varies among all continuous representations of $\group$ (i.e. where $V$ is stabilized by some $H \eps \W$)
 and let
 $ A \eps \C$. Then the {\it motivic suspension spectrum associated to} $A$ is the spectrum
$\Sigma _{mot}(A)$ whose value on $T_V \eps \C_0'$ is given by $T_V \wedge A$. 
If $\T$ is as in Example ~\ref{main.eg}(ii), 
the {\it $\T$-motivic suspension spectrum} associated to $A$ is is the spectrum whose value on 
$T_K \eps \C_0'$ is given by
$T_K \wedge A$. This will be denoted $\Sigma _{\T}(A)$. 
If the subgroup $K$ of $\group$ is fixed, we will denote the corresponding suspension spectrum by
$\Sigma _{\T, K}(A)$. 
If we take $A=S^0$, we obtain the
 {\it motivic sphere spectrum} $\Sigma_{mot}$ and the {\it $\T$-motivic sphere spectrum}
 $\Sigma _{\T}$. In case $F: \C \ra \C$ is a functor commuting with $\wedge$ and with
colimits, then we obtain the motivic $F$-spectrum $\Sigma _{mot, F}$ and the $F(\T)$-motivic
 sphere spectrum $\Sigma _{F(\T)}$. 
\vskip .3cm
(ii)  If we take $A= S^n \wedge B$, for some $n \ge 1$ and
$B \eps \C$ (as above), the resulting spectra are the $S^n$-suspensions of the above suspension
spectra associated to $A$. This will be denoted $\Sigma^{n,0}_{mot}B$. 
If we take $A= T_V \wedge S^s \wedge B$, $T_V \eps \C_0'$, the resulting $\T$-motivic suspension spectrum 
associated to $A$ will be denoted $\Sigma ^{s+2|T_V|, T_V}B$, where $|T_V| = dim(V)$.
\vskip .3cm
If we take 
$A = (\T)^n \wedge B$ for some $B \eps \C$, the resulting $\T$-motivic
 suspension spectrum associated to $A$ will be the $(\T)^n$-suspension of the $\T$ -motivic
suspension spectrum associated to $B$. This will be denoted $\Sigma ^{2n, n}_{\T}(B)$.
In case $A = S^n \wedge (\T)^m \wedge B$ for some $B \eps \C$, the resulting 
$\T$-motivic suspension spectrum will be denoted $\Sigma^{n+2m, m}_{\T}B$.
\end{examples}
\subsubsection{\bf Suspending and de-suspending spectra}
Let $E$ denote a spectrum in any one of the above categories, with $T_V$ denoting elements of 
$\C'_0$. Then we define the {\it suspension} $\Sigma ^{2|T_V|, T_V}E$ to be the spectrum defined by
$\Sigma ^{2|T_V|, T_V}E(T_U) = E(T_V \wedge T_U)$. The {\it de-suspension }
$\Sigma ^{-2|T_V|, -T_V}E$ will be the spectrum defined by
$\Sigma ^{-2|T_V|, -T_V}E(T_W) = \Hom_{\C}(T_V, E(T_W))$.
\vskip .3cm
\subsubsection{\bf Equivariant vs. non-equivariant spectra}
\label{equiv.nonequiv}
The equivariant spectra will denote spectra as in Example ~\ref{main.eg}(i) and Example ~\ref{main.eg}(ii).
If $\T$ denotes a fixed object in $\C_0'$ as in ~\ref{main.eg}(ii), $\Sigma_{\T, \group}$ will denote the {\it $\group$-equivariant sphere spectrum} associated to
$\T$. The non-equivariant spectra will denote the $\group$-spectra when the group $\group$ is the trivial group.
$\Sigma_{\T}$ will denote {\it the corresponding (non-equivariant) sphere spectrum} associated to
$\T$. When we only consider equivariant spectra, $\Sigma_{\T}$ will
also be used to denote $\Sigma_{\T, \group}$.
\vskip .3cm
It is important to observe that the equivariant  spectrum $\Sigma_{\T, \group}$
 is a module spectrum over the non-equivariant  spectrum $\Sigma_{\T}$.
\subsubsection{\bf Mapping spectra} 
\label{mapping.spectra}
This can be defined as right adjoint to the smash product and will be an
internal Hom in the category of spectra. This will be denoted simply by $\Hom $.  We leave a more
explicit formula for this to the reader. 

\subsection{\bf Eilenberg-Maclane and  Abelian group spectra}
It is shown in \cite[Example 3.4]{Dund2}, how to interpret the motivic complexes, 
${\mathbb Z} =\oplus _{r \ge 0} {\mathbb Z}(r)$ and ${\mathbb Z}/\ell= 
\oplus _{r \ge 0} {\mathbb Z}/\ell(r)$ as ring-spectra in the above stable 
motivic homotopy category. While this is discussed there only in the non-equivariant
framework, the constructions there extend verbatim to the $\group$-equivariant setting.
\vskip .3cm
This spectrum will be denoted $H^{\group}({\mathbb Z})$. The generalized
cohomology with respect to it in the sense defined above, will define $\group$-equivariant motivic
cohomology. $H^{\group}({\mathbb Z}/\ell)$ will denote the mod$-l$
variant for each prime $\ell$. 
\vskip .3cm 
With $\C$ denoting 
$\PSh_*^{{\mathcal O}_{\group}^o}$ ($\PSh/S^{{\mathcal O}_{\group}^o}$), we let 
$\C_{Ab}$ denote the Abelian group objects
in $\C$ {\it with transfers}. (Observe that these are nothing but  diagrams of simplicial Abelian presheaves with transfers
indexed by the orbit category ${\mathcal O}^o_{\group}$. )  One observes as in \cite[section 2]{RO}
that the category $\C_{Ab}$ has tensor structure defined by $\otimes^{tr}$ (which is different from the usual
tensor product).
\vskip .3cm
Let $\C_0'$ denote the subcategory $(Thm.sps)^{{\mathcal O}_{\group}^o}$ considered in Examples ~\ref{main.eg}(i).
Let $\C_{0, tr}'$ denote the $\C_{Ab}$-enriched category with the same objects as $\C_0'$, but where the $\C_{Ab}$-enriched
hom is given by 
\[Z_{tr}(\Hom_{\C}( \quad, \quad )).\] 
 Then we define an Abelian group spectrum
$A$ to be a $\C_{Ab}$-enriched functor $A: \C_{0, tr}' \ra \C_{Ab}$. Observe that this means for each $T_V, T_U \eps \C_{0, tr}'$, one is provided with
 maps 
\[ Z_{tr}(\Hom_{\C}(T_U,  T_V)) \otimes^{tr} A(T_U) \ra A( T_V)\]
 and these are compatible as $U$ and $V$ vary 
in $\C_{0, tr}'$.  (Here the functor $Z_{tr}$ is the same functor denoted ${\mathbb Z}^{tr}$ in \cite[section 2]{RO}.)
\vskip .3cm
If $\C = \PSh_{*,mot}^{\group}$, $\PSh/S_{mot}^{\group}$ 
($\C= \PSh_{*,des}^{\group}$, $\PSh/S_{des}^{\group}$) the corresponding category of abelian group spectra 
will be denoted $\Spt_{ mot}^{Ab, \group}$, $\Spt/S_{mot}^{Ab, \group}$ ($\Spt_{ et}^{Ab, \group}$, $\Spt/S_{et}^{Ab, \group}$, \res). 
The objects
in this category  will be called {\bf motivic 
Abelian group spectra} ({\bf \'etale Abelian group spectra}, \res). 
One may now define the ${\mathbb P}^1$-suspension spectrum associated to  the motive ${\mathcal M}_{\group}(X)$ as an
object in the above  category of Abelian group spectra.  This will be denoted $\Sigma _{{\mathbb P}^1}
({\mathcal M}_{\group}(X))$. 
\vskip .3cm
Next we proceed to consider spectra in the category of $Z/\ell$-vector spaces, with no requirement that they have transfers
as in the discussion above. Such spectra are often needed, especially in connection with completions as in the next section.
\subsubsection{Spectra of $Z/\ell$-vector spaces}
\be \begin{enumerate}[\rm(i) ]
\label{Zlspectra}
\item{For a pointed simplicial presheaf $P \eps \PSh_*$ or $P \eps \PSh/S$, $Z/\ell(P)$ will
denote the presheaf of simplicial $Z/\ell$-vector spaces defined in \cite[Chapter,
2.1]{B-K}.
Observe that the presheaves of homotopy  groups of $Z/\ell(P)$ identify with the
reduced homology presheaves of $P$ with respect to the ring $Z/\ell$. Hence these
are all $\ell$-primary torsion. The functoriality of this construction shows that if $P$ has
an action by the  group $\group$, then $Z/\ell(P)$ inherits this action. Moreover, if the action by
$\group$ on $P$ is continuous, then so is the induced action on $Z/\ell(P)$. Moreover one applies the functor
$Z/\ell$ to the category $\C$ to obtain a $Z/\ell$-linear category: $Z/\ell(\C)$. i.e. $\Hom_{Z/\ell(\C)}(Z/\ell(A), Z/\ell(B)) =
Z/\ell(\Hom_{\C}(A, B)$.}
\item{Next one extends the construction in the previous step to spectra.
Let $\C=\PSh_*^{\O^{o}_{\group}}$ or $\PSh/S^{\O^{o}_{\group}}$ and let
$E \eps Spectra(\C)$. Then first one
applies
the functor $(\quad ) \mapsto Z/\ell(\quad)$ to each $E(T_W)$, $T_W \eps \C_0'$. Then there exist
natural maps $ Z/\ell(\Hom_{\C}(T_V,  T_W)) {\underset {Z/\ell} \otimes} Z/\ell(E(T_V))
\ra Z/\ell(\Hom_{\C}(T_V,  T_W) \wedge E(T_W))$. Therefore, one may compose the above maps with the obvious
map $Z/\ell(\Hom_{\C}(T_V, T_W)\wedge E(T_W)) \ra Z/\ell(E( T_W))$ to define an object in $Spectra_{Z/\ell}(\C)$.}
\item{A pairing $M \wedge N \ra P$ in $\PSh_*^{{\mathcal
O}_{\group}^o}$ (or in $\PSh/S^{\O^{o}_{\group}}$) induces a pairing
 \[Z/\ell(M) \wedge Z/\ell(N) {\overset {\cong} \ra}  Z/\ell(M){\underset {Z/\ell} \otimes} Z/\ell(N) \ra Z/\ell(M \wedge N) \ra Z/\ell(P).\]
 Similarly 
a pairing $M \wedge N \ra P$ in $Spectra (\C)$ induces a similar 
pairing 
\[Z/\ell(M) \wedge Z/\ell(N) {\overset {\cong} \ra} Z/\ell(M) {\underset {Z/\ell} \otimes }Z/\ell(N) \ra Z/\ell(P). \]
(To see this, one needs to recall the construction of the smash-product of spectra from 
\cite[2.3]{Dund1}: first one takes the point-wise smash-product $M \wedge N: \C_0' \times \C_0' \ra \C$
and then one applies a left-Kan extension of this along the $\wedge: \C_0' \times \C_0' \ra \C_0'$.
The functor $Z/\ell$ commutes with this left-Kan extension.) 
This shows that if 
$R \eps Spectra (\C)$ is a ring spectrum so is $Z/\ell(R)$ and that if
 $M \eps Spectra(\C)$ 
is a module over the ring spectrum $R$, $Z/\ell(M)$ is a module object over the ring spectrum $Z/\ell(R)$.}
\item{If $\{f:A \ra B\} =\oJ '_{Sp}$ is a set of generating cofibrations (trivial cofibrations) for $Spectra (\C)$, then,
 we will let $\{Z/\ell(f): Z/\ell(A) \ra Z/\ell(B)|f \eps \oJ '_{Sp}\}$ be the set of generating  cofibrations 
(trivial cofibrations, \res) for
$Spectra_{Z/\ell}(\C) = Spectra(Z/\ell(\C))$. This defines a model structure on $Spectra_{Z/\ell}(\C)$: see 
\cite[Theorem 11.3.2]{Hirsch}. 
\vskip .3cm
Here is an outline of an argument to verify the hypotheses of the above theorem. Let $U$ be the 
forgetful functor right-adjoint to
the functor $Z/\ell(\quad )$. First it seems important that we use the object-wise model structures on 
$\C=\PSh_*^{\O^{o}_{\group}}, \PSh/S^{\O^{o}_{\group}}$ and on the corresponding unstable model categories of spectra so that all
objects are cofibrant and a map of spectra $\chi' \ra \chi$ is a cofibration if each map $\chi'(T_V) \ra \chi(T_V)$ is a 
monomorphism. 
\vskip .3cm
One may  verify that the functors $Z/\ell(\quad )$ and $U$ 
satisfy the hypotheses there, first unstably. For this, one may first observe that $U$ commutes
with transfinite colimits where the structure maps of the direct system are monomorphisms. 
The first hypothesis on the domains of the generating cofibrations
 (generating trivial cofibrations) being small with respect to the corresponding cells follows from
the adjunction between the functors $Z/\ell(\quad )$ and $U$ and the above observation. $U$ clearly preserves
monomorphisms which identify with the cofibrations since we are working with the object-wise model structure
unstably. One defines a map of spectra $\phi:\chi' \ra \chi$ in $Spectra_{Z/\ell}(\C)$ to be a fibration (weak-equivalence) if and only if
$U(\phi)$ is a fibration (weak-equivalence, \res) in $Spectra(\C)$. In particular, an object $\chi$ in $Spectra_{Z/\ell}(\C)$ is fibrant 
if and only if
$U(\chi) \eps Spectra(\C)$ is fibrant. 
\vskip .3cm
Next suppose $f: X' \ra X$ is a map in $Spectra(\C)$ that is a weak-equivalence. Then, for any fibrant object 
$\chi \eps Spectra_{Z/\ell}(\C)$ one obtains the commutative diagram with the vertical maps isomorphisms and the bottom map a weak-equivalence:
\vskip .3cm
\xymatrix{{Map(Z/\ell(X), \chi)} \ar@<1ex>[r] \ar@<1ex>[d]^{\cong} & {Map(Z/\ell(X'), \chi)} \ar@<1ex>[d]^{\cong} \\
{Map(X, U(\chi))} \ar@<1ex>[r]^{\simeq} &{Map(X', U(\chi'))} }
\vskip .3cm \noindent
Therefore, it follows that the functor $Z/\ell$ sends weak-equivalences to weak-equivalences. (A similar proof also works
to prove the corresponding statement for pointed simplicial  presheaves.)
It follows that $U$ sends any $Z/\ell(\oJ'_{Sp})$-cells to weak-equivalences.
(One may also prove  that $Z/\ell(\quad )$ preserves usual weak-equivalences of simplicial presheaves  from the Hurewicz theorem.)
 }
\end{enumerate} 
\vskip .3cm
 If 
\[\C= \PSh_{*,mot}^{\group}, \PSh/S_{mot}^{\group} (\C= \PSh_{*,des}^{\group}, \PSh/S_{des}^{\group})\]
 the corresponding category of spectra 
will be denoted $\Spt_{mot}^{Z/\ell, \group}, \Spt/S_{mot}^{Z/\ell, \group}$ 
($\Spt_{et}^{Z/\ell,\group}, \Spt/S_{et}^{Z/\ell, \group}$, \res). These categories are also locally presentable
 and hence the model
categories are combinatorial. The unstable and stable model structures on $\Spt_{mot}^{Z/\ell, \group}$
 are obtained by the process of ${\mathbb A}^1$-localization of $\Spt^{Z/\ell, \group}$ (as discussed earlier)
 and a similar description 
holds for the model categories $\Spt/S_{mot}^{Z/\ell, \group}$ 
($\Spt_{et}^{Z/\ell,\group}, \Spt/S_{et}^{Z/\ell, \group}$, \res).
One observes that the functor $Z/\ell(\quad )$
 induces a functor 
\[Z/\ell(\quad ): \Spt_{mot}^{\group} \ra \Spt_{mot}^{Z/\ell, \group},  \Spt/S_{mot}^{\group} \ra \Spt/S_{mot}^{Z/\ell, \group} \mbox{ and }\]
\[Z/\ell(\quad ): \Spt_{ et}^{\group} \ra \Spt_{ et}^{Z/\ell,\group}, Z/\ell(\quad ): \Spt/S_{ et}^{\group} \ra \Spt/S_{et}^{Z/\ell,\group}. \]
\vskip .3cm \noindent
The observations above already show that these are left-Quillen functors. 
\subsection{\bf ${\mathbb A}^1$-localization in the \'etale setting}
We will assume throughout this section that the base scheme $\B$ is the spectrum of a separably closed field $k$.
Recall that $\Spt_{et}^{\group}$ ($\Spt/S_{et}^{\group}$) denotes the category of spectra on the big \'etale site.
We will let $\Spt_{et} (\Spt/S_{et})$ denote the corresponding spectra with trivial action by the  group $\group$.
\vskip .3cm 
Observe that the \'etale homotopy type of affine-spaces over a separably closed
field $k$, are trivial when completed away from $char(k)$.
\begin{proposition} Assume the base scheme  $\B$ is a separably closed field $k$  and
$char(k) =p$. Let $E \eps \Spt_{et} (\Spt/S_{et})$ be a constant sheaf of spectra so
that all the (sheaves of) homotopy groups
$\pi_n(E) $ are $\ell$-primary torsion, for some $l \ne p$. Then $E$ is ${\mathbb
A}^1$-local in $\Spt_{et} \quad (\Spt/S_{et})$, i.e. the projection $P \wedge {\mathbb
A}^1_{ +} \ra P$ induces a weak-equivalence: $Map (P, E) \simeq Map (P
\wedge {\mathbb A}^1_{ +}, E)$, $P \eps \Spt_{et} \quad (\Spt/S_{et}, \res)$. 
\end{proposition}
\begin{proof} First let $P$ denote the suspension spectrum associated to some
scheme $X$. Then $Map(P, E)$ ($Map (P \wedge {\mathbb A}^1_{ +},
E)$) identifies with the spectrum defining the generalized \'etale cohomology
of 
$X$ (of $X \times {\mathbb A}^1$, \res) with respect to the spectrum $E$.
(Observe that in the relative case, i.e. if $E \eps \Spt/S_{et}$, the affine line
${\mathbb A}^1$ denotes the affine line over $S$ and $\wedge$ then denotes $\wedge_S$.)
There exist Atiyah-Hirzebruch spectral sequences that converge to these
generalized \'etale cohomology groups with the $E_2^{s, t}$-terms being $H^s_{et}(X_+,
\pi_{-t}(E))$
and $H^s_{et}((X \times {\mathbb A}^1)_+, \pi_{-t}(E))$, \res. Since the
sheaves of homotopy groups $\pi_{-t}(E)$ are all $\ell$-torsion with $l \ne p$, and
$\B$ is assumed to be a separably closed field, $X$ and $X \times {\mathbb A}^1$
have finite $\ell$-cohomological dimension. Therefore these spectral sequences
converge strongly and the conclusion of the proposition holds in this
case. For a general simplicial presheaf $P$, one may find a simplicial
resolution 
where each term is a disjoint union of schemes as above (indexed by a small
set). Therefore, the conclusion of the proposition holds also for suspension
spectra
 of all simplicial schemes and therefore for all spectra $P$. \end{proof}
\begin{corollary} 
 Assume the base field $\B$ is a separably closed field $k$ and
$char(k) =p$. 
\vskip .3cm
Let $E \eps \Spt_{et}^{\group} (\Spt/S_{et}^{\group})$ be a constant sheaf of spectra so
that all the (sheaves of) homotopy groups
$\pi_n(E) $ are $\ell$-primary torsion, for some $l \ne p$. Then $E$ is ${\mathbb
A}^1$-local in $\Spt_{et}^{\group} (\Spt/S_{et}^{\group})$, i.e. the projection $P \wedge {\mathbb
A}^1_{+} \ra P$ induces a weak-equivalence: $Map _{\group}(P, E) \simeq Map_{\group} (P
\wedge {\mathbb A}^1_{+}, E)$ where $Map_{\group}$ is part of the simplicial structure on
$\Spt_{et}^{\group}$ \quad ($\Spt/S_{et}^{\group}$, \res).
\end{corollary}
\begin{proof} We will only consider the case where $E \eps \Spt_{ et}^{\group}$. 
The last proposition shows that the conclusion is true were it not for the
 group action. The functor 
$\bar \Phi: P \mapsto \{P(\group/H) \mid H  \eps \W \}$
sending 
\[\Spt_{et}^{\group} \mbox{ to } {\underset {H \eps \W} \Pi} \Spt_{et}\]
\vskip .3cm \noindent
 has a left-adjoint defined
by sending $\{Q(\group/H) \mid  H \eps \W \}$ to 
$\bigvee Q(\group/H) \wedge (\group/H)_+$,  where $(\group/H)_+$ is the free diagram considered in
~\ref{diagram.cats}.This provides a triple, whereby one may find 
a simplicial resolution of a given $P \eps \Spt_{et}^{\group}$ by spectra of the form
 $Q \wedge (\group/H)_+$ where $Q \eps \Spt_{et}$.  To see this is a simplicial resolution, it suffices
to show this after applying $\bar \Phi$ in view of the fact that a map $f: A \ra B$ in 
$\Spt_{et}^{\group}$ is a weak-equivalence if and only if the induced maps $f(\group/H)$ are all weak-equivalences.
On applying $\bar \Phi$, the above augmented simplicial object will have an extra degeneracy.
\vskip .3cm
Therefore, one reduces to considering $P$ 
which are of the form $Q \wedge (\group/H)_+$. Then one obtains by adjunction, the following
weak-equivalences: $Map_{\group}(Q \wedge (\group/H)_+, E) \simeq Map(Q, E(\group/H))$ and
$ Map_{\group}(Q \wedge {\mathbb A}^1_+ \wedge (\group/H)_+, E) \simeq Map (Q \wedge {\mathbb A}^1_+, E(\group/H))$
where $Map$ is part of the simplicial structure on $\Spt_{et}$. One obtains the 
weak-equivalence $Map(Q, E(\group/H)) \simeq Map(Q \wedge {\mathbb A}^1_+, E(\group/H))$ by the last proposition.
This completes the proof of the corollary.
\end{proof}
\section{\bf Effect of ${\mathbb A}^1$-localization  on mod-$\ell$ completions of spectra}
\label{ring.module.completions}
Completions, especially at a prime $\ell$, play a key role in our work:  this is to be expected since even the
\'etale homotopy type of schemes has good properties only after completion away
from the residue characteristics. Therefore,
it is important to show that such completions may be carried out so as to be compatible with the process of
${\mathbb A}^1$-localization. We
will show how to adapt  the Bousfield-Kan completion
(see \cite{B-K}) to our framework as follows.
 \vskip .3cm
First recall that the Bousfield-Kan completion is defined with respect to 
a commutative ring with $1$: in our framework, this ring will be always $Z/\ell$
for a fixed prime $\ell$ different from the residue characteristics of the base-scheme $\B$. Therefore, we will work with
category $\Spt_{ mot}^{Z/\ell, \group}$. All our arguments and constructions work also on 
$\Spt/S_{mot}^{Z/\ell, \group}, \Spt_{et}^{Z/\ell, \group}$ and on $\Spt/S_{et}^{Z/\ell, \group}$ but we do not discuss these explicitly.
\vskip .3cm 
As observed above (see the discussion in ~\ref{Zlspectra})(iv), the functor $Z/\ell(\quad)$ does preserve ${\mathbb A}^1$-equivalences.
Recall a presheaf of spectra $P$ is motivically fibrant if it satisfies the conditions listed
in ~\ref{spectra.construct.2}: of these, most of the conditions there can easily be checked to be preserved
by the functor $Z/\ell(\quad)$, with one of the main issues being the condition that the presheaf
of spectra be one of $\Omega$-spectra or that it be fibrant in the unstable projective or 
injective model structure. One may circumvent this issue by redefining $Z/\ell(E)$
for a fibrant spectrum to be obtained by first taking $Z/\ell(E(T_V))$ for all the constituent spaces
and then replacing this by a fibrant spectrum. For the non-equivariant case, i.e. for symmetric spectra,
one may in fact use the projective unstable model structure for spectra. Then, as 
observed already  the usual stabilization process applies to obtain a fibrant replacement: 
see ~\ref{alt.stab.symm.sp}. Since the functor $\Omega_T$ and the filtered colimit involved there
preserve ${\mathbb A}^1$-equivalences, this fibrant replacement functor preserves ${\mathbb A}^1$-equivalences.
Clearly it also sends $Z/\ell$-module spectra to
$Z/\ell$-module spectra. Therefore, in this case, the usual Bousfield-Kan $Z/\ell$-completion extends to the motivic
setting. The following discussion is simply an extension of the above discussion to handle the equivariant spectra
considered in the earlier sections of this paper.

\vskip .3cm
It may also be necessary to point out that the explicit ${\mathbb A}^1$-localization functor in \cite[Theorem 4.2.1]{Mor}
works only in model categories that are weakly finitely generated, i.e. where the domains and co-domains of the generating cofibrations
and generating trivial cofibrations are small. Since we use injective model structures, it follows that the last hypothesis fails
in our setting and therefore, one cannot make use of the above ${\mathbb A}^1$-localization functor of Morel.
\vskip .3cm
The above paragraphs should provide ample justification for the modified Bousfield-Kan completion considered below.
Our constructions below, have the advantage of showing explicitly that the
completion functor may be made to be compatible with the ${\mathbb A}^1$-localization
functor. 
  We proceed to do this presently, but we will
digress to develop a bit of the technicalities. 
\vskip .3cm
\subsubsection{}
\label{generators}
We will let $\Spt^{Z/\ell, \group}$ denote either 
$\Spt_{ mot}^{Z/\ell, \group}$ or $\Spt/S_{ mot}^{Z/\ell, \group}$. Let $\Spt^{\group}$ denote
$\Spt_{ mot}^{\group}$ or $\Spt/S_{ mot}^{ \group}$ where $J_{Sp}$ ($I_{Sp}$) will denote the generating trivial cofibrations
generating cofibrations, \res). We will next provide $\Spt^{Z/\ell, \group}$ with the cofibrantly generated model category structure
where the generating trivial cofibrations (generating cofibrations) are $Z/\ell (J_{Sp})$ ($Z/\ell(I_{Sp})$, \res).  Observe  that
the free $Z/\ell$-module functor is left-adjoint to the underlying functor $U: \Spt^{Z/\ell, \group} \ra \Spt^{\group}$ and $U$-applied to any map
in $Z/\ell(J_{Sp})$-cell ($\Z/\ell(I_{Sp}$-cell) is a $J_{Sp}$-cofibration ($I_{Sp}$-cofibration, \res). Therefore,  one may readily see that 
the choice of $Z/\ell(J_{Sp})$ and $Z/\ell(I_{Sp})$ provides $\Spt^{Z/\ell, \group}$ with the structure of a cofibrantly generated model
category where the generating trivial cofibrations (generating cofibrations) are $Z/\ell(J_{Sp})$ ($Z/\ell(I_{Sp})$, \res). 
\vskip .4cm
\subsubsection{\bf The basic construction}
\begin{enumerate} [\rm (i)]
 \item {For each $Y \eps \Spt^{Z/\ell, \group}$, we let $S(Y) $ denote the set of
all commutative squares 
\[\xymatrix{A \ar@<1ex>[r] \ar@<-1ex>[d]^f & Y \ar@<-1ex>[d] \\
B \ar@<1ex>[r] & 0}\]
\vskip .3cm \noindent
where $f \eps Z/\ell(J_{Sp})$.}
\item{ Then we let $P(Y)= (\oplus B) {\underset {(\oplus A)} {\oplus} } Y$ where the sum is over all elements of $S(Y)$.}
\end{enumerate}
\vskip .3cm
Now we start with an $E \eps \Spt^{Z/\ell, \group}$  and will construct a factorization of the map
 $Z/\ell(E) \ra 0$ into $Z/\ell(E) \ra F(E) \ra 0$ , where the first map is a cofibration 
(that is an ${\mathbb A}^1$-equivalence) followed by an 
${\mathbb A}^1$-local fibration  {\it all in the category $\Spt^{Z/\ell, \group}$}.
 Let $F^0(E) = Z/\ell(E)$. 
The factorization is obtained as in the diagram
\be \begin{equation}
     \label{factorization}
\xymatrix{{F^0(E)} \ar@<-1ex>[drr]^{p_0} \ar@<1ex>[r] &{F^1(E)} \ar@<-1ex>[dr]^{p_1} \ar@<1ex>[r] & {F^2(E)} 
\ar@<1ex>[r] \ar@<-1ex>[d]^{p_2} & {\cdots}  \ar@<1ex>[r] &{F^{\beta} (E)} 
\ar@<-1ex>[dll]^{p_{\beta}} 
\ar@<1ex>[r] & {\cdots}
\\
& & {0=Zl/(*)}}
 \end{equation} \ee
\vskip .3cm \noindent
 Assume we have constructed the above diagram up to $F^{\beta}(E)$ all in 
$\Spt^{Z/\ell, \group}$ beginning with $F^0(E)$. To continue the construction, 
  we consider maps from families of maps $ A \ra B$  in $Z/\ell(J_{Sp})$  to 
$(F^{\beta}(E))  \ra 0 =Z/\ell(*)$. One then lets $F^{\beta +1}(E) =$ be defined by the pushout
 $(\oplus B){\underset {\oplus A} \oplus} F^{\beta}(E)$ in the category 
 $\Spt^{Z/\ell, \group}$. i.e. $F^{\beta +1}(E) = P(F^{\beta}(E))$.
For a limit ordinal $\gamma $, $F^{\gamma}(E)$ is defined as ${\underset {\beta < \gamma } \colim} 
F^{\beta}(E)$
 with the colimit taken in the same category $\Spt^{Z/\ell, \group}$.
We obtain the required factorization by taking $F(E) = {\underset {\beta < \lambda} \colim} 
F^{\beta}(E)$ and
with the obvious induced maps $Z/\ell(P) \ra F(E)$ and $E \ra 0$. Here $\lambda$ is a sufficiently large
enough ordinal: if the domains of $J'_{Sp}$ are $\kappa$-small for some ordinal $\kappa$, then
$\lambda $ is required to be $\kappa$-filtered. This means $\lambda$ is a limit ordinal and that 
if $A \subseteq \lambda$ and $|A| \le \kappa$, then $sup A < \lambda$.
\vskip .3cm
 It is clear from the construction that
 $F(E) \eps \Spt^{Z/\ell, \group}$. That it is ${\mathbb A}^1$-local (i.e. is a fibrant object of $\Spt^{Z/\ell, \group}$)
 follows from the observation that 
the domain of any object of $Z/\ell(J_{Sp})$ is small with respect to $Z/\ell(J_{Sp})$-cell.. 
{\it Henceforth we will denote the above object $F(E)$ as $\widetilde {Z/\ell(E)}$.} 
\vskip .3cm
\begin{proposition} 
\label{Reedy}
Suppose ${\widetilde E} \eps \Spt_{ mot}^{Z/\ell, \group} $ is fibrant. Then any map $G \ra  U({\widetilde {E}})$
defines an induced map ${\widetilde {Z/\ell}}(G) \ra {\widetilde E}$. i.e. The underlying functor $U$
sending a fibrant spectrum in $\Spt_{ mot}^{Z/\ell, \group}$ to $\Spt_{ mot}^{\group}$ has ${\widetilde {Z/\ell}}$
as its left-adjoint.
\end{proposition}
\begin{proof} We will start with a map $G \ra U({\widetilde {E}})$ in $\Spt_{mot}^{\group} $,
where ${\widetilde E} \eps \Spt_{ mot}^{Z/\ell, \group} $.
 This corresponds to map $Z/\ell(G) \ra {\widetilde {E}}$ of $Z/\ell$-module spectra. We proceed to show that this map induces
a map ${\widetilde {Z/\ell}(G)} \ra {\widetilde {E}}$ so that pre-composing with the map 
$Z/\ell(G) \ra {\widetilde {Z/\ell}(G)}$ is the map $Z/\ell(G) \ra {\widetilde {E}}$.
We will show inductively the above map $Z/\ell(G) \ra {\widetilde {E}}$  extends to a map
$F^{\beta}(G)\ra {\widetilde E}$ for all $\beta < \lambda$. Assume that $\beta$ is such an ordinal for which 
we have extended the given map $Z/\ell(G) \ra {\widetilde {E}}$ to a map $F^{\beta}(G) 
\ra {\widetilde {E}}$.
 Let $\{A_{\alpha} \ra B_{\alpha}|\alpha \}$ denote a family of trivial cofibrations belonging to
$Z/\ell(J_{ Sp})$ indexed by a small set so that one is provided with a map 
$\oplus _{\alpha} A_{\alpha} \ra F^{\beta}(G)$. Since ${\widetilde {E}}$ is a fibrant object of
$\Spt_{ mot}^{Z/\ell, \group} $, and each map $A_{\alpha} \ra B_{\alpha}$ is a generating trivial
cofibration in the same category, one obtains an extension $\oplus_{\alpha} B_{\alpha} \ra {\widetilde {E}}$.
\vskip .3cm
Next recall that $F^{\beta+1}(G) = 
(\oplus_{\alpha}B_{\alpha}) {\underset {(\oplus _{\alpha} A_{\alpha})} \oplus} F^{\beta}(G)$. 
Therefore, this
has an induced map of $Z/\ell$-module spectra to ${\widetilde {E}}$. Since $F^{\gamma}(G)$ for a limit ordinal $\gamma <\lambda$
is defined as $\colim_{\beta < \gamma} F^{\beta}(G)$, it follows that we obtain the 
required extension of the given map to a map ${\widetilde {Z/\ell}(G)} \ra {\widetilde {E}}$.
\end{proof}
\begin{proposition}
 Assume the above situation. Then  $\widetilde {Z/\ell(E)}$ is $Z/\ell$-complete in the following sense.
For every map $\phi: A \ra B$ in $\Spt^{\group}$ with both $A$ and $B$ cofibrant and
which induces an isomorphism
 $H_*(A, { Z}/\ell) {\overset {\cong} {\ra}} H_*(B, { Z}/\ell)$ of the homology
presheaves with $Z/\ell$-coefficients, then the induced map
 $\phi^*:Map(B, U(\widetilde {Z/\ell(E)})) \ra Map(A, U(\widetilde {Z/\ell}(E)))$ is a weak-equivalence of simplicial sets.
\end{proposition}
\begin{proof}
This follows readily from the
following observations: 
\vskip .3cm \noindent
(i)  such a homology isomorphism induces a weak-equivalence 
$Z/\ell(A) \ra Z/\ell(B) $ and 
\vskip .3cm \noindent
(ii) $Map(B, U(\widetilde {Z/\ell(E)})) \simeq Map(Z/\ell(B), \widetilde {Z/\ell(E)})$  and 
$Map(A, U(\widetilde {Z/\ell(E)})) \simeq Map(Z/\ell(A), \widetilde {Z/\ell(E)})$ 
\vskip .3cm \noindent
where the $Map$ on the 
left-side (right-side) is 
taken in the category $\Spt_{mot}^{\group} $ 
\newline \noindent
($\Spt_{mot}^{Z/\ell, \group}$ ).
 One may observe that the functor $Z/\ell(\quad )$ preserves cofibrant objects
and since ${\widetilde {Z/\ell(E)}}$ is a fibrant object in $\Spt_{S, mot}^{Z/\ell, \group}$, the weak-equivalence
$Z/\ell(A) \ra Z/\ell(B)$ induces the weak-equivalence $Map(Z/\ell(B), \widetilde {Z/\ell(E)}) 
\simeq Map(Z/\ell(A), \widetilde {Z/\ell(E)})$.
\end{proof}
\subsubsection{}
\label{obs.holim}
Proposition ~\ref{Reedy} above shows that $\{{\widetilde {Z/\ell}_n(E)}|n \ge 1\}$ (see Definition ~\ref{completions.def} below)
 provides a cosimplicial
 object in $\Spt_{mot}^{\group} $ (with ${\widetilde {Z/\ell}(E)}$ in degree $0$) which is 
group-like, i.e. each ${\widetilde {Z/\ell}_n(E)}$
belongs to $\Spt_{mot}^{Z/\ell, \group} $ and that the structure maps of the cosimplicial object except
for $d^0$ are $Z/\ell$-module maps. Such a cosimplicial object is  fibrant in the
Reedy model structure on the category of cosimplicial objects of $\Spt_{S, mot}^{Z/\ell, \group} $:
see for example, \cite[Chapter 10, 4.9 proposition]{B-K}.
 Therefore, one may take the homotopy inverse limit of this cosimplicial object
just as in \cite[Chapter X]{B-K}. 
\vskip .3cm
In view of the observation, it follows that if ${\widetilde {Z/\ell}_{i \le n}(E)}$ denotes
the corresponding truncated cosimplicial object, truncated to degrees $\le n$, one obtains
a compatible collection of maps $\{\holim {\widetilde {Z/\ell}_{\bullet}(E)} \ra 
\holim {\widetilde {Z/\ell}_{i \le n}(E)}|n \ge 1\}$. Moreover the fiber of the map
 $\holim {\widetilde {Z/\ell}_{i \le n}(E)} \ra \holim {\widetilde {Z/\ell}_{i \le n-1}(E)}$
identifies with $ker(s^0) \cap \cdots \cap ker(s^{n-1})$, where $s^i: {\widetilde {Z/\ell}_n(E)} \ra {\widetilde {Z/\ell}_{n-1}(E)}$
is the $i$-th co-degeneracy.
\vskip .3cm
\begin{definition} ($Z/\ell$-completions.)
\label{completions.def} First we will  apply the construction above to define ${\widetilde {Z/\ell(P)}}$.
 One  repeatedly applies the functor $ \widetilde {Z/\ell}$ to an object $E
\eps \Spt_{mot}^{\Z/\ell, \group}$, to define the tower $\{{\widetilde {(Z/\ell)_n(E)}}|n\}$ in $Spectra(\C)$. We let 
$ {\widetilde {Z/\ell}}_{\infty}(E) = holim_n \{{\widetilde {(Z/\ell)_n(E)}}|n\}$.
In contrast, we let $ { {Z/\ell}}_{\infty}(E) = holim_n \{{ {(Z/\ell)_n(E)}}|n\}$.
\end{definition}
\vskip .3cm
Completions will be most often applied to pro-objects in $Spectra(\C)$.
Let $X = \{X_i|i \eps I\} \eps pro-Spectra(\C)$ denote a pro-object
indexed by a small category $I$. Let $E \eps Spectra(\C)$. Then we let
\be \begin{align}
\label{complete.pro.1}
\Hom_{Sp}(X, {\widetilde {Z/\ell}}_{\infty}(E)) &= holim_n colim_I \Hom_{Sp}(X_i, {\widetilde {(Z/\ell)_n(E)}})\\
Map(X, {\widetilde {Z/\ell}}_{\infty}(E)) &= holim_n colim_I Map (X_i, {\widetilde {(Z/\ell)_n(E)}}) 
\end{align} \ee
\vskip .3cm \noindent
Here $\Hom$ denotes the internal hom in the category $Spectra(\C)$. 
\subsubsection{\bf Properties of the completion functor}
\label{completion.mult}
\vskip .3cm
Here we list a sequence of key properties of the  completion functor so as to serve as a reference.
\begin{enumerate}[\rm(i)]
\item{The $Z/\ell$-completion, $\widetilde {Z/\ell _{\infty}(E)}$ is $Z/\ell$-complete.
i.e. (i) $ U (\widetilde {Z/\ell_{\infty}(E)})$ it is fibrant in
 $\Spt^{\group}$ and for every map $\phi: A \ra B$ in $\Spt^{\group}$ between cofibrant 
objects which induces an isomorphism
 $H_*(A, { Z}/\ell) {\overset {\cong} {\ra}} H_*(B, { Z}/\ell)$ of the homology
presheaves with $Z/\ell$-coefficients,  the induced map
 $\phi^*:Map(B, \widetilde {Z/\ell _{\infty}(E)}) \ra Map(A, \widetilde
{Z/\ell _{\infty}(E)})$ is a weak-equivalence of simplicial sets. In particular, this applies to the case where $B= Z/\ell _{\infty}(A)$ and
$f:A \ra B$ is the obvious Bousfield-Kan completion map.}

\item{For each $E \eps \Spt_{}^{\group}$, each 
 $\widetilde {Z/\ell_n(E)}$ is ${\mathbb A}^1$-local and belongs to $\Spt^{Z/\ell, \group}$.}
\item{For $ E \eps \Spt_{et}^{\group}$ and let $Z/\ell(E)$ denote the free $Z/\ell$-vector space functor
applied to $E$. Then $Z/\ell_n(E) = Z/\ell(U(Z/\ell_{n-1}(E)))$ defines the $n$-th term of the tower defining the usual 
Bousfield-Kan $Z/\ell$-completion. One obtains a natural map $\{Z/\ell_n(E) \ra {\widetilde {Z/\ell_n(E)}}|n \}$
of towers and therefore an induced map on taking the homotopy inverse limits.}
\end{enumerate}

\label{completions}



\end{document}